\documentclass[aip,cha,reprint,nofootinbib]{revtex4-2}

\usepackage{graphicx}
\usepackage{pifont}
\usepackage{leftindex}
\usepackage{amsmath,amssymb,amsfonts}
\usepackage{amsthm}
\usepackage{enumerate}
\usepackage{bm}
\usepackage{dsfont}
\usepackage[mathscr]{eucal}
\usepackage{upgreek}
\usepackage{xcolor}
\usepackage{hyperref}

\makeatletter
\newcommand{\pushright}[1]{\ifmeasuring@#1\else\omit\hfill$\displaystyle#1$\fi\ignorespaces}
\newcommand{\pushleft}[1]{\ifmeasuring@#1\else\omit$\displaystyle#1$\hfill\fi\ignorespaces}
\makeatother
\newcommand{\tr}{\mathrm{tr}}

\DeclareMathOperator*\uplim{\overline{lim}}

\newcommand{\diag}{\mathbf{diag}}

\newcommand{\dist}[2]{\mathbf{dist} \left[ #1; #2\right]}
\newcommand{\brak}[1]{^{\left\langle #1 \right\rangle} }
\newcommand{\inter}{\mathbf{int}\,}
\newcommand{\mA}{\mathscr{A}}
\newcommand{\Ball}[2]{B_{#2}^{\,#1}} % ball B_x^r

\newcommand{\ecc}{\operatorname{ecc}}
\newcommand{\varpid}{\varpi_d}
\newcommand{\omegad}{\omega_d}
\newcommand{\Ln}{\mathbf{Ln}}

\DeclareMathAlphabet{\mathpzc}{OT1}{pzc}{m}{it}

\newcommand{\trn}{^{\scriptscriptstyle \top}} %Transpose symbol
\newcommand{\trnm}{^{\scriptscriptstyle -\top}} %--- minus Transpose symbol

\newcommand{\dotm}[1]{\dot{#1}}
\newfont{\bfb}{msbm10 scaled 1200}
\newfont{\mfb}{msbm8}
\newtheorem{asm}{Assumption}[section]
\newtheorem{proposition}{Proposition}[section]
\newtheorem{dfn}{Definition}[section]

\newtheorem{remark}{Remark}[section]
\newtheorem{lemma}{Lemma}[section]
\newtheorem{thmm}{Theorem}[section]
\newtheorem{corollary}{Corollary}[section]

\DeclareRobustCommand{\longrightarrow}{
  \DOTSB\relbar\joinrel\relbar\joinrel\rightarrow
}
\newcommand{\R}{\mathbb{R}}%

\newcommand{\ff}{g}
\newcommand{\br}{\mathbb{R}}
\newcommand{\ve}{\varepsilon}
\newcommand{\spr}[2]{\left\langle #1; #2 \right\rangle}

\newcommand{\mc}{}
\makeatletter
\DeclareFontFamily{U}  {MnSymbolF}{}
\DeclareSymbolFont{symbolsMN}{U}{MnSymbolF}{m}{n}
\SetSymbolFont{symbolsMN}{bold}{U}{MnSymbolF}{b}{n}
\DeclareFontShape{U}{MnSymbolF}{m}{n}{
    <-6>  MnSymbolF5
   <6-7>  MnSymbolF6
   <7-8>  MnSymbolF7
   <8-9>  MnSymbolF8
   <9-10> MnSymbolF9
  <10-12> MnSymbolF10
  <12->   MnSymbolF12}{}
\DeclareFontShape{U}{MnSymbolF}{b}{n}{
    <-6>  MnSymbolF-Bold5
   <6-7>  MnSymbolF-Bold6
   <7-8>  MnSymbolF-Bold7
   <8-9>  MnSymbolF-Bold8
   <9-10> MnSymbolF-Bold9
  <10-12> MnSymbolF-Bold10
  <12->   MnSymbolF-Bold12}{}
\DeclareMathSymbol{\tbigtimes}{\mathop}{symbolsMN}{2}
\newcommand*{\bigtimes}{%
  \DOTSB
  \tbigtimes
  \slimits@
}
\makeatother
\allowdisplaybreaks%

\begin{document}

\title{Contraction theory: Hausdorff--Riemann Measures as Set-Based Lyapunov Functions}

\author{A. Matveev}
\email{almat1712@yahoo.com}
\affiliation{Department of Mathematics and Mechanics, Saint Petersburg University, St.\ Petersburg, Russia}
\affiliation{Department of Information Technologies and Artificial Intelligence, Sirius University of Science and Technology, Sirius, Russia}

\author{A. Pogromsky}
\email{A.Pogromsky@tue.nl}
\affiliation{Department of Mechanical Engineering, Eindhoven University of Technology, Eindhoven, The Netherlands}

\date{\today}

\begin{abstract}
We {\mc offer} a measure-theoretic extension of {\mc the concept and theory of} $k$-contraction, {\mc including their generalization on fractional dimensions $d$.}
{\mc The respective contraction property is defined through} the exponential
decay of the $d$-dimensional volume of compact sets transported by a nonlinear flow.
For autonomous systems on positively invariant compact sets, we derive {\mc comprehensive, i.e.,} necessary
and sufficient, conditions {\mc for $d$-contractivity} in two complementary forms. The first is expressed in
terms of {\mc the} finite-time Lyapunov characteristic exponents and {\mc is akin in spirit to} the
first Lyapunov method.
The second {\mc one is consonant with the} second Lyapunov method {\mc and relies on} existence
of a Riemannian metric ensuring exponential decay of the metric-induced
$d$-dimensional Hausdorff measure.

To {\mc acquire} monotone {\mc measure-theoretic-based} Lyapunov functions, we introduce a family of
\emph{Hausdorff-Riemann measures}, {\mc which are} elliptic, metric-dependent $d$-measures that
strictly decrease along the trajectories and thus {\mc may} serve as Lyapunov functions. These measures enable an anytime characterization of {\mc the rate of} contraction
and provide constructive tools for stability analysis and feedback design.
{\mc To illustrate the applicability of the approach}, we derive tractable criteria for orbital stability of periodic
solutions {\mc of autonomous ODE's} and {\mc employ} several {\mc prototypical particular} examples, including a rigid body
with dissipation and constant torque, the R\"ossler system, and the Langford system.
\end{abstract}


\maketitle
\section*{Lead paragraph}
{\bf
Contraction theory studies the conditions under which the trajectories of a dynamical system converge toward one another. This property underlies robust behavior of the system and simplifies analysis of its nonlinear models. Traditional contraction criteria focus on shrinkage of lengths, areas, or volumes - typically in integer dimensions and using standard Euclidean notions of the size. This paper extends that viewpoint to non-integer dimensions $d$ via introducing Hausdorff -Riemann $d$-measures, which combine Hausdorff-type scaling with using state-dependent metrics. We show that a nonlinear system is uniformly $d$-contractive
(in the sense that under the action of the dynamic flow, the ordinary Hausdorff $d$-measure of any compact set eventually falls strictly below its initial value)
if and only if there exists a Hausdorff -Riemann $d$-measure that decreases exponentially at any time. In this sense, the latter measure of a compact set can play the role of a Lyapunov function, providing an ``anytime'' certificate of contraction that bridges dimension theory with nonlinear dynamics and enables new contraction tests beyond the conventional integer framework.
}
\section{Introduction}

Contraction theory \cite{Lohmiller_Slotine,Aminzare_Sontag_2014,Forni_Sep_2014,Forni_Sepulchre19,Bullo_book}, together with related concepts such as incremental stability \cite{DAngeli_2002} and convergent systems \cite{Pavlov_book,Pavlov_demi}, has proven to be a powerful framework for the analysis of nonlinear systems and the design of automatic controllers; see, for example, the survey \cite{TSUKAMOTO2021135}. At its core, {\mc there} lies the notion of system contractivity, which means that all trajectories converge to one another over time, regardless of their initial states.

This understanding of contraction is closely related to a weaker concept called $k$-contraction. Following V.~I.~Arnold, a mapping is said to be {\em $k$-contracting} in \cite[Sec.~3.1]{Ilyushenko} if it decreases the volume element of any $k$-dimensional plane.

%Recently, Wu and co-authors \cite{Kanevski} advanced the notion of $k$-contraction to encompass systems that contract any $k$-dimensional parallelotop. This approach has proven to be especially effective in the context of 2-contractive systems. For them, a related elegant argument from \cite{Muldowney93} permits one to exclude the existence of limit cycles. Specifically, assuming a limit cycle, located in a simply connected positively invariant area, along with appealing to existence of the surface of minimal area circumscribed by this cycle,\footnote{The classical Plateau problem concerns the existence of a minimal surface with a specified boundary. This problem was initially posed by J.L.~Lagrange in 1760. Notwithstanding, the problem is named after Joseph Plateau, who conducted experiments with soap films. Existence of solution was independently established by J.~Douglas and T.~Rad\'{o} in the 1930s.} result in a contradiction: that surface is transformed into one of a lesser area due to $2$-contraction, which is impossible since the resulting surface is still circumscribed by the same cycle. In turn, the absence of limit cycles and the closing lemma ultimately entail that any 2-contractive system on $\mathbb{R}^n$ exhibits a non-oscillatory behavior and all its trajectories eventually rest around an equilibrium point. This combination of two properties, though weaker than the conventional global asymptotic stability, has proven to be remarkably useful in a number of applications; for a survey of which, we refer the reader to \cite{DAngeli21}.

Recently, Wu and co-authors \cite{Kanevski} extended the notion of $k$-contraction to systems that contract every $k$-dimensional parallelotope. This approach has proven particularly effective in the study of $2$-contractive systems. For such systems, an elegant argument from \cite{Muldowney93} allows one to exclude the existence of limit cycles.

Specifically, suppose that a limit cycle lies in a simply connected positively invariant region. The existence of a surface of minimal area bounded by this cycle\footnote{The classical Plateau problem concerns the existence of a minimal surface with a specified boundary. This problem was initially posed by J.L.~Lagrange in 1760. Notwithstanding, the problem is named after Joseph Plateau, who conducted experiments with soap films. Existence of a solution was independently established by J.~Douglas and T.~Rad\'{o} in the 1930s.} leads to a contradiction: under $2$-contraction, this surface is mapped to another surface of strictly smaller area, which is impossible because the resulting surface remains bounded by the same cycle.

Consequently, the absence of limit cycles, together with the closing lemma, implies that any $2$-contractive system on $\mathbb{R}^n$ exhibits non-oscillatory behavior and that all trajectories eventually settle near an equilibrium point. Although this combination of properties is weaker than conventional global asymptotic stability, it has proven remarkably useful in a number of applications; for a survey, see \cite{DAngeli21}.

All previous studies of $k$-contraction dealt with integer $k$'s and classical volumes/measures. In this paper, we generalize the notion of $k$-contraction to non-integer dimensions $k$ and to more flexible geometric measures. Our motivation stems from the well-known efficacy of dimensional and volume-based instruments in nonlinear dynamics, which is particularly demonstrated by the Hausdorff and Lyapunov dimension theories \cite{Douad,Smith86,Temam,Leonov_PS,BLReitman,AfraimovichPesin1987,Kuznetsov16,Rui_Kato}.

A central role in our analysis is played by a family of metric-dependent measures, which we call the {\em Hausdorff-Riemann elliptic $d$-measures}. These measures extend the classical Hausdorff $d$-measure by incorporating a state-dependent Riemannian metric and therefore provide a natural geometric tool for describing volume evolution of compact sets along trajectories. In this framework, contraction of sets can be studied through the monotonic decay of appropriately constructed measures rather than through classical integer-dimensional volumes. This broader concept is called the {\em uniform $d$-contraction} and means, in a nutshell, that under the action of the dynamic flow, the Hausdorff-Riemann $d$-measure of any compact set decreases relative to its initial value, starting from some finite time instant.

Our analysis relies on two complementary perspectives: the {\em first Lyapunov method}, based on characteristic exponents, and the {\em second Lyapunov method}, based on Lyapunov functions. Within each framework, we derive conditions that are both necessary and sufficient for uniform $d$-contraction.
Moreover, we prove that uniform $d$-contraction is equivalent to the existence of a suitable state-dependent metric tensor under which the associated Hausdorff–Riemann $d$-measure decreases monotonically in time. In this sense, the proposed measures act as {\em Lyapunov functions for compact sets}, yielding an anytime characterization of contraction.
In contrast to standard approaches based on compound matrix logarithmic norms \cite{Muldowney93,Kanevski,DAngeli21}, our framework accommodates variable metrics. In addition, we introduce the notion of a Hausdorff–Riemann elliptic $d$-measure, combining the construction of elliptic $d$-measures due to Douady and Oesterl'{e} \cite{Leonov_PS} with a Riemannian metric structure.

The main theoretical results are supported by three examples involving a rigid body subject to an external torque, the R\"ossler system, and the Langford system respectively, where we derive bounds ensuring uniform contraction. As the examples show, manual calculations may suffice for simple systems; however, more complex applications may require numerical tools - potentially those discussed in \cite{Louzeiro20222610,PeterGiesl2023} - to compute appropriate metrics. In the first example, the Jacobian is independent of the control input, which limits the applicability of the techniques from \cite{Kanevski}. {\mc To illustrate benefits from the developed apparatus in the field of theoretical analysis, tractable criteria for orbital stability of periodic solutions of autonomous ODE's are obtained.}

From a technical standpoint, the proofs of the key results presented here rely heavily on arguments borrowed from \cite{PoMaNonlin,MaPo_automatica,MaPo_automaticaII,KawanMaPo_automatica}, which focus on upper and lower bounds of the so called restoration entropy - the notion that appears in studies related to the remote state estimation problem.

{\bf Paper outline.}
{\mc Section~\ref{sec.motiv} offers an introductory discussion focused on motivation issues.}
Section~\ref{UDC} introduces the class of uniformly $d$-contractive systems.
Section~\ref{CUDCLE} {\mc presents} necessary and sufficient conditions for uniform $d$-contraction in terms of Lyapunov characteristic exponents.
Section~\ref{CUDCLF} characterizes uniform $d$-contraction via a state-dependent (Riemannian) metric.
We then provide necessary and sufficient conditions for $d$-contraction based on the existence of measures, called Hausdorff-Riemann {\mc elliptic d-}measures, that decay monotonically along trajectories, thus playing the role of set{\mc -based} Lyapunov functions.
Specializing to $d=2$, we derive new, tractable criteria certifying orbital stability of periodic solutions.
Illustrative examples appear in Sections~\ref{Tennis_bat}, \ref{ross.sec}, and \ref{sec:langford}.
Technical proofs of the supplementary  results are deferred to the Appendix.

The following notations are adopted in the paper:
\begin{itemize}
\item $S_x^r$ and $\Ball{r}{x}$ are the sphere and open Euclidean ball with a radius of $r$ centered at $x$, respectively;
\item $:=$ means ``is defined to be'', ``is used to denote'';
\item $\dist{x}{O}$ is the distance from the point $x$ to the set $O$;
\item $N^O_\ve := \{x: \dist{x}{O} < \ve\}$ is the $\ve$-neighborhood of the set $O$;
\item $I_n$ is the $n \times n$ unit matrix;
\item $Df(x)$ is the Jacobian matrix of the function $f$ at the point $x$;
\item $\inter O$ is the interior of the set $O$;
\item $\lfloor d\rfloor$ is the integer floor of the real number $d$, i.e., the maximal integer not exceeding $d$;
\item $\|\cdot\|_2$ is the standard Euclidean norm of a vector from $\br^n$ or the spectral norm of a matrix;
\item $\spr{\cdot}{\cdot}$ is the standard Euclidean inner product in $\br^n$;
\item $\diag(a_1,\ldots,a_n)$ is the diagonal matrix with the listed diagonal entries.
\end{itemize}
\section{Some motivation issues}
\label{sec.motiv}
A classical result, due to Liouville, states
that for any linear ODE with time-varying coefficients
\begin{equation}\label{linsys_At}
\dot \xi = A(t) \xi, \quad \xi \in \mathbb{R}^n
\end{equation}
the determinant of the fundamental matrix $X(t)$
satisfies the following relation:
\[
\det X(t) = \exp \left[\int_{0}^{t} \tr A(s)\, ds \right] \det X(0).
\]
Hence if $\tr A(t) < 0 \ \ \forall t$, then $\det X(t)$ decreases over time; if moreover, $\sup_t \tr A(t) < 0$, then $\det X(t)$ exponentially goes to zero as $t \to \infty$. Meanwhile, the $n$-dimensional volume (i.e., the Lebesgue measure) of any measurable
set $E \subset \br^n$ is multiplied by $|\det X(t)|$ when passing from $E$ to its image $X(t)E$ under the action of the dynamic flow of \eqref{linsys_At}.
Thus, if the volume of $E$ is finite, then the volume of its image similarly decreases and goes to zero in the same respective conditions.
\par
An important application of this result to the theory of nonlinear dynamics is concerned with the situation where $A(t)$ is the Jacobian matrix of a vector nonlinear differential equation:
\begin{equation}\label{given_system}
  \dot x = f(x) \in \br^n
\end{equation}
%\[
%\dot x = f(x), \quad x \in \mathbb{R}^n,
%\]
evaluated along some of its solutions $x_\star(t)$:
\begin{equation*}
%\label{linearization-}
A(t) := D f\big[ x_\star(t) \big].
\end{equation*}
In this case, the Liouville’s formula implies that if ${\rm div}\, f(x) < 0\; \forall x$, any infinitesimal $n$-dimensional volume element shrinks as time evolves, and goes to zero as $t \to \infty$ if in addition $\sup_{t \geq 0} {\rm div}\, f[x_\star(t)] < 0$. This situation is referred to as \emph{$n$-contraction}.

Recently, Wu et al.~\cite{Kanevski} employed the theory of multiplicative and additive compound matrices (see, e.g., \cite[Ch.~3]{BLReitman}), along with logarithmic matrix norms (see, e.g., \cite[Ch.~4]{BLReitman}), to introduce the notion of \emph{$k$-contraction}. Loosely speaking, such contraction means that infinitesimal $k$-dimensional volume elements vanish as time progresses, where $k$ is an integer not exceeding $n$. Furthermore, they derived {\em sufficient} conditions for the $k$-contraction in terms of inequalities involving the logarithmic norms of the fundamental matrices associated with the $k$-compound variational equations.
Apart from general theoretical contributions, \cite{Kanevski} discusses several impressive control-related applications in the context of $k$-contractive systems, thus highlighting the importance and practical relevance of contraction properties.

These insights motivate further development of the topic and naturally raise extra issues. First and foremost, is it possible to provide an exhaustive and somewhat constructive
criteria for $k$-contractivity? Another natural inquiry arises in the situation where the system passes, in a continuous fashion, from being, e.g., 4-contractive to 3-contractive since its parameters gradually change due to some reason. Whether the system preserves some contraction properties in between, and if so, in which sense and is it possible to use them for analysis of the system's dynamics?
\iffalse
 As a motivating example for our study, suppose a system depends on certain parameters. During its evolution, changes in these parameters may alter the contraction properties of the system, e.g., from being 4-contractive to 3-contractive, and eventually to 2-contractive. This naturally raises the question: what happens during the transition between 4- and 3-contraction? we characterize this situation using volume elements other than those in a $k$-dimensional Euclidean space? More importantly, can we establish tractable and verifiable {\em necessary and sufficient} conditions that ensure contraction in this generalized sense?
\fi
Can this sense be expressed in terms of ordinary volume elements in a $k$-dimensional Euclidean space or another means are required? This paper responds to these and other inquiries.
\par
 To this end, we extend the concept of $k$-contraction to the case of non-integer $k$'s and present two complementary
necessary and sufficient criteria for such contractivity of a nonlinear system. This set consists of the following:
\begin{itemize}
\item[i)] A criterion that uses the characteristic exponents of the system and so is akin in spirit to the first Lyapunov approach;
\item[ii)] A criterion that is in the vein of the second Lyapunov's method when being applied to the variational equation of the system.
\end{itemize}
The property addressed by these criteria means that contraction of the Hausdorff $d$-measure comes into effect if the experiment duration is large enough.
At the same time, we additionally show that whenever that property holds, the dynamic flow reduces another, yet similar $d$-measure at any time. This measure is based on properly introducing a Riemannian metric on $\br^n$ and is sandwitched between two copies of the above Hausdorff measure, each with its own scaling factor.
Since this new measure can be used in the role of a Lyapunov function when studying \(d\)-contraction, it may be particularly valuable for control design.
\iffalse
\begin{itemize}

%\mc a \emph{second Lyapunov method for the variational equation}, yielding a flow–adapted metric tensor and the decay rate of \(d\)-volume elements;
\item[iii)] An enhancement of the criterion ii) that states that a
{\color{red} an \emph{extension of the second Lyapunov method} (Lyapunov functions) that constructs a \(d\)-measure which decays monotonically with the flow under \(d\)-contraction.}
\end{itemize}
\fi

\section{Uniformly $d$-contractive systems}
\label{UDC}

\subsection{Hausdorff measure as a generalization of the $k$-measured volume}

This section presents a background material on a generalization of the $d$-dimensional volume with an integer $d$ to a possibly fractional $d \ge 0$. To this end, we consider  compact subsets $K \subset \mathbb{R}^n$, as well as $d\in [0, \infty)$, $\varepsilon>0$, and various coverings of $K$ by finitely many open balls $B_{x_i}^{r_i}$
with radii $r_i\le\varepsilon$. The {\it Hausdorff $(d,\varepsilon)$-pre-measure of $K$} is defined as
\begin{equation}
\label{measure_eps_defn}
\mu_{d,\varepsilon}(K)=\inf \sum_{i}r_i^d,
\end{equation}
where the infimum is over all the above coverings. The {\it Hausdorff $d$-measure} of the set $K$ is the limit
\[
\mu_d(K)=\lim_{\varepsilon\rightarrow 0}\mu_{d,\varepsilon}(K)=\sup_{\varepsilon>0}\mu_{d,\varepsilon}(K).
\]
The following facts are well known (see, e.g., Secs. 4.1, 4.2 and 4.7 in \cite{mattila}) and will be used in the sequel.
\begin{proposition}
\label{prop_H_mes}
For any $d \in [0, \infty)$, the following statements hold:
\begin{enumerate}[{\bf i)}]
\item Both $\mu_{d,\ve}(\cdot)$ and $\mu_d(\cdot)$ are well-defined and are outer measures on $\br^n$;
\item The restriction of $\mu_d(\cdot)$ to the Borel $\sigma$-algebra is a complete Borel measure;
\item If $\mu_d(K)>0$ then $\mu_{d'}(K)=\infty$ for any $d'<d$ and $\mu_{d'}(K)=0$ for any $d'>d$;
%\item If $\mu_d(K)<\infty$ then $\mu_{d'}(K)=0$ for any $d'>d$;
\item $\mu_d(x+c K) = |c|^d \mu_d(K)$ for all $x \in \br^n, c \in \br$, where the sum is in the Minkowski sense.
\end{enumerate}
\end{proposition}

The Hausdorff $d$-measure is a generalization of the Lebesgue measure for possibly fractional numbers $d$. If $d=n$, they differ by only a scalar factor. For integer $d \in [0,n)$ and Borel subsets of $d$-dimensional manifolds in $\br^n$ (and even more general sets, see, e.g., \cite{Federer} for details), the Hausdorff $d$-measure is equal to the Riemannian $d$-volume up to a normalization constant \cite{Federer}. In fact, there are various definitions of the Hausdorff $d$-measure in the literature, e.g., coverings by arbitrary open sets with diameters less than $\varepsilon$ can be used instead of balls. However, the resulting measures differ from $\mu_d$ only by a scaling factor. The results of this paper are insensitive to this.

Due to iii) from Prop. \ref{prop_H_mes}, the following infimum and supremum assume a common value
\begin{equation}
\label{def.dimen}
{\rm dim}_H(K):=\inf_{d \geq 0} \{\mu_{d }(K)=0\}=\sup_{d \geq 0}\{\mu_d(K)=\infty\},
\end{equation}
which is called the {\em Hausdorff dimension} of the set $K$. This concept has found many fruitful applications in the theory of nonlinear dynamics \cite{Douad,Smith86,Temam,Leonov_PS,BLReitman,Kuznetsov16,Rui_Kato}.

\subsection{Definition of the uniform $d$-contraction}
\label{def.dcont}
We consider a dynamical system described by an autonomous vector ODE \eqref{given_system},
%\begin{equation}\label{given_system}
%  \dot x = f(x),
%\end{equation}
where $f: \br^n \to \br^n$ is continuously differentiable function. The solution to (\ref{given_system}) with the initial state $x_\circ$ evaluated at time $t$ is denoted by $x(t,x_\circ)$, and the shift mapping along the solutions by $\varphi^t$, so that $\varphi^t(x_\circ)=x(t,x_\circ)$.

Throughout the paper, the system \eqref{given_system} is studied on a certain set $\mathbb{X} \subset \br^n$, while assuming the following.
\begin{asm}
\label{ass_on_X}
The solution $x(t,x_\circ)$ is defined for all $t \geq 0$ if $x_\circ$ lies in a certain open set $O$ containing $\mathbb{X}$.
\end{asm}
\begin{asm}
\label{ass_on_X+}
The set $\mathbb{X}$ equals the closure of its interior $\inter\, \mathbb{X}$, is compact and positively invariant $\varphi^t(\mathbb{X})\subset \mathbb{X}$ $\forall t>0$.
\end{asm}
\par
By Asm.~\ref{ass_on_X}, the mapping $\varphi^t$ is well defined on $O$ for any $t \geq 0$. Under this assumption, we introduce the following definition, where $ d\in[0, n]$. .
\begin{dfn}
The system is said to be
\begin{enumerate}[{\bf i)}]
%\item {\em minimally and uniformly $d$-contractive on the set $\mathbb{X}$} if there exists $ \varepsilon \in [0, 1)$ and
% a time instant $T>0$ such that for any compact set $K\subset\mathbb{X}$ with $\mu_d(K)< \infty$, the following inequality holds
\item {\em uniformly $d$-contractive on $\mathbb{X}$} if for arbitrary $ \varepsilon \in (0,1)$ there exists $\mathscr{T}(\ve) >0$ such that for any compact set $K\subset\mathbb{X}$ with $\mu_d(K)< \infty$, the following holds:
\begin{equation}
\label{def_uniform_contr}
\mu_d\big[ \varphi^T(K)\big] \le \varepsilon \mu_d(K) \qquad \forall T\ge \mathscr{T}(\ve);
\end{equation}
\item {\em exponentially $d$-contractive on $\mathbb{X}$} if there are positive numbers $\kappa$ and $\lambda$ such that for every compact set $K\subset \mathbb{X}$ with $\mu_d(K)<\infty$, the following holds:
\begin{equation}
\label{def_contr}
\mu_d \big[ \varphi^t(K) \big]\le \kappa e^{-\lambda t}\mu_d(K) \qquad\forall t\ge 0.
\end{equation}
\end{enumerate}
\label{def.contr}
\end{dfn}

We stress that $\mathscr{T}(\ve)$ is common for all compact sets $K \subset \mathbb{X}$. The attribute ``uniformly'' expresses exactly this fact. For logical consistency of Defn.~\ref{def.contr}, it is not needed that the set $\mathbb{X}$ meets Asm.~\ref{ass_on_X+}.

Despite a visible difference between the properties from i) and ii) in Defn.~\ref{def.contr}, they are in fact the same.
\begin{thmm}\label{equiv_uniform}
If Asm.~{\rm \ref{ass_on_X}} and {\rm \ref{ass_on_X+}} hold,
the two properties introduced in Defn.~{\rm \ref{def.contr}} are equivalent.
\end{thmm}
\par
The proof of this theorem is given in Sec.~\ref{sec.proofs1}, following introduction of a necessary technical apparatus in Sec.~\ref{PMEDM}.

\begin{corollary}
\label{dim_est}
If the system \eqref{given_system} is uniformly $d$-contractive on $\mathbb{X}$, then the Hausdorff dimension \eqref{def.dimen} of any compact positively invariant set $K\subset \mathbb{X}, \varphi^t(K) = K \; \forall t \geq 0$ with $\mu_d(K)<\infty$ does not exceed $d$.
\end{corollary}
\par
Indeed, inequality \eqref{def_contr} leaves only one option $\mu_d(K) =0$ and so ${\rm dim}_H(K) \leq d$ by \eqref{def.dimen}.

This illustrates that the task of estimating the Hausdorff dimension of invariant sets can be casted in the framework of uniformly $d$-contractive systems.

\subsection{Compound matrices and a related sufficient criterion for $k$-contraction with a natural $k$}
\label{sec.link}
The purpose of this section is to link the concept of $d$-contraction introduced in Defn.~\ref{def.contr} to results of the recent study \cite{Kanevski}.
This study deals with natural $d$'s and defines $d$-contraction in a different way (we refer the reader to \cite{Kanevski} for details). At the same time, the basic sufficient criterion from \cite{Kanevski} guarantees $d$-contraction not only in the sense of \cite{Kanevski} but also in the sense of Defn.~\ref{def.contr}. (Strictly speaking, the latter is true if ``not too bad'' sets $K$ are considered in \eqref{def_uniform_contr} and \eqref{def_contr}; this statement will be specified later on.) In the rest of this section, we explain this fact but, due to its secondary importance for the current paper, do not provide a formal proof.
A good reference for most of the algebraic facts and constructions used in this section is \cite{Horn}.

To start with, we introduce the Jacobian matrix $X(t,\xi):=D \varphi^t(\xi)$ of the flow related to the system \eqref{given_system}. A useful observation is that $X(t,\xi)$ is the fundamental matrix of the following linear time-varying system of ODE's
\begin{eqnarray}
\label{linearization}
  \dot X(t,\xi)&=&A(t,\xi)X(t,\xi),~~~X(0,\xi)=I_n,  \\
   A(t,\xi)&:=&D f[x(t,\xi)].\nonumber
\end{eqnarray}
Given a linear mapping $\mA:\mathbb{R}^n\to\mathbb{R}^n$ (and so the associated matrix, which is denoted by the same symbol), its $k$th \emph{exterior power}
$
\wedge^{k}\mA
$
is the linear operator in the exterior power $\wedge^{k}\mathbb{R}^n$ (the space of alternating multilinear forms) of the Euclidean space that is uniquely defined by the following specification of its value on any simple $k$–vector (“$k$–blade”):
\[
(\wedge^k \mA)\,(v_1\wedge \cdots \wedge v_k)\ :=\ (\mA v_1)\wedge \cdots \wedge (\mA v_k)
~~\forall v_j\in\mathbb{R}^n .
\]
%and extended by linearity, it is the \emph{induced} map on the $k$th exterior space that wedges the
%images $Av_i$.
The matrix $\wedge^k \mA$ related to this operator has the dimension
$\binom{n}{k}\times \binom{n}{k}$, is denoted in parallel as $\mA^{(k)}$, and is called the {\it multiplicative compound matrix of} $\mA$.
\par
Let $e_1, \ldots, e_n$ stand for the standard orthonormal basis of $\br^n$. Then all $k$-blades of the form $e_{i_1}\wedge\cdots\wedge e_{i_k}, i_1<i_2< \ldots < i_k$ form a basis of $\wedge^{k}\mathbb{R}^n$. The matrix of $\wedge^k \mA$ in this basis is in fact a 2D-array of all $k\times k$ minors of the matrix $\mA$. Furthermore, the space $\wedge^{k}\mathbb{R}^n$ can be endowed with an uniquely defined inner product for which the considered basis is orthonormal. The related spectral norm of $\mA^{(k)}$ and the singular values $\sigma_1(\mA)\ge\cdots\ge\sigma_n(\mA)$ of $\mA$ are linked by the formula
\iffalse
It is well known that
\[
(AB)^{(k)}= A^{(k)} B^{(k)},\quad  A^{(1)}=A,\quad A^{(n)}=\det(A).
\]
\fi
\begin{equation}
\label{sigma_boy}
\|\mA^{(k)}\|_2\ =\ \sigma_1(\mA)\cdots \sigma_k(\mA)\ = : \omega_k(\mA).
\end{equation}
\par
%It is well known (see, e.g. Proposition 3.3.1) in \cite{BLReitman} that
For the fundamental matrix $X(t)$ of \eqref{linsys_At}, the matrix $X(t)^{(k)}$ is fundamental for the following ODE \cite[Prop.~3.3.1]{BLReitman}:
\begin{equation}
\label{fundament}
\dot{\Xi} =A(t)^{[k]}\Xi, \quad \Xi \in \wedge^{k}\mathbb{R}^n .
\end{equation}
Here the symbol $B^{[k]}$ denotes the so-called {\it $k$th additive compound matrix} of $B \in \br^{m \times m}$, which is defined as $$B^{[k]}:= \frac{d}{d \ve}(I_m+\ve B)^{(k)}|_{\ve=0}.$$
\par
A key result of \cite{Kanevski} states that for natural $k$'s, the system (\ref{given_system}) is $k$-contractive in the sense of \cite{Kanevski} whenever
\begin{eqnarray}
\label{logarithmic_norm}
\exists c>0 \;\text{s.t.}\quad \nu\Big[Df(x)^{[k]}\Big]\le - c <0~~~\forall x\in\mathbb{R}^n, \\
\quad \text{where}\;\nu(A):= \lim\limits_{\ve\to0+} \frac{\|I_N+\ve A\|_2 -1}{\ve}\nonumber
\end{eqnarray}
is the logarithmic norm. It is well known that under minor technical assumptions, $\|X(t) X(\tau)^{-1} \|_2 \leq \exp \int_\tau^t \nu [A(s)]\;ds \; \forall \tau\leq t$ \cite{Soder06} for any linear ODE \eqref{linsys_At}. By applying this fact to the ODE \eqref{fundament} with $A(t):=A(t,\xi)$ given by \eqref{linearization}, we see that
$\Big|\Big| X(t,\xi)^{(k)}\Big|\Big|_2 \le e^{-ct }$ and thereby, $  \omega_k[X(t,\xi)]  \le e^{-ct }$ thanks to \eqref{sigma_boy}.

Now let us consider a set $K \subset \br^n$ that is compact and in addition (and unlike Defn.~\ref{def.contr}),
is countably $(\mu_d,d)$-rectifiable %and $\mu_d$-measurable
(see the definition in \cite[Sec.~3.2.14]{Federer}).
Then by the classic {\it area formula} (see, e.g., (3) in \cite[Thm.~3.2.22]{Federer}),
\begin{eqnarray*}
%\label{area}
\mu_k[\varphi^t(K)] &=& \int_{\varphi^t(K)} \mu_k(dy) =
\int_{K}  \omega_k \big[ D \varphi^t (\xi)\big] \mu_k(d\xi) \\
&=& \int_{K}  \omega_k \big[ X (t, \xi)\big] \mu_k(d\xi) \leq e^{-ct } \mu_k(K).
\end{eqnarray*}
Thus we see that the criterion \eqref{logarithmic_norm} from \cite{Kanevski} guarantees the exponential contractivity in the sense of Defn.~\ref{def.contr}, but with the caveat that \eqref{def_contr} (with $d:= k, \kappa:=1, \lambda:=c$) is guaranteed only for compact sets $K$ with the above extra property.

In contrast to \cite{Kanevski}, we do not rely on the apparatus of compound matrices \cite{BLReitman,Muldowny90} and basically operate within the framework of classical Lyapunov's approaches.

%{\color{red}  this paper mostly follows the conventional Lyapunov approach, while the compound matrices (see{}\cite{BLReitman,Muldowny90}) will be used only in the proof of the %result on orbital stability.}

\section{Characterization of the uniform $d$-contraction via the Lyapunov exponents}
\label{CUDCLE}
The {\it finite-time Lyapunov exponents} are the quantities
\[
\Lambda_i(t,x)=t^{-1}\ln \alpha_i(t,x),
\]
where
$
\quad \alpha_1(t,x) \geq \ldots \geq \alpha_n(t,x)
$
are the singular values of the Jacobian matrix $X(t,x):=D \varphi^t(x)$ of the flow. A useful observation is that the map $(t,x) \mapsto X(t,x) \in GL(\br^n)$ is a cocycle over the flow $\{ \varphi^t(\cdot)\}_{t \geq 0}$, i.e.,
\begin{equation}
\label{cocycle}
X(0,x) = I_n, \quad X[t+\tau,x] = X[\tau,\varphi^t(x)]X[t,x].
\end{equation}
%Another useful fact is that $X(t,\xi)$ is the fundamental matrix of the following linear time-variant system of ODE's
%\begin{equation}
%\label{linearization}
%  \dot X(t,\xi)=A(t)X(t,\xi),~~~X(0,\xi)=I_n, \qquad \text{where} \quad A(t):=D f[x(t,\xi)].
%\end{equation}

The long-standing interest to the Lyapunov exponents is much related to studies of expansion and contraction properties of dynamic flows.
Regarding infinitesemally small volumes, these properties are expected to be nearly the same for the flow $\varphi^t$ and its linearization. For the latter, they are reflected by the image of the unit ball
under the action of the linear operator $X(t,\xi)$. This is an ellipsoid whose semi-axes equal the singular values $\alpha_i(t,\xi)$ of $X(t,\xi)$. Thereby, they and their time-averaged logarithms $\Lambda_i(t,x)$ provide us with a tool for the afore-mentioned studies.
\par
To apply this philosophy to analysis of $d$-contraction,
we write $d \in [0,n]$ as a sum of the integer and fractional parts:
\begin{equation}
\label{decopp}
d=d_0+s,~~d_0=\lfloor d\rfloor,~~s=d-d_0 \in [0,1).
\end{equation}

For any matrix $A \in \br^{n \times n}$ with singular values $\sigma_1(A)\ge\cdots\ge\sigma_n(A)\ge 0$, we put
\begin{equation}
\label{def.omega}
\omegad(A)\ :=\ \sigma_1(A)\cdots \sigma_{d_0}(A)\,\sigma_{d_0+1}(A)^{\,s},
\end{equation}
where the multiplier $\sigma_{d_0+1}(A)^{\,s}$ is dropped if $d=n$. (We assume that $a^0 := 1 \; \forall a \geq 0$.) Finally, we denote
\begin{multline}
\Sigma_d(t,x)=\frac1t\ln\left[\omega_d(X(t,x))\right]
\\
=\sum_{i=1}^{d_0}\Lambda_i(t,x)+s\Lambda_{d_0+1}(t,x).
\label{def.sigd}
\end{multline}
The following lemma introduces a parameter $\mathbf{\Sigma_d}$ of the flow that will play a key role in our analysis.
\begin{lemma}
\label{lem.cocycle}
Let Asm.~{\rm \ref{ass_on_X}} and {\rm \ref{ass_on_X+}} hold. Then
\begin{multline}
\label{three,eq}
\mathbf{\Sigma_d} := \sup_{x\in\mathbb{X}}\uplim_{t\rightarrow\infty}\Sigma_d(t,x)
=\lim_{t\rightarrow\infty}\max_{x\in\mathbb{X}}\Sigma_d(t,x)
\\
=\inf_{t>0}\max_{x\in\mathbb{X}}\Sigma_d(t,x).
\end{multline}
\end{lemma}
\begin{proof}
By applying the Horn's inequality \cite[Prop.~2.3.1]{BLReitman}
\begin{equation}
\label{eq:submult}
\omega_d(AB)\le\omega_d(A)\omega_d(B) \qquad \forall A,B \in \br^{n \times n}
\end{equation}
to \eqref{cocycle}, we see that $\psi_{t+\tau}(x)\le\psi_{t}(x)+\psi_{\tau}(\varphi^tx)$,
where $\psi_t(x):=\log\omega_d[X(t,x)]$. This means that the cocycle $\frac1t\psi_t(x)$ is subadditive. The proof is completed by \cite[Thm.~1]{Schreiber} (see also App. B.4 in \cite{Kawan_book}).
\end{proof}

The following statement is the main result of Sec.~\ref{CUDCLE}.

\begin{thmm}
\label{firstmethod}
Let Asm.~{\rm \ref{ass_on_X}} and {\rm \ref{ass_on_X+}} hold. Then the system \eqref{given_system} is uniformly $d$-contractive on $\mathbb{X}$ if and only if $\mathbf{\Sigma_d}<0$.
\end{thmm}
\par
The proof of this theorem is given in Sec.~\ref{sec.proofs2}, following introduction of a necessary technical apparatus in Sec.~\ref{PMEDM}.
\begin{remark}
\label{rem.drop}
Even if the first requirement from Asm.~{\rm \ref{ass_on_X+}} is dropped, the inequality $\mathbf{\Sigma_d}<0$ is still sufficient for the uniform $d$-contractivity, as will be shown in the proof of Thm. {\rm \ref{firstmethod}}.
\end{remark}

\section{Characterization of the uniform $d$ contraction in the vein of the second Lyapunov method}
\label{CUDCLF}

Whereas Thm.~\ref{firstmethod} offers a useful and illuminative insight, its direct application is limited to mostly numerical analysis.
The objective of this section is to provide tractable analytical tools for establishing the uniform $d$-contraction. To this end, we follow the Lyapunov's functions approach (the second Lyapunov method).

For a mapping $g(x)$ of $x\in \mathbb{X}$ into a finite-dimensional real vector space $V$, the {\it orbital derivative} $\dotm{g}(a)$ at point $a \in \mathbb{X}$ is defined as
$\dotm{g}(a) := \frac{d}{dt} g \big[ x(t,a)\big]|_{t=0}$. If $g$ is defined and continuously differentiable in a neighborhood of $a$, then this derivative exists and is equal to $[D g(a)] f(a)$.

In this section, we consider matrix-functions $P(\cdot)$ with the following properties.
\begin{asm}
\label{kawan}
The mapping $P(\cdot)$ is defined and continuous on $\mathbb{X}$, assumes values in the set of symmetric and positive definite $n\times n$ matrices, and has the orbital derivative $\dotm{P}(a)$ at any point $a \in \mathbb{X}$. Furthermore, this derivative is orbitally continuous, i.e., the function $t \mapsto \dotm{P}[x(t,a)]$ is continuous for any $a \in \mathbb{X}$.
\end{asm}
For any $x \in \mathbb{X}$, the $n$ roots of the following polynomial equation of degree $n$ are real (see, e.g., \cite{PoMaNonlin,KawanMaPo_automatica})
\begin{equation}\label{eigenvalue_problem}
\det\big[ A\trn(x)P(x)+P(x)A(x)+\dotm{P}(x)-\lambda P(x) \big]=0.
\end{equation}
We enumerate them in the descending order
\begin{equation}
\label{roots}
\lambda_1(x)\ge\lambda_2(x)\ge\ldots\ge\lambda_n(x)
\end{equation}
and put
\begin{equation}\label{XI_P}
\Xi^{P(\cdot)}_d(x)=\lambda_1(x)+\lambda_2(x)+\lambda_{d_0}(x)+s\lambda_{d_0+1}(x).
\end{equation}
\par
The quantities $\lambda_i(x)$ have appeared in several contexts.
They were used to obtain tractable upper bounds on the topological entropy of the system in \cite{PoMaNonlin}
and on the so-called restoration entropy in \cite{KawanMaPo_automatica}. These entropies are upper limited by the maximum values assumed by the sum of the top roots $\lambda_i(x)$ on $\mathbb{X}$. Moreover, the restoration entropy can be approximated from the both
sides within this framework: for any function $P(x)$, the associated roots
provide an upper bound, and conversely, for any accuracy $\varepsilon>0$, there exists a
choice of $P(x)$ such that the resulting bound is within $\varepsilon$ of the true value, see \cite{KawanMaPo_automatica}.

The following main result of this section shows that in a somewhat similar way, the $d$-contractivity can be comprehensively characterized
in terms of $P$ and $\lambda_i$.

\begin{thmm}\label{second_method}
Suppose that Asm.~{\rm \ref{ass_on_X}} and {\rm \ref{ass_on_X+}} hold. Then
the following statements are equivalent:
\begin{itemize}
\item[i)] The system \eqref{given_system} is uniformly $d$-contractive on $\mathbb{X}$;
\item[ii)] There is a function $P(\cdot)$ such that it satisfies Asm.~{\rm \ref{kawan}} and
%    \begin{equation}
%    \label{Pmetric}
    $\Lambda:=\sup_{x\in\mathbb{X}}\Xi^{P(\cdot)}_d(x)<0$.
    % \end{equation}
\end{itemize}
\end{thmm}
\par
The proof of this theorem is given in Sec.~\ref{sec.proofs3}, following introduction of a necessary technical apparatus in Sec.~\ref{PMEDM}.
\begin{remark}
\label{rem.drop+}
If the first requirement is dropped from Asm.~{\rm \ref{ass_on_X+}}, the implication {\rm ii) $\Rightarrow$ i)} in Thm.~{\rm \ref{second_method}} remains valid, as will be shown in the proof of Thm. {\rm \ref{second_method}}.
\end{remark}

\section{Hausdorff-Riemann elliptic $d$-measures}
\label{PMEDM}
Defn.~\ref{def.contr} means that contraction of the Hausdorff $d$-measure is made manifest only after indefinitely ``large enough'' time has passed.
%occurs asymptotically (more precisely, if the duration of the experiment is indefinitely ``large enough'').
However, contraction is not guaranteed for any particular pre-specified time, whatever it might be.
In Sec.~\ref{sec:monotone-decay}, we show that then the dynamic flow constantly diminishes another, yet similar $d$-measure. Moreover, this measure is sandwitched between two copies of the above Hausdorff measure, each with its own scaling factor. It follows that this new $d$-measure is able to serve as a Lyapunov function when studying \(d\)-contraction. This fact is particularly valuable for control design, opening the door for the use of Lyapunov-based methods in the synthesis of feedback rules that render the system \(d\)-contractive. In our current study, this hint is in fact employed for the proof of Thm.~\ref{second_method}.

The above novel $d$-measures play a key role in the proofs of Thm.~\ref{equiv_uniform}, \ref{firstmethod}, and \ref{second_method}.
Their definition combines the core idea of the so-called \emph{elliptic Hausdorff \(d\)-measure} (see, e.g., \cite{BLReitman,Kuz_Reit}) with using a Riemannian metric, which defines individual concepts of distance and ellipsoid for various points $x$ in the state space. We however, start with a simpler, yet relevant case.
\iffalse
Motivated by this, we introduce such measures. Informally, by combining the \emph{elliptic}  Hausdorff \(d\)-measure (see, e.g., \cite{BLReitman,Kuz_Reit}) with the   Riemannian metric \(P\) from Theorem~\ref{second_method}, we obtain a \(d\)-measure that decreases \emph{monotonically} along trajectories, thereby certifying uniform \(d\)-contraction and enabling Lyapunov-based control synthesis.
\fi

\subsection{Elliptic $d$-measures}
\label{sec.ellipt}
%For an ellipsoid $E\subset\R^n$ with ordered semiaxes $\sigma_1(E)\ge\cdots\ge\sigma_n(E)>0$, define
%\[
%\mathrm{ecc}(E):=\frac{\sigma_1(E)}{\sigma_n(E)}\ge 1,\qquad
%\varpi_d(E):=\sigma_1(E)\cdots\sigma_{d_0}(E)\,\sigma_{d_0+1}(E)^{\,s}.
%\]
%If $E=x+A\Ball{0}{1}$ with a linear $A$, then
%\begin{equation}\label{eq:varpi=omega}
%\varpi_d(E)=\omegad(A).
%\end{equation}
Let $d\in(0,n]$ be given.
For any ellipsoid $E\subset\R^n$ with the ordered semi-axes $\varsigma_1(E)\ge\cdots\ge\varsigma_n(E)>0$, we denote
\begin{gather}
\label{def.ecc}
\ecc(E)=\frac{\varsigma_1(E)}{\varsigma_n(E)} \geq 1, \\ \qquad \varpid(E)\ :=\ \varsigma_1(E)\cdots \varsigma_{d_0}(E)\,\varsigma_{d_0+1}(E)^{\,s}, \label{deff.oomm} \\ \qquad \text{where} \; d_0:=\lfloor d\rfloor, s:=d-d_0\in[0,1) \nonumber
\end{gather}
are the integer and fractional parts of $d$. If $d$ is integer, the multiplier $\sigma_{d_0+1}(E)^{\,s}$ is dropped here. The quantity $\ecc(E)$ is called the {\it eccentricity} of the ellipsoid $E$.

For the ease of references, we first highlight the following trivial fact, which uses the quantity $\omegad(A)$ defined in \eqref{def.omega}.
\begin{lemma}
\label{lem.same}
For any $A \in \br^{n \times n}$ with $\det A \neq 0$ and $x \in \br^n$, the parameters of the ellipsoid $E:= x + A(\Ball{1}{0})$ are as follows: $\varpid (E) =\omegad(A), \varsigma_i(E)= \sigma_i(A)$.
\end{lemma}

The following concept is mainly due to Douady and Oesterl\'{e} \cite{Leonov_PS}.
\begin{dfn}[Elliptic $d$-measure]
\label{def.ellip}
For any $\varepsilon>0, \alpha \geq 1$ and a compact set $K\subset\R^n$, we put
\begin{gather}
\chi_{d,\varepsilon}^{\langle \alpha \rangle}(K)
:=\inf\Big\{\sum_i \varpi_d(E_i)\ :\ K\subset \bigcup_i E_i,\ \varsigma_1(E_i)<\varepsilon, \nonumber\\ \label{def.meas} \ecc(E_i) \leq \alpha
\Big\},
\\
\label{def.meas1}
\chi_{d}^{\langle \alpha \rangle}(K) := \sup_{\ve>0} \chi_{d,\varepsilon}^{\langle \alpha \rangle}(K) = \lim_{\ve \downarrow 0} \chi_{d,\varepsilon}^{\langle \alpha \rangle}(K),
\end{gather}
where $\inf$ is over all finite coverings $\{E_i\}$ of $K$ by ellipsoids with the properties specified in \eqref{def.meas}.
The {\em elliptic $d$-measure} of $K$ is defined as the limit
\begin{equation}
\label{ellip.meas}
\chi_{d}(K) := \lim_{\alpha \to \infty } \chi_{d}^{\langle \alpha \rangle}(K).
\end{equation}
\end{dfn}
The limit in \eqref{def.meas1} (either finite or infinite) exists and \eqref{def.meas1} is true since as $\ve \downarrow$, the collection of coverings employed in \eqref{def.meas} reduces and so the $\inf$ from \eqref{def.meas} increases. The limit in \eqref{ellip.meas} exists since as $\alpha \uparrow$, the collection of coverings with a fixed $\ve$ expands and so $\chi_{d,\varepsilon}^{\langle \alpha \rangle}(K)$ reduces.
For $\alpha=1$, the quantity $\chi_{d}^{\langle \alpha \rangle}(K) $ equals the Hausdorff $d$-measure $\mu_d(K)$.
Compared with Defn.~5.3.3 in \cite{Leonov_PS}, the inequality $ \sigma_1(E_i)<\varepsilon$ is substituted in place of $[\varpid(E_i)]^{1/d} < \ve^d$ and a bound on the eccentricity is added. Because of these alterations, we modified (by truncation) the term "elliptical Hausdorff $d$-measure" from \cite[Defn.~5.3.3]{Leonov_PS} in the wording of Defn.~\ref{def.ellip}.
\begin{lemma}[Sandwich with Hausdorff   measures]
\label{lem:sandwich}
For any compact set $K\subset\R^n$, the following is true:
\begin{equation}\label{eq:sandwich}
\mu_d(K)\ \le\ 2^n\,\chi_d(K),\qquad \chi_d(K)\ \le\ \mu_d(K).
\end{equation}
%where one may take $C_n:=2^{n}$.
\end{lemma}
\begin{proof}\footnote{Lem.~\ref{lem:sandwich} is very similar to Lemma 5.3.1 in \cite{Leonov_PS}. Meanwhile, Lem.~\ref{lem:sandwich}
handles, as was discussed, another type of $d$-measure. Furthermore, the subsequent proof of Lem. \ref{lem:sandwich}
demonstrates in addition that in \eqref{eq:sandwich}, the scaling factors do not depend on the choice of the inner product in $\br^n$.
At the same time, this proof
much retraces the arguments from the proof of Lemma 5.3.1 in \cite{Leonov_PS}.}
In \eqref{eq:sandwich}, the second inequality is trivial since any ball is an ellipsoid with the unit eccentricity, and so $\inf$ is over a smaller set in \eqref{measure_eps_defn} than in \eqref{def.meas}.
\par
To prove the first inequality, it suffices to show that $\mu_{d, \sqrt{n} \ve}(K)\le  2^n \chi_{d,\ve}\brak{\alpha}(K)$ for any compact set $K \subset \br^n$ and $\ve>0, \alpha \geq 1$. Here $\chi_{d,\ve}\brak{\alpha}(K)$ can be approximated by a sum from \eqref{def.meas} as close as desired. So thanks to i) in Prop.~\ref{prop_H_mes}, it suffices to show that
\begin{gather}
\label{key.rel}
\mu_{d, \sqrt{n} \ve}(E)\le  2^n \varpid(E) \; \forall E \quad \text{with} \quad \varsigma_1(E) < \ve .
\end{gather}
\par
To this end, we consider an orthonormal basis of $\br^n$ such that in the related coordinates $x_i$, the ellipsoid $E$ is described by the inequality
$\sum_{i=1}^n \frac{x_i^2}{\varsigma_i(E)^2} \leq 1$. We put
\begin{gather*}
E_\ast := \{y = \{y_i\}_{i=0}^{d_0} : (y_1, \ldots, y_{d_0}, 0, \ldots, 0) \in E \} \subset \Pi \\ := \bigtimes_{i=1}^{d_0} [-\varsigma_i(E), \varsigma_i(E)] , ~~ \varrho := \varsigma_{d_0+1}(E) \geq \varsigma_{n}(E) .
\end{gather*}
The parallelepiped $\Pi$ can be covered by no more than $k := \prod_{i=1}^{d_0} \left\lceil \frac{\varsigma_i(E)}{\varrho}\right\rceil$ cubes of dimension $d_0$ with the side length $2 \varrho$. By expanding each of them to $n$-cube, we obtain a covering of $E$. Every resultant cube can be covered by a $n$-ball with a radius $\sqrt{n} \varrho < \sqrt{n} \ve$. Altogether, the resulting balls cover $\Pi$ and so $E$. It follows that
\begin{gather*}
\mu_{d,\sqrt{n} \ve} (E) \leq \varrho^d k = \varrho^d \prod_{i=1}^{d_0}  \left\lceil \frac{\sigma_i(E)}{\varrho}\right\rceil = \varrho^s \prod_{i=1}^{d_0} \varrho \left\lceil \frac{\sigma_i(E)}{\varrho}\right\rceil \\ \leq \varrho^s \prod_{i=1}^{d_0} ( \sigma_i(E)+\varrho ) \overset{\sigma_i(E)\geq \varrho\, i \leq d_0}{\leq}
2^{d_0} \sigma_{d_0+1}(E)^s \prod_{i=1}^{d_0} \sigma_i(E).
\end{gather*}
This completes the proof of the first inequality from \eqref{key.rel} since the last expression does not exceed its r.h.s.
\end{proof}
\par
It is immediate from Lemma~\ref{lem:sandwich} that the Hausdorff dimension of any compact set $K$ can be easily determined based on $\chi_d$ since
$\chi_d(K)=0\Leftrightarrow \mu_d(K)=0$ and $\chi_d(K)=\infty\Leftrightarrow \mu_d(K)=\infty$. By retracing the relevnt argument from, e.g., Sec. 4.1, 4.2 and 4.7 in \cite{mattila}, it can be seen that
Prop.~\ref{prop_H_mes} remains valid if the Hausdorff $d$-measure is replaced by the elliptic $d$-measure in its formulation.

\subsection{$P$-elliptic $d$-measures}
\label{sec.pellipt}

Up to this point, it was tacitly assumed that $\br^n$ is equipped with the standard inner product $\spr{\cdot}{\cdot}$. So the singular values of linear operators (including those representing $n \times n$ matrices), semi-axes of ellipsoids in $\br^n$, and the balls were taken with respect to this product and the associated norm. Meanwhile, all the above constructions and results are trivially extended to the case of an arbitrary inner product \raisebox{6.0pt}{\rotatebox{180}{\ding{226}}}$\cdot,\cdot$\raisebox{-1.0pt}{\ding{226}} in $\br^n$.
These products are in a one-to-one correspondence with positive definite $n \times n$ matrices $P$; specifically, \raisebox{6.0pt}{\rotatebox{180}{\ding{226}}}$\cdot,\cdot$\raisebox{-1.0pt}{\ding{226}}=\raisebox{6.0pt}{\rotatebox{180}{\ding{226}}}$x,y$\raisebox{-1.0pt}{\ding{226}}$^P$$:=\spr{Px}{y} \; \forall x,y$. So the related constructions (balls, singular values, etc.) are equipped with an extra index $P$, and the term ``elliptic $d$-measure'' is supplied with the prefix $P$, thus shaping into the {\it $P$-elliptic $d$-measure}. The absence of such an index means that the construction relates to the standard inner product.

The following lemma links some ordinary and eponymous $P$-related quantities, respectively.
\begin{lemma}
\label{lem.epony}
Suppose that $P_1, P_2 \in \br^{n \times n}$ are symmetric positive definite matrices, $A \in \br^{n \times n}$ is a non-singular matrix, $S_i:= \sqrt{P}_i $ stands for the symmetric positive definite square root of $P_i$, and $E \subset \br^n$ is an ellipsoid.
\par
Then the following statements hold:
\begin{enumerate}[{\bf i)}]
\item The ball {\color{black} $\leftindex^{P_2\!}{B}_x^r = S^{-1}_2 B_0^r +x = S_{2 \to 1}^{-1} \leftindex^{P_1\!}{B}_0^r +x$, where  $S_{2 \to 1} := S_2^{-1} S_1$};
\item $\sigma_i^{P_2}(A) = \sigma_i(S_2 A S_2^{-1}) = \sigma_i^{P_1}(S_{1 \to 2}  A S_{1 \to 2}^{-1} )$ and $\omega_d^{P_2}(A) = \omega_d(S_2AS^{-1}_2) = \omega_d^{P_1}(S_{1 \to 2}  A S_{1 \to 2}^{-1} )$;
\item $\varpi_d^{P_2}(E) = \varpi_d(S_2E)= \varpi_d^{P_1}(S_{1 \to 2}  E)$, $\varsigma_i^{P_2}(E) = \varsigma_i(S_2E)= \varsigma^{P_1}_i(S_{1 \to 2}  E)$;
\item
$
\frac{1}{\omega_d (S_1S_2^{-1})} \varpi_d^{P_1}(E) \leq \varpi_d^{P_2}(E) \leq \omega_d (S_{2} S_1^{-1}) \varpi_d^{P_1} (E)
$;
\item $\frac{\sigma_n(S_2)}{\sigma_1 (S_1)} \varsigma_1^{P_1} (E) \leq \varsigma_1^{P_2} (E) \leq \frac{\sigma_1(S_2)}{\sigma_n( S_1)} \varsigma_1^{P_1} (E)$,   $\frac{\sigma_n(S_2)}{\sigma_1 (S_1)} \varsigma_n^{P_1} (E) \leq \varsigma_n^{P_2} (E) \leq \frac{\sigma_1(S_2)}{\sigma_n (S_1)} \varsigma_n^{P_1} (E)$;
\item $ \frac{\sigma_n(S_2)}{\sigma_1(S_2)} \frac{\sigma_n(S_1)}{\sigma_1 (S_1)} \ecc^{P_1}(E) \leq\ecc^{P_2}(E) \leq \frac{\sigma_1(S_2)}{\sigma_n(S_2)} \frac{\sigma_1(S_1)}{\sigma_n (S_1)} \ecc^{P_1}(E)$;
\item $\varpid^{P_2} (AE) \leq \omega_d (S_2AS_1^{-1}) \varpid^{P_1} (E) $.
\end{enumerate}
\end{lemma}
The proof is presented in the Appendix.
\subsection{Hausdorff-Riemann elliptic $d$-measure}
\label{sec.rimellipt}
This is a direct extension of the $P$-elliptic $d$-measure from the previous section to the case where the matrix $P$ varies over the space. So any location is associated with its own $P$-related inner product, metric, and quantitative interpretation of the singular values, semi-axes of ellipsoids, their  eccentricity, and the quantities \eqref{deff.oomm}.
The major alteration undertaken in this section is that in the definition \eqref{def.meas} of the elliptic $d$-measure, the quantities $\varpi_d(E_i), \varsigma_1(E_i)$, and $\ecc(E_i)$ are evaluated locally in the metric related to the center of the handled ellipsoid $E_i$.
\par
In this section, we assume a mapping $P(\cdot)$ that satisfies the following.
\begin{asm}
The mapping $P(\cdot)$ is defined on $\br^n$, its values are symmetric positive definite matrices, and
\begin{multline}
\label{two.ineq}
M_{P(\cdot)}:= \sup_{x\in \br^n} \sqrt{\sigma_1[P(x)]} < \infty, \\
 m_{P(\cdot)} :=  \inf_{x\in \br^n} \sqrt{\sigma_n[P(x)]} >0.
\end{multline}
\label{ass.p}
\end{asm}
\iffalse
{\color{red}
Here the conditions \eqref{two.ineq} are introduced to simplify the matters and do not, in essence, restrict the generality of the forthcoming results. The point is that
in this paper, the $P(\cdot)$-related constructions are used for study of the system \eqref{given_system} on an invariant set $\mathbb{X}$, which is compact.
Accordingly, the properties of $P(\cdot)$ on $\mathbb{X}$ is what really matters, whereas \eqref{two.ineq} would follow from the continuity of $P(\cdot)$ if $\mathbb{X}$ is put in place of $O$ there.
For technical reasons, we, however, have to temporarily extend our consideration, along with $P(\cdot)$, to a (possibly small) open neighborhood $O$ of $\mathbb{X}$.
This $O$ and the extension can always be selected so that \eqref{two.ineq} and overall Asm.~\ref{ass.p} do hold if $P(\cdot)$ originally is continuous on $\mathbb{X}$.}
\fi
\par
From now on, the symbol $\mathbf{c}(E)$ stands for the center of the ellipsoid $E$, and the symmetric positive definite square root of $P(x)$ is denoted by $\sqrt{P(x)}$ or $S(x)$.
The following is a direct extension of Defn.~\ref{def.ellip}.

\begin{dfn}[Hausdorff-Riemann elliptic $d$-measure]\label{def:riemmetric}
For any $\varepsilon>0$, $\alpha\ge 1$, and compact set $K\subset O$, we put
\begin{gather}
\nonumber
\leftindex^{P(\cdot)\!}\pi^{\langle\alpha\rangle}_{d,\varepsilon}(K)
:=\inf\Big\{\sum_i \varpi_d^{P[\mathbf{c}(E_i)]}\big(E_i\big)\  :\ K\subset \bigcup\nolimits_i E_i,\\  \varsigma_1^{P[\mathbf{c}(E_i)]}(E_i)<\varepsilon,\ \mathrm{ecc}^{P[\mathbf{c}(E_i)]}(E_i)\le \alpha\Big\},
\label{def.meas+}
\\
\label{def.meas1+}
\leftindex^{P(\cdot)\!}\pi^{\langle\alpha\rangle}_d(K):=\lim_{\varepsilon\downarrow 0}\leftindex^{P(\cdot)\!}\pi^{\langle\alpha\rangle}_{d,\varepsilon}(K).
\end{gather}
Here $\inf$ is over all finite coverings $\{E_i\}$ of $K$ by ellipsoids with the properties specified in \eqref{def.meas+}.
The {\em Hausdorff-Riemann elliptic $d$-measure} of $K$ is the limit
\begin{equation}
\label{ellip.meas+}
\pi_d^{P(\cdot)}(K):=\lim_{\alpha\to\infty}\leftindex^{P(\cdot)\!}\pi^{\langle\alpha\rangle}_d(K)=\sup_{\alpha>1}\leftindex^{P(\cdot)\!}\pi^{\langle\alpha\rangle}_d(K).
\end{equation}
\end{dfn}
\par
Existence of the limits in \eqref{def.meas1+} and \eqref{ellip.meas+}, and formula \eqref{ellip.meas+} are justified just as the similar facts and relations in
Sec.~\ref{sec.ellipt}. The concepts introduced in Defn.~\ref{def.ellip} and \ref{def:riemmetric} clearly coincide if the matrix-function $P(\cdot)$ is constant. If this constant is the identity matrix, the Hausdorff-Riemann elliptic $d$-measure is identical to the elliptic $d$-measure from Sec.~\ref{sec.ellipt}.
\par
The following lemma shows that the Hausdorff-Riemann elliptic $d$-measures born by various matrix-functions evaluate each other.
\begin{lemma}
Let the functions $P_i: \br^n \to \br^{n \times n}, i=1,2$ meet  Asm.~{\rm \ref{ass.p}}. Then for any compact set $K \subset \br^n$,
\begin{equation}
\label{sandwich1+}
\pi_d^{P_1(\cdot)}(K) \leq \big[ m_{P_2(\cdot)}\big]^{-d}\, \big[ M_{P_1(\cdot)}\big]^{d} \, \pi_d^{P_2(\cdot)}(K) .
\end{equation}
\label{lem.eval}
\end{lemma}
\begin{proof}
By using iv)--vi) in Lem.~\ref{lem.epony}, we see that
\begin{gather}
\nonumber
\varsigma_1^{P_1} (E) \overset{\text{v)}}{\leq} \frac{\sigma_1(S_1)}{\sigma_n( S_2)} \varsigma_1^{P_2} (E)  \overset{\text{\eqref{two.ineq}}}{\leq}
\frac{M_{P_1(\cdot)}}{m_{P_2(\cdot)}}  \varsigma_1^{P_2} (E),
 \\
 \nonumber
  \ecc^{P_1}(E) \overset{\text{vi)}}{\leq} \frac{\sigma_1(S_1)}{\sigma_n(S_1)} \frac{\sigma_1(S_2)}{\sigma_n (S_2)} \ecc^{P_2}(E) \\
\overset{\text{\eqref{two.ineq}}}{\leq} \frac{M_{P_1(\cdot)}}{m_{P_1(\cdot)}} \frac{M_{P_2(\cdot)}}{m_{P_2(\cdot)}} \ecc^{P_2}(E),
\nonumber
 \\
 \varpi_d^{P_1}(E) \overset{\text{iv)}}{\leq} \omega_d \left[ P_{1}(x)^{1/2} P_2(x)^{-1/2} \right]\varpi_d^{P_2} (E) \nonumber
\\ \overset{\text{\eqref{eq:submult}}}{\leq}
\omega_d \left[ P_{1}(x)^{1/2}\right] \omega_d \left[ P_2(x)^{-1/2} \right]\varpi_d^{P_2}(E) \nonumber\\
 \overset{\text{\eqref{def.omega}}}{\leq} \frac{\big[ M_{P_1(\cdot)}\big]^{d}}{\big[m_{P_2(\cdot)}\big]^{d}}\varpi_d^{P_2}(E).
 \label{omeg.ineq}
\end{gather}
\par
Let $\Upsilon_i (\alpha,\ve)$ stand for the collection of coverings $\mathfrak{E}=\{E_j\}$ over which $\inf$ is taken in \eqref{def.meas+} with $P(\cdot) := P_i(\cdot)$. By the foregoing,
$
\Upsilon_1[\alpha,\ve] \supset \Upsilon_2 \left[ \alpha_{1\to 2}, \ve_{1 \to 2} \right]
$, where $\alpha_{1\to 2} := \frac{m_{P_2(\cdot)}}{M_{P_2(\cdot)}} \frac{m_{P_1(\cdot)}}{M_{P_1(\cdot)}} \alpha$ and $\ve_{1 \to 2} :=  \frac{m_{P_1(\cdot)}}{M_{P_2(\cdot)}} \ve$. Hence
\begin{gather*}
\leftindex^{P_1(\cdot)}\pi^{\langle\alpha\rangle}_{d,\varepsilon}(K)
\overset{\text{\eqref{def.meas+}}}{=} \inf_{\mathfrak{E} \in \Upsilon_1 (\alpha,\ve)} \sum_i \varpi_d^{P_1[\mathbf{c}(E_i)]}\big(E_i\big) \\
\overset{\text{\eqref{omeg.ineq}}}{\leq}   \frac{\big[ M_{P_1(\cdot)}\big]^{d}}{\big[m_{P_2(\cdot)}\big]^{d}} \inf_{\mathfrak{E} \in \Upsilon_1 (\alpha,\ve)} \sum_i \varpi_d^{P_2[\mathbf{c}(E_i)]}\big(E_i\big)
\\
\leq \frac{\big[M_{P_1(\cdot)}\big]^{d}}{\big[m_{P_2(\cdot)}\big]^{d}} \inf_{\mathfrak{E} \in \Upsilon_2 \left[ \alpha_{1\to 2}, \ve_{1 \to 2} \right]  } \sum_i \varpi_d^{P_2[\mathbf{c}(E_i)]}\big(E_i\big).
\end{gather*}
It remains to pass to the limit as $\ve\to 0+$ and then as $\alpha\to \infty$ and invoke \eqref{def.meas1+} and \eqref{ellip.meas+}.
\end{proof}
\par
By Lem.~\ref{lem.eval}, the Hausdorff-Riemann elliptic measure is sandwiched between both Hausdorff $\mu_d$ and elliptic $\chi_d$ measures.
\begin{corollary}
Let a function $P: \br^n \to \br^{n \times n}, i=1,2$ meet Asm.~{\rm \ref{ass.p}}. Then for any compact set $K \subset \br^n$,
\begin{gather}
\label{sadnwich2}
[ m_{P(\cdot)}]^{d} \,\chi_d(K) \leq \pi_d^{P(\cdot)}(K) \leq [M_{P(\cdot)}]^{d} \, \chi_d(K), \\
 2^{-n} [m_{P(\cdot)}]^{d} \,\mu_d(K) \leq \pi_d^{P(\cdot)}(K) \leq [M_{P(\cdot)}]^{d} \, \mu_d(K) . \nonumber
\end{gather}
\end{corollary}
Indeed, the first two inequalities are obtained by taking first, $P_1(\cdot):= P(\cdot), P_2(\cdot):=I_n$ and second, $P_2(\cdot):= P(\cdot), P_1(\cdot):=I_n$ in \eqref{sandwich1+}. The third and fourth inequalities follow from the first two ones and \eqref{eq:sandwich}.

\subsection{Area type inequality for Hausdorff-Riemann measures}
\label{area.sec.ellip}
The well-known {\it area formula} (see, e.g., (3) in \cite[Thm.~3.2.22]{Federer}) deals with natural $d$'s and countably $(\mu_d,d)$-rectifiable, $\mu_d$-measurable (see the definitions in \cite[Sec.~3.2.14]{Federer})
sets $K, Z \subset \br^n$. It states that if $g(\cdot)$ is a continuously differentiable mapping and $g(K) \subset Z$, then (under minor technical assumptions)
\begin{equation}
\label{area}
\int_{K}  \omega_d \big[ Dg(x)\big] \mu_d(dx) = \int_{Z} N (y) \mu_d(dy).
\end{equation}
Here $N(y)$ is the number of elements in the set $\big\{x:\ g(x)=y\bigr\}$ ($N(y):= 0$ if this set is empty) and $\mu_d$ is the ordinary Hausdorff $d$-measure. With $Z:= g(K)$, \eqref{area} trivially implies the following {\it area type inequality}:
\begin{gather*}
\mu_d[g(K)] \leq \int_{g(K)} \underbrace{N (y)}_{\geq 1} \mu_d(dy) = \int_{K}  \omega_d \big[ D\ff(x)\big] \mu_d(dx) \\ \leq \mu_d(K) \, \sup_{x \in K} \omega_d \big[ D\ff(x)\big].
\end{gather*}
In this section, we obtain an analog of this inequality for the fractional $d$'s and Hausdorff-Riemann elliptic $d$-measures. (The authors are unaware of any such an analog established so far.)
\par
Throughout the section, we deal with a mapping $g(\cdot)$ satisfying the following.
\begin{asm}
\label{ass.map}
The function $g: O \to \br^n$ is defined and continuously differentiable on an open set $O \subset \br^n$.
% and the set $\mathbb{X} \subset O$ is compact.
%The Jacobian matrix $Dg(\cdot)$ is nonsingular $\sigma_n \big[ D g (x)\big] > 0$ in a vicinity of $\mathbb{X}$.
\end{asm}
\par
The main result of this section is the following.
\begin{thmm}\label{thm:area-infty}
Suppose that Asm.~{\rm \ref{ass.map}} holds, a matrix-function $P: \br^n \to\R^{n\times n}$ satisfies Asm.~{\rm \ref{ass.p}}, and the Jacobian matrix $Dg(x)$ is nonsingular $\sigma_n \big[ D g (x)\big] > 0$ for any point from a compact set $K\subset O$.
\par
Then the following inequality holds:
\begin{gather}
\label{area.for}
\pi_d^{P(\cdot)}\big[g(K)\big]\ \le\ \Omega_K  \pi_d^{P(\cdot)}(K),\quad  \text{\rm where} \\  \Omega_K:=
\max_{x\in K}\omegad\big\{S[g(x)]\,Dg(x)\,S(x)^{-1}\big\}, \quad S(x):= \sqrt{P(x)} .\nonumber
\end{gather}
\end{thmm}
\par
The rest of Sec.~\ref{area.sec.ellip} is devoted to the proof of Thm.~\ref{thm:area-infty}, which is broken into a stream of several lemmas.
By the third assumption in Thm.~\ref{thm:area-infty}, there exists $\gamma_\star>0$ such that the $\gamma_\star$-neighborhood $N_{\gamma_\star}^{K} := \{x \in \br^n : \dist{x}{K} \leq \gamma_\star\}$ of $K$ lies in $O$ and the Jacobian matrix $Dg(\cdot)$ is nonsingular in it.
For any $\gamma \in [0,\gamma^\star)$, we introduce the modulus of continuity of the derivative $Dg(\cdot)$:
\begin{gather}
\delta_{g,\gamma}(r):= \max \Bigl\{ \|D\ff(x_1) - D\ff(x_2)\|: x_1,x_2 \in N_\gamma^{K}, \nonumber\\ \|x_2 - x_1\| \leq r \Bigr\}. \label{modulus}
%\kappa_{\gamma} := \max_{x\in N_\gamma^{\mathbb{X}}}\frac{\sigma_1\big[ D \ff(x) \big] }{\sigma_n[D\ff(x)]} < \infty .
%\ov{\sigma}_f:= \max \left\{ \sigma_1 \big[ Df(x)\big]: x \in N_\gamma^Y \right\}, \quad \underline{\sigma}_f:= \min \left[ \sigma_n \big[ Df(x)\big]: x \in N_\gamma^Y \right\}.
\end{gather}
By Asm.~\ref{ass.map}, $\delta_{g,\gamma}(r) \searrow 0$ as $r \to 0+$.
In the next lemma, $\mathbf{diam}\, M:= \sup_{x,y \in M} \|x-y\| \; \forall M \subset \br^n$.
\begin{lemma}
\label{lem.dif0}
Suppose that $\gamma \in (0,\gamma_\star)$ and $M \subset \br^n$ is a compact set such that $M \cap K \neq \emptyset$ and $ r:= \mathbf{diam}\,M < \gamma$. Then $M \subset N_\gamma^{K} \subset O$ and for any $x_\ast \in M$, the following inclusions hold:
\begin{gather}
 \ff(M) \subset \ff(x_\ast) + D \ff (x_\ast)[M - x_\ast] +  \Ball{r \delta_{g,\gamma}(r)}{0}, \nonumber
\\
  \ff(x_\ast) + D \ff (x_\ast)[M - x_\ast] \subset  \ff(M) +  \Ball{r \delta_{g,\gamma}(r)}{0} .
  \label{difff}
  \end{gather}
\end{lemma}
\begin{proof} We pick $x_\ast \in M\cap K$, consider an arbitrary $x \in M$,
and put $x(\theta) := (1-\theta)x_\ast + \theta x$. Then for any $\theta \in [0,1]$,
\begin{gather*}
\|x(\theta) - x_\ast\|_2 = \theta \|x - x_\ast\|_2 \leq \mathbf{diam}\,M < \gamma.
\end{gather*}
%\begin{gather*}
%\label{segment}
%\|x(\theta) - y\| \leq
%(1-\theta) \|x_\ast - y\| + \theta \|x-y\|  \leq \gamma \\ \leq (1-\theta) \mathbf{diam}\,M + \theta \mathbf{diam}\,M < \gamma \quad \forall \theta \in [0,1].
%\end{gather*}
Hence the straight line segment bridging $x_\ast$ and $x$ lies in $N_\gamma^{K} \subset O$; in particular, $ x \in N_\gamma^{K}$ and so $M \subset N_\gamma^{K}$.
For $\Delta:= x - x_\ast$, we have $\|\Delta\|_2 \leq \mathbf{diam} M = r$. Hence
\begin{multline*}
%\label{bas.argmod}
\ff(x) -\ff(x_\ast) - D \ff (x_\ast) \Delta = \\ \int_0^1 \left[ D\ff(x_\ast + \theta \Delta ) -  D \ff (x_\ast) \right] \Delta \; d \theta,
\\
\Rightarrow
\left\| \ff(x) -\ff(x_\ast) - D \ff (x_\ast) \Delta \right\|_2 \\ \leq \|\Delta\|\int_0^1 \left\| D\ff(x_\ast + \theta \Delta ) -  D \ff (x_\ast) \right\|_2  \; d \theta
\overset{\text{\eqref{modulus}}}{\leq} r \delta_{g,\gamma}(r).
\end{multline*}
It remains to note that this inequality implies \eqref{difff}.
\end{proof}
Another technical facts are listed in the following.
\begin{lemma}
\label{lemm.elip}
Let $P \in \br^{n \times n}$ be a positive definite matrix,
$E \subset \br^n$ be an ellipsoid, $A \in \br^{n \times n}$ be a nonsingular matrix, and $x \in \br^n$.
%Let $P$ be a positive definite $n \times n$ matrix.
%For any ellipsoid $E \subset \br^n$ and nonsingular matrix $A \in \br^{n \times n}$,
Then
the following relations hold:
\begin{enumerate}[{\bf i)}]
\item \label{lemm.elip1} $\varsigma_i^P(x+cE) = c \varsigma_i^P(E)$, $ \ecc^P(x+cE) = \ecc^P(E)$, $ \varpid^P(x+cE) = c^d \varpid^P(E)$\; $\forall c >0$;
\item \label{lemm.elip2} $ \varpid^P(AE)\ \le\ \omega_d^P(A)\, \varpid^P(E)$;
\item \label{lemm.elip3}
$E+\Ball{\lambda}{0} \subset E^+_\lambda := \mathbf{c}(E) + \left[ 1+ \frac{\lambda}{\varsigma_n(E)}\right] [E - \mathbf{c}(E)]$, $\varpid\bigl(E^+_\lambda\bigr)
 \ \le\ \Bigl(1+\frac{\lambda}{\varsigma_n(E)}\Bigr)^{\!d}\, \varpid(E) \; \forall \lambda >0$ ;
 \item \label{lemm.elip4} $ \varsigma_1^P(x+AE) \leq \sigma_1^P(A) \varsigma_1^P(E)$,  $\varsigma_n^P(x+AE) \geq  \sigma_n^P(A) \varsigma_n^P(E)$, $\ecc^P(x+AE) \leq \frac{\sigma_1^P(A)}{\sigma_n^P(A)}\,\ecc^P(E)$.
\end{enumerate}
\end{lemma}
This lemma is proven in the Appendix.

To proceed, we need the following notations:
\begin{gather}
\sigma^+_{\gamma}:= \max_{x \in N_\gamma^{K}} \sigma_1 \big[ D\ff(x)\big] \overset{\text{\rm Asm.~\ref{ass.map}}}{<} \infty,
 \\
   \sigma_{\gamma}^-:= \min_{x \in N_\gamma^{K}} \sigma_n \big[ D\ff(x)\big] \overset{\text{\rm Asm.~\ref{ass.map}}}{>} 0,
   \label{def.sigg}
   \\
\label{quant0}
\Omega_{\gamma} := \max_{x\in N_\gamma^{K}}\omega_d\big[ S(g(x)) D\ff(x) S(x)^{-1}\big] \overset{\text{\rm Asm.~\ref{ass.map}}}{<} \infty ,
\\
\label{defki}
e_{P(\cdot)}:= \frac{M_{P(\cdot)}}{ m_{P(\cdot)}}, \; \xi_{\ve,\alpha,\gamma}:=2 \frac{e_{P(\cdot)}}{\sigma^-_\gamma} \alpha \delta_{g,\gamma}\left( \frac{2\ve}{m_{P(\cdot)}} \right), \;
\\
\label{defki1}
\epsilon_\gamma(\ve,\alpha):= e_{P(\cdot)}  M_{P(\cdot)} \sigma^+_\gamma (1 + \xi_{\ve,\alpha,\gamma}) \ve ,
\\
\label{defki2}
a(\alpha):= \big[e_{P(\cdot)} \big]^4 \frac{ \sigma_\gamma^+}{\sigma_\gamma^-} \alpha.
\end{gather}
\begin{lemma}
\label{thm:unit-constant}
For any $\gamma \in (0,\gamma_\star)$, $\ve \in \left(0, \frac{\gamma m_{P(\cdot)}}{2} \right)$, $\alpha \geq 1$, and compact set $K\subset O$,
%we denote
%\begin{multline}
%\label{defki}
%e_{P(\cdot)}:= \frac{M_{P(\cdot)}}{ m_{P(\cdot)}}, \; \xi_{\ve,\alpha,\gamma}:=2 \frac{e_{P(\cdot)}}{\sigma^-_\gamma} \alpha \delta_{g,\gamma}\left( \frac{2\ve}{m_{P(\cdot)}} %\right), \; \\ \epsilon_\gamma(\ve,\alpha):= e_{P(\cdot)}  M_{P(\cdot)} \sigma^+_\gamma (1 + \xi_{\ve,\alpha,\gamma}) \ve , \\ a(\alpha):= \big[e_{P(\cdot)} \big]^4 \frac{ %\sigma_\gamma^+}{\sigma_\gamma^-} \alpha.
%\end{multline}
%Then
the following inequality holds:
\begin{gather}
\label{def.max}
\leftindex^{P(\cdot)\!}\pi^{\langle a(\alpha)\rangle}_{d,\epsilon_\gamma(\varepsilon,\alpha)}[g(K)] \leq \leftindex^{P(\cdot)\!}\pi_{d,\varepsilon}^{\langle \alpha \rangle}(K) (1+\xi_{\ve,\alpha,\gamma})^d \Omega_\gamma .
\end{gather}
\end{lemma}
The lemma is proven in the Appendix.

{\bf Proof of Theorem~\ref{thm:area-infty}:} It suffices to consecutively let $\ve \to 0+, \alpha \to \infty$, and $\gamma \to 0+$ in \eqref{def.max}, with noting that $\Omega_\gamma \to \Omega$ as $\gamma \to 0+$ by \eqref{area.for} and \eqref{quant0}. \endproof

\section{Proofs of Theorems~\ref{equiv_uniform}, \ref{firstmethod}, and \ref{second_method}}
\label{sec.proofs}
Now we are in a position to prove these theorems.
\subsection{Proof of Theorem~\ref{equiv_uniform}}
\label{sec.proofs1}
Since the implication \textbf{ii) $ \Rightarrow $ i)} is trivial, it remains to show that
\textbf{i) $ \Rightarrow $ ii)}.
To this end, we choose $\ve \in [0,1)$, consider $\mathscr{T}(\ve)$ from \textbf{i)} in Defn.~\ref{def.contr}, and then pick
$T> \mathscr{T}(\ve)$
%and $\gamma >0$ so that the $\gamma$-neighborhood $N^{\mathbb{X}}_\gamma$ of $\mathbb{X}$ lies in the open set from Asm.~\ref{ass_on_X}.
Let $K \subset \mathbb{X}$ be a compact set. Then $\varphi^t(K) \subset \mathbb{X}\; \forall t \geq 0$ by Asm.~\ref{ass_on_X} and \ref{ass_on_X+}. Thanks to the continuity of the mapping $\varphi^T$, the set $\varphi^T(K)$ is also compact. So \eqref{def_uniform_contr} can be iterated to yield that
\begin{equation}
\label{uter,hh}
\mu_d \big[ \varphi^{k T}(K)\big] \leq \ve^k \mu_d \big[ K\big] \quad k=0,1,2, \ldots.
\end{equation}
By writing an arbitrary $t \geq 0$ in the form $t = k T+\tau$ with $\tau \in [0, T)$ and an integer $k \geq 0$, we see that
\begin{gather*}
\mu_d \big[ \varphi^{t}(K)\big] \overset{\text{\eqref{eq:sandwich}}}{\leq} 2^n  \chi_d \Big\{ \varphi^{\tau}\big[ \varphi^{kT}(K)\big]\Big\}.
\end{gather*}
Applying Thm.~\ref{thm:area-infty} with $P(x):= I_n \; \forall x, g(\cdot) := \varphi^\tau(\cdot)$ to the compact set $\mathscr{K}:=\varphi^{kT}(K) \subset \mathbb{X}$ yields that
\begin{gather}
\label{teo.lem.in}
\chi_d\big[\varphi^\tau(\mathscr{K})\big]\ \le\ \Omega_\tau  \chi_d(\mathscr{K}) \leq \Omega  \chi_d(\mathscr{K}), \\ \text{\rm where}~~ \Omega_\tau:=
\max_{x\in \mathbb{X}}\omegad\big\{ D\varphi^\tau(x)\big\} , \Omega := \max_{\tau \in [0,T]} \Omega_\tau < \infty .\nonumber
\end{gather}
Hence
\begin{gather*}
\mu_d \big[ \varphi^{t}(K)\big]
\leq 2^n \Omega \chi_d \big[ \varphi^{kT}(K)\big]
\overset{\text{\eqref{eq:sandwich}}}{\leq} 2^n \Omega \mu_d \big[ \varphi^{kT}(K)\big] \\ \overset{\text{\eqref{uter,hh}}}{\leq} 2^n \Omega \ve^k \mu_d \big[ K\big]
 = \frac{\Omega 2^n}{\ve^{\tau/T}}  \big(\ve^{1/T} \big)^{kT+\tau} \mu_d \big[ K\big] \\ \overset{\ve<1, \tau \leq T}{\leq}
\frac{\Omega 2^n}{\ve} \big(\ve^{1/T} \big)^{t} \mu_d \big[ K\big] = \frac{\Omega 2^n}{\ve} e^{t \frac{1}{T} \ln \ve } \mu_d \big[ K\big].
\end{gather*}
This completes the proof of ii) since $\ln \ve <0$. \hfill \endproof
\subsection{Proofs of Theorem~\ref{firstmethod} and Remark~\ref{rem.drop}}
\label{sec.proofs2}
These proofs are prefaced by two technical lemmas.
\begin{lemma}
\label{linear_lemma}
For any $d \in (0,n]$, there exists $c_d>0$ such that
any matrix $A \in \mathbb{R}^{n \times n}$ has a compact set $K \subset \mathbb{R}^n$ for which
\[
0 < \mu_d(K) < \infty
\quad\text{\rm and}\quad
\mu_d(A K) \geq c_d \omega_d(A)\mu_d(K).
\]
\end{lemma}
{\bf Proof:}
If $d=n$, the claim evidently holds with $K:= \Ball{1}{0}$. Let $d < n$.
We consider the singular value decomposition $A = U_+ \varSigma U_-$ of $A$. Here $U_\pm \in \br^{n\times n}$ are orthogonal matrices and $\varSigma$ is the diagonal matrix $\mathbf{diag}\, \big[ \sigma_1(A), \ldots, \sigma_n(A)\big]$. Since $\mu_d(U_\pm K) = \mu_d(K)$ for any compact set $K \subset \br^n$, it suffices to prove the lemma assuming that $A$ is diagonal $A = \varSigma$.
\par
We use $d_0 \leq n-1$ and $s$ from \eqref{decopp} and the self-similar Cantor set $C \subset [0,1]$ of Hausdorff dimension $s$ (whose existence is proven in, e.g., \cite[Sec. 4.10]{mattila}), and put $K:=[0,1]^{d_0} \times C \subset \mathbb{R}^{d_0+1} \subset \mathbb{R}^n$, where $\mathbb{R}^{d_0+1}$ is embedded into $\mathbb{R}^n$ by placing zeros to the missing (right) entries of $x \in \mathbb{R}^{d_0+1}$. It is easy to see by inspection that this embedding does not affect $\mu_d(K)$.
\par
According to \cite[Sec. 4.10]{mattila}), $0 < \mu_{s}(C) < \infty$. So $\mu_d(K) >0$ by \cite[Prop.~7.1]{Falconer52}, whereas Thm.~C (with $q:=0, E:=[0,1]^{d_0}, \xi[\Ball{r}{x}]:=r^{d_0}, F:= C, \xi^\prime[\Ball{r}{x}]:=r^{s}$) from \cite{AtJeGu23} guarantees that
\begin{equation}
\label{inn.cantor}
\mu_d(K) \leq \mu_{d_0}\big( [0,1]^{d_0} \big) \times \mu_s(C) = \mu_s(C) < \infty.
\end{equation}
Meanwhile, $AK = [0,\sigma_1(A)] \times [0,\sigma_2(A)] \times \cdots \times [0,\sigma_{d_0}(A)] \times \sigma_{d_0+1}(A) C \subset \mathbb{R}^{d_0+1} \subset \mathbb{R}^n$. The proof is completed by \cite[Prop.~7.1]{Falconer52}, by which there exists a constant $c_d$ (that depends on $d$ only) such that
\begin{gather*}
\mu_d(K)\geq c_d \mu_{d_0} \left( [0,\sigma_1(A)] \times [0,\sigma_2(A)] \times \cdots   [0,\sigma_{d_0}(A)]\right) \\ \times \mu_s (\sigma_{d_0+1}(A) C)
\\
= c_d \sigma_1(A) \times \cdots \times \sigma_{d_0}(A) \sigma_{d_0+1}(A)^s \mu_s ( C) \\ = \omega_d(A) \mu_s(C) \overset{\text{\eqref{inn.cantor}}}{\geq} c_d \omega_d(A) \mu_d(K). \qquad \qedhere
\end{gather*}
\par
{\bf Proof of Theorem \ref{firstmethod}:}
{\em Sufficiency.} Let $\mathbf{\Sigma_d}<0$. By \eqref{three,eq}, $\lim\limits_{t\rightarrow\infty}\max\limits_{x\in\mathbb{X}}\Sigma_d(t,x) <0$. So there are
$T >0, \lambda > 0$ such that for $t \geq T$, the following inequality holds
\begin{gather*}
\max_{x\in\mathbb{X}}\Sigma_d(t,x) \leq - \lambda \overset{\text{\eqref{def.sigd}}}{\Rightarrow} \omega_d\big[ X(t,x) \big] \leq e^{- t \lambda} \quad \forall x\in\mathbb{X}.
\end{gather*}
Hence for any compact set $K \subset \mathbb{X}$, we have
\begin{gather*}
t \geq T \Rightarrow \mu_d\big[ \varphi^t(K)\big] \overset{\text{\eqref{eq:sandwich}}}{\leq} 2^n \chi_d\big[  \varphi^t(K)\big] \\ \overset{\text{\eqref{quant0},\eqref{teo.lem.in}}}{\leq}
2^n \max_{x\in\mathbb{X}} \omega_d\big[ X(t,x) \big] \chi_d\big[K\big] \leq 2^n e^{- t \lambda}  \chi_d\big[K\big] \\ \overset{\text{\eqref{eq:sandwich}}}{\leq}  2^n e^{- t \lambda}  \mu_d\big[K\big].
\end{gather*}
We denote $c := \max_{t \in [0,T]} \max_{x\in\mathbb{X}} \omega_d\big[ X(t,x) \big] < \infty$.
For $t\in [0,T]$, similar arguments show that
$$
\mu_d\big[ \varphi^t(K)\big] \leq 2^n c \mu_d\big[K\big] = 2^n c  e^{\lambda T} e^{-\lambda t} \mu_d \big[K\big] .
$$
Hence \eqref{def_contr} does hold with $\kappa := 2^n \max\{c 2^n   e^{\lambda T}; 1\}$.
\par
{\em Necessity.} Let the system be uniformly $d$-contractive on $\mathbb{X}$. By Thm.~\ref{equiv_uniform}, there exist $\kappa>0$ and $\lambda>0$ such that \eqref{def_contr} holds for any compact set $K \subset \mathbb{X}$. We pick $x_\ast \in \inter \mathbb{X}$ and $\ve_\ast>0$ such that $\Ball{\ve_\ast}{x_\ast} \subset \mathbb{X}$.
For any $t >0$,
applying Lem.~\ref{linear_lemma} to $A:= D \varphi^t(x_\ast) = X(t,x_\ast)$ shows that there exists a compact set $K_t$ for which
\begin{gather}
0 < \mu_d(K_t) < \infty \qquad \text{and} ~\nonumber \\ \mu_d\big[ X(t,x_\ast) K_t\big] \geq c_d \omega_d\big[ X(t,x_\ast)\big]\mu_d(K_t),
\label{basjs1}
\end{gather}
where the constant $c_d>0$ does not depend on $t$ and $\ve_\ast$.
\par
Whenever $0<\ve < \ve_\dagger:= \ve_\ast/(2\dist{x_\ast}{K_t}+1)$, the compact set $K_{\ve,t} := x_\ast + \ve \big[ K_t - x_\ast \big]$ is such that
\begin{gather}
\nonumber
\dist{x_\ast}{K_{\ve,t}} \leq \ve \dist{x_\ast}{K_t} < \ve_\ast \Rightarrow K_{\ve,t} \subset \Ball{\ve_\ast}{x_\ast} \subset \mathbb{X},
 \\
 \nonumber
  \mu_d\big[ X(t,x_\ast)K_{\ve,t} \big] = \ve^d  \mu_d\big[X(t,x_\ast) K_{t} \big]
  \\
  \overset{\text{\eqref{basjs1}}}{\geq} \ve^d c_d \omega_d\big[ X(t,x_\ast)\big]\mu_d(K_t).
  \label{frst.ienq}
\end{gather}
Now we introduce the smooth mappings $\psi_t(\cdot) := \varphi^t(\cdot), g(\cdot) := X(t,x_\ast)\psi_t^{-1}(\cdot)$
%, the point $y_{\ast,t} := \varphi^t(x_\ast)$,
and the compact set $\mathscr{K}_{\ve,t} := \varphi^t(K_{\ve,t}) \subset \mathbb{X}$. By \eqref{def_contr},
\begin{multline}
\mu_d\big[ \mathscr{K}_{\ve,t} \big] = \mu_d \big[ \varphi^t(K_{\ve,t}) \big]
\\
\le \kappa e^{-\lambda t}\mu_d(K_{\ve,t}) = \kappa \ve^d e^{-\lambda t}\mu_d(K_{t}) .
\label{exxppnn}
\end{multline}
Meanwhile,
\begin{gather}
\nonumber
g \big[ \mathscr{K}_{\ve,t} \big] = X(t,x_\ast)\psi_t^{-1}\big[ \varphi^t(K_{\ve,t}) \big] = X(t,x_\ast)K_{\ve,t} \Rightarrow
\\
\mu_d \big[ X(t,x_\ast)K_{\ve,t} \big] = \mu_d \big[ g ( \mathscr{K}_{\ve,t} ) \big]
\overset{\text{\eqref{eq:sandwich}}}{\leq} 2^n \chi_d \big[ g ( \mathscr{K}_{\ve,t} ) \big].
\label{ssccnnd.in}
\end{gather}
Applying Thm.~\ref{thm:area-infty} to $P(\cdot) := I_n$, $K:= \mathscr{K}_{\ve,t}$ yields that
\begin{gather}
\nonumber
\chi_d \big[g(\mathscr{K}_{\ve,t})\big]\ \le\ \Omega_{\ve,t}  \chi_d (\mathscr{K}_{\ve,t}) \overset{\text{\eqref{eq:sandwich}}}{\leq}
\Omega_{\ve,t}  \mu_d (\mathscr{K}_{\ve,t}), \quad \text{where}
\\
\nonumber
\Omega_{\ve,t} := \max_{y\in \mathscr{K}_{\ve,t}}\omegad\big[ Dg(y) \big] =
\max_{y\in \mathscr{K}_{\ve,t}}\omegad\big[ X(t,x_\ast)D \psi_t^{-1}(y) \big]
\\
\nonumber
= \max_{y\in \mathscr{K}_{\ve,t}}\omegad\Big\{ X(t,x_\ast) X \big[ t,\psi_t(y)\big]^{-1} \Big\}
\\
\overset{x:=\psi_t(y)}{=\!=\!=\!=\!=} \max_{x \in K_{\ve,t}}\omegad\big[ X(t,x_\ast) X ( t,x)^{-1} \big] .
\label{def.ooom}
\end{gather}
By invoking \eqref{frst.ienq} and \eqref{ssccnnd.in}, we thus see that
\begin{gather*}
\ve^d c_d \omega_d\big[ X(t,x_\ast)\big]\mu_d(K_t) \leq  \mu_d\big[ X(t,x_\ast)K_{\ve,t} \big]
\\
\leq 2^n\chi_d \big[ g ( \mathscr{K}_{\ve,t} ) \big] \leq 2^n \Omega_{\ve,t}  \mu_d (\mathscr{K}_{\ve,t})
\overset{\text{\eqref{exxppnn}}}{\leq} 2^n \Omega_{\ve,t} \kappa \ve^d e^{-\lambda t}\mu_d(K_{t})
\\
\Rightarrow
\omega_d\big[ X(t,x_\ast)\big] \leq 2^n \Omega_{\ve,t} \frac{\kappa}{c_d}  e^{-\lambda t} \quad \forall \ve \in (0,\ve_\dagger).
\end{gather*}
Since the singular values are Lipschitz continuous functions of the matrix entries, the mapping $A \in \br^{n \times n} \mapsto \omega_d(A)$ is continuous.
So $\Omega_{\ve,t} \to 1$ as $\ve \to 0+$ by \eqref{def.ooom}. Hence letting $\ve \to 0+$ in the last inequality obtained so far yields that for all $t>0$ and $x_\ast \in \inter \mathbb{X}$,
\begin{gather*}
\omega_d\big[ X(t,x_\ast)\big] \leq 2^n \frac{\kappa}{c_d}  e^{-\lambda t}
\\
\overset{\text{\eqref{def.sigd}}}{\Rightarrow} \Sigma_d(t,x_\ast) \leq - \lambda + \frac{\ln \kappa - \ln c_d +n \ln 2}{t}.
\end{gather*}
The proof is completed by invoking \eqref{three,eq} and that $\mathbb{X}$ is the closure of $\inter \mathbb{X}$ by Asm.~\ref{ass_on_X+} and so
$\sup_{x_\ast \in \inter X} \Sigma_d(t,x_\ast) = \max_{x_\ast \in X} \Sigma_d(t,x_\ast)$ since $\Sigma_d(t,x_\ast)$ depends on $x_\ast$ continuously.
\hfill
\endproof
\par
{\bf Proofs of Remark~\ref{rem.drop}:} This proof is by inspection of the sufficiency part in the proof of Thm.~\ref{firstmethod}. \hfill \endproof
\subsection{Proof of Theorem~\ref{second_method} and Remark~\ref{rem.drop+}}
\label{sec.proofs3}
They are prefaced with preliminary facts. The first of them uses the notation $w^\prime(\tau+) := \lim_{t \to \tau+} \frac{w(t) - w(\tau)}{t-\tau}$ for the {\it right derivative} of a function $w(\cdot)$ at point $\tau$. We start with a simple fact from the classic calculus.
\begin{lemma}
Let a function $w(\cdot): [a,b] \to \br$ be continuous and have the right derivative at any $t \in [a,b)$. Then the following statements hold with any $\lambda \in \br$:
\begin{enumerate}[{\bf i)}]
\item if $w^\prime(t+) \leq \lambda \; \forall t \in [a,b)$, then $w(\beta) \leq w(\alpha) + \lambda (\beta-\alpha)$ whenever $b \geq \beta \geq \alpha \geq a$;
\item if $w^\prime(t+) \leq \lambda w(t), w(t) >0 \; \forall t \in [a,b)$, then $w(\beta) \leq w(\alpha) e^{\lambda (\beta-\alpha)}$ whenever $b \geq \beta \geq \alpha \geq a$.
\end{enumerate}
\end{lemma}
\begin{proof}
{\bf i)} For any $\ve >0$, the set $E:= \{t \in [\alpha,\beta]: w(t) \leq w(\alpha) + (\lambda + \ve) (t-\alpha) \} \ni \alpha$ is nonempty and closed due to the continuity of $w(\cdot)$.
Hence $\tau := \sup \{t: t \in E\} \in E$. We are going to show that $\tau = \beta$. Suppose to the contrary that $\tau < \beta$. Then for $t\in (\tau,\beta], t \approx \tau$, we have
\begin{gather*}
w(t) = w(\tau) + \frac{w(t) - w(\tau)}{t-\tau}(t-\tau) \overset{t \approx \tau, w^\prime(\tau+) \leq  \lambda}{\leq} \\ w(\alpha) + (\lambda + \ve) (\tau-\alpha)  + (\lambda + \ve)(t-\tau) \\ = w(\alpha) + (\lambda + \ve) (t-\alpha)  \Rightarrow t \in E,
\end{gather*}
in violation of the definition of $\tau (< t)$. Hence $\beta \in E$. The proof is completed by letting $\ve \to 0+$.
\par
{\bf ii)} It suffices to apply i) to the function $t \mapsto \ln w(t)$.
\end{proof}
\par
Another technical fact is from matrix algebra.
\begin{lemma}
For any $A,B,C \in \br^{n \times n}$,
\begin{gather*}
\sigma_i(ABC) \leq \sigma_1(A) \sigma_i(B) \sigma_1(C) \; \forall i =1,\ldots, n, \\ \qquad \omega_d(ABC)  \leq \sigma_1(A) \sigma_1(C) \omega_d(B) \; \forall d \in [0,n].
\end{gather*}
\label{lem.sv}
\end{lemma}
\begin{proof}
We note that $\|M\|_2 = \sigma_1(M) \; \forall M \in \br^{n \times n}$ and denote by $\mathfrak{L}$ the set of all linear subspaces of $\br^n$. By the min-max characterization of the singular values,
\begin{gather*}
\sigma_i(AB) = \min_{L \in \mathfrak{L}: \dim L = n-i+1 } \max_{x \in L, \|x\|_2 =1} \|AB x\|_2 \\
\leq \sigma_1(A) \min_{L \in \mathfrak{L}: \dim L = n-i+1 } \max_{x \in L, \|x\|_2 =1} \|B x\|_2 = \sigma_1(A) \sigma_i(B);
\\
\sigma_i(BC) = \sigma_i(C\trn B\trn) \leq \sigma_1(C\trn)\sigma_i(B\trn) = \sigma_1(C)\sigma_i(B).
\end{gather*}
By combining the two just established facts, we arrive
at the first inequality from the lemma's statement. The second inequality follows from the first one and \eqref{def.omega}.
\end{proof}
The next lemma is borrowed from \cite{Smith86}.
\begin{lemma}
\label{lem:Smith}
For the fundamental matrix $X(t)$ of a linear ODE $\dot z=F(t)z, t \geq 0$ with $F(t) \in \br^{n \times n}$, $X(0) = I_n$ and any $ t \geq 0, d=1,\dots,n$, we have
\begin{gather*}
\omega_d\big[ X(t) \big] \le\ \exp\!\left\{\frac12\int_0^t \sum_{i=1}^d \nu_i(\tau) \,d\tau\right\},
\end{gather*}
where $\nu_1(\tau)\ge\cdots\ge\nu_n(\tau)$ are the eigenvalues of the symmetric matrix
$F(\tau)+F(\tau)\trn$.
\end{lemma}
\begin{lemma}
\label{lem.hjdfg}
Let the ODE from Lem~{\rm \ref{lem:Smith}} be defined for $t \in [-\delta,\infty)$ with some $\delta>0$, and let $d \in [0,n]$.
Invoking $d_0$ and $s$ from \eqref{decopp}, we denote
\begin{gather}
\label{def.aleph}
\aleph^+_d(t):= \nu_1(t)+\cdots+\nu_{d_0}(t) + s \nu_{d_0+1}(t), \quad \\ \aleph_d^-(t):= \nu_n(t)+\cdots+\nu_{n-d_0+1}(t) + s \nu_{n-d_0}(t). \nonumber
\end{gather}
Then the following inequalities hold:
\begin{gather}
\label{w}
\omega_d[X(t)] \leq  \exp\!\left\{\frac12\int_0^t \aleph_d^+(\tau)\,d\tau\right\} \quad \forall t \geq 0, \qquad \\
\omega_d[X(t)] \leq  \exp\!\left\{\frac12\int_0^t \aleph_d^-(\tau)\,d\tau\right\} \quad \forall t \in [-\delta, 0]. \nonumber
\end{gather}
\end{lemma}
\begin{proof}
To prove the first inequality, we denote $X:=X(t)$ and use Lem.~\ref{lem:Smith} with $i=d_0, d_0+1$ to see that
\begin{gather*}
\ln \big\{ \sigma_1\big[ X(t)  \big] \cdots\sigma_{d_0}\big[ X (t) \big]\big\} \le\ \frac12\int_0^t \sum_{i=1}^{d_0} \nu_i(\tau)\,d\tau,
\\
\ln \big\{ \sigma_1\big[ X(t)  \big] \cdots\sigma_{d_0+1}\big[ X(t)  \big]\big\} \le\ \frac12\int_0^t  \sum_{i=1}^{d_0+1} \nu_i(\tau)\,d\tau .
\end{gather*}
It remains to multiply the first and second inequalities by $1-s$ and $s$, respectively, and to sum up the results, while taking into account \eqref{def.omega} and \eqref{def.aleph}.
\par
To justify the second inequality in \eqref{w}, we observe that $\mathscr{X}(\theta):= X(-\theta), \theta \in [0, \delta]$ is the fundamental matrix  for the ODE $\dot{z} = \mathscr{F}(\theta) z, \theta \in [0, \delta]$, where $\mathscr{F}(\theta) := - F(-\theta)$. The eigenvalues of
$\mathscr{F}(\theta)+\mathscr{F}(\theta)\trn = - \big[ F(-\theta)+F(-\theta)\trn \big]$ are $-\nu_n(-\theta)\ge\cdots\ge-\nu_1(-\theta)$.
So Lem.~\ref{lem:Smith} yields that for any $i=1, \ldots, n$ and $ t \in [-\delta, 0]$,
\begin{gather*}
\sigma_1\big[ X(t) \big]\cdots\sigma_i\big[ X(t)\big]  =  \sigma_1\big[ \mathscr{X}(-t) \big]\cdots\sigma_i\big[ \mathscr{X}(-t) \big] \\
\le\ \exp\!\left\{\frac12\int_0^{-t} - \big[\nu_n(-\tau)+\cdots+\nu_{n-i+1}(-\tau)\big]\,d\tau \right\}
\\
\overset{\theta:= -\tau}{=\!=\!=\!=} \exp\!\left\{\frac12\int_{0}^t \big[\nu_n(\theta)+\cdots+\nu_{n-i+1}(\theta)\big]\,d\theta\right\}.
\end{gather*}
The proof is completed like in the case of the first inequality from \eqref{w}.
\end{proof}
\par
Given a mapping $P(\cdot): \mathbb{X} \to \br^{n \times n}$ satisfying Asm.~\ref{kawan} and $t\ge0$, $x\in\mathbb{X}$, we introduce the following notations:
\begin{gather}
\label{def.yxode}
 A(x):=Df(x), X_x(t):= D\varphi^t(x), S(x) := \sqrt{P(x)},
\\
\label{def.yxode1}
 Y_x(t)\ :=\ S\!\big[\varphi^t(x)\big]\,X_x(t)\,S(x)^{-1}.
\end{gather}
We also recall that the symbols $\dot{S}(x)$ and $\dot{P}(x)$ stand for the orbital derivatives.
\begin{lemma}
If $x \in \mathbb{X}$, the matrix function $t \mapsto Y_x(t) $ is fundamental for the following ODE and $Y_x(0) = I_n$
\begin{multline}
\dot{z} = F\big[\varphi^t(x)\big]\,z,~~\text{\rm where}
\\
F(y):=S(y)A(y)S(y)^{-1}+\dot S(y)S(y)^{-1}.
\label{comp.ode}
\end{multline}
\label{lem.funode}
\end{lemma}
{\bf Proof:}
Formula \eqref{def.yxode1} yields that $Y_x(0) = I_n$.
The proof is completed by differentiating equation \eqref{def.yxode1}:
\begin{gather*}
\dot{Y}_x(t) = \dot{S} \!\big[\varphi^t(x)\big]\,D\varphi^t(x)\,S(x)^{-1} \\+ S \!\big[\varphi^t(x)\big]\,D f \!\big[\varphi^t(x)\big] D\varphi^t(x)\,S(x)^{-1}
\\
= \dot{S} \!\big[\varphi^t(x)\big] S \!\big[\varphi^t(x)\big]^{-1} S \!\big[\varphi^t(x)\big]\,D\varphi^t(x)\,S(x)^{-1} \\ + S \!\big[\varphi^t(x)\big]\,A S \!\big[\varphi^t(x)\big]^{-1} S \!\big[\varphi^t(x)\big] D\varphi^t(x)\,S(x)^{-1} \\ = F\big[\varphi^t(x)\big] Y_x(t).
\qquad \qedhere
\end{gather*}
\begin{lemma}
\label{lem.eigenv}
For any $x \in \mathbb{X}$, the roots $\lambda_1(x)\ge\lambda_2(x)\ge\ldots\ge\lambda_n(x) \geq 0$ of equation \eqref{eigenvalue_problem} form the sequence of the eigenvalues of the symmetric matrix $F(x)+F(x)\trn$ enumerated in the descending order.
\end{lemma}
{\bf Proof:} We note that $P = SS \Rightarrow \dot{P} = \dot{S}S+ S \dot{S}$. It remains
to invoke the definition of $F(\cdot)$ in \eqref{comp.ode} to see that
\begin{gather*}
F+F\trn\ = S A S^{-1}+\dot{S} S^{-1} + S^{-1} A\trn S + S^{-1} \dot{S} \\ = S^{-1}\big[ S^2 A + S\dot{S}  + A\trn S^2 + \dot{S} S \big] S^{-1}
\\
= S^{-1}\big[ P A + A\trn P + \dot{P} \big] S^{-1} \qquad \text{and hence}
 \\
 \det\big[ (F+F\trn) - \nu I_n \big] =0 \\ \Leftrightarrow \det \big[ S^{-1}( P A + A\trn P + \dot{P} ) S^{-1}- \nu I_n \big]  =0 \\ \Leftrightarrow
\det \big[ P A + A\trn P + \dot{P}- \nu P \big]  =0. \qquad \qedhere
\end{gather*}
\par
The next lemma uses the quantity $\Xi^{P(\cdot)}_d(x)$ from \eqref{XI_P} and $\Lambda:=\sup_{x\in\mathbb{X}}\Xi^{P(\cdot)}_d(x)$ from Thm.~\ref{second_method}.
\begin{lemma}
\label{lem.kkeeyy}
If $\Lambda < \infty$, the following inequality holds:
\begin{gather}
\label{chto.nado}
\omega_d[Y_x(t)] \leq e^{\frac{\Lambda}{2} t} \qquad \forall t \geq 0.
\end{gather}
\end{lemma}
{\bf Proof:}
For $x \in \mathbb{X}$, Lem.~\ref{lem.hjdfg} and \ref{lem.funode} imply that
\begin{gather*}
\omega_d[Y_x(t)] \leq  \exp\!\Bigl\{\frac12\int_0^t \bigg[ \sum_{i=1}^{d_0} \lambda_i[\varphi^\tau(x)] + s \lambda_{d_0+1}[\varphi^\tau(x)]\bigg],d\tau\Bigr\}
%\\ \lambda_1[\varphi^\tau]+\cdots+\lambda_{d_0}[\varphi^\tau]
%\\ + s \lambda_{d_0+1}[\varphi^\tau]\big)(x)\,d\tau\Bigr\}
\\
\overset{\text{\eqref{XI_P}}}{=} \exp\!\left\{\frac12\int_0^t \Xi^{P(\cdot)}_d [\varphi^\tau(x)] \,d\tau\right\} \overset{\varphi^\tau(x) \in \mathbb{X}}{\leq}
e^{\frac{\Lambda}{2} t}. \qquad \qedhere
\end{gather*}
\par
The following last preliminary fact is an immediate corollary of Proposition 27 in \cite{KawanMaPo_automatica}.
\begin{lemma}
\label{le.prop27}
  For any $t>0$, there exists a mapping $P(\cdot)~:~\mathbb{X}\rightarrow \br^{n \times n}$ that meets Asm.~{\rm \ref{kawan}} and is such that
  \begin{equation}
  \label{prop27}
  \frac12\Xi^{P(\cdot)}_d(x) \le \frac1t\ln\omega_d\big[ X(t,x) \big] \overset{\text{\eqref{def.sigd}}}{=}\Sigma_d(t,x)~~\forall x\in\mathbb{X}.
  \end{equation}
\end{lemma}
\par
{\bf Proof of Thm.~\ref{second_method}:} \textbf{ii) $ \Rightarrow $ i):}
Due to Asm.~\ref{ass_on_X+} and \ref{kawan}, there exist $q_\pm \in (0,\infty )$ such that $\sigma_1[S(x)] \leq q_+, \sigma_1[S(x)^{-1}] \leq q_- \; \forall x \in \mathbb{X}$. Lem.~\ref{lem.sv} guarantees that
\begin{gather*}
\omega_d \big[Y_x(t) \big] \leq q_-q_+ \omega_d \big[X_x(t) \big] , \\ \qquad \omega_d \big[X_x(t) \big] \overset{\text{\eqref{def.yxode1}}}{=} \omega_d \left\{ S\!\big[\varphi^t(x)\big]^{-1}\,Y_x(t)\,S(x) \right\} \\
\leq q_-q_+ \omega_d \big[Y_x(t) \big]
\\
\Rightarrow
\uplim_{t\rightarrow\infty}\frac1t\ln\omega_d\big[ Y_x(t) \big] \leq \uplim_{t\rightarrow\infty}\frac1t\ln\omega_d \big[ X_x(t) ] \\ + \uplim_{t\rightarrow\infty} \frac{\ln q_-q_+}{t} = \uplim_{t\rightarrow\infty}\frac1t\ln\omega_d \big[ X_x(t) ] .
\end{gather*}
By the same arguments, the converse inequality also holds and, overall,
$$
\uplim_{t\rightarrow\infty}\frac1t\ln\omega_d\big[ Y_x(t) \big] = \uplim_{t\rightarrow\infty}\frac1t\ln\omega_d\big[ X_x(t) \big].
$$
 Since $Y_{x}(t+\tau) = Y_{\varphi^\tau}(t) Y_x(\tau)$, inequality \eqref{eq:submult} implies that the mapping
$(t,x) \mapsto\frac1t\ln\omega_d \big[ Y_x(t)\big]$ is a subadditive cocycle. So the statement of Lem.~\ref{lem.cocycle} extends to it
by App.~B.6 in \cite{Kawan_book} (see also \cite[Thm.~1]{{Schreiber}}. Therefore
\begin{gather*}
{\bf \Sigma_d}\overset{\text{\eqref{def.sigd},\eqref{three,eq}}}{=\!=\!=\!=}\sup_{x\in\mathbb{X}}\uplim_{t\rightarrow\infty}\frac1t\ln\omega_d\big[ X_x(t)\big] \\ =\sup_{x\in\mathbb{X}}\uplim_{t\rightarrow\infty}\frac1t\ln\omega_d\big[ Y_x(t)\big]
\overset{\text{\eqref{chto.nado}}}{\leq} \Lambda/2 \overset{\text{(a)}}{<}0,
\end{gather*}
where (a) holds by ii) in Thm.~\ref{second_method}. The proof of i) in Thm.~\ref{second_method} is completed by
Thm.~\ref{firstmethod}.
\par
\textbf{i) $ \Rightarrow $ ii):} For any $t>0$, Lem.~\ref{le.prop27} guarantees existence of a mapping $P(\cdot)$ that meets Asm.~{\rm \ref{kawan}} and \eqref{prop27}. For any $\ve>0$, \eqref{three,eq} implies that $\max_{x\in\mathbb{X}}\Sigma_d(t,x) \leq \mathbf{\Sigma_d} + \ve $ if $t$ is large enough, where $\mathbf{\Sigma_d} <0$ by Thm.~\ref{firstmethod}. The proof is completed by taking $\ve < - \mathbf{\Sigma_d}/2$.
\endproof
\par
{\bf Proof of Remark~\ref{rem.drop+}:} It suffices to inspect the proof of ii) $\Rightarrow$ i) part in Thm.~\ref{second_method}, with regard to Rem.~~\ref{rem.drop}. \hfill \endproof
\section{Anytime contraction}\label{sec:monotone-decay}
In this section, we deal with Hausdorff-Riemann elliptic $d$-measures
\eqref{ellip.meas+} born of functions $P:\mathbb{X}\to\R^{n\times n}$ satisfying Asm.~\ref{kawan}. However, for this measure to be correct, the mapping $P(\cdot)$ should satisfy Asm.~\ref{ass.p} and, in particular, be defined on the entire $\br^n$.
We further assume that $P(\cdot)$ is preliminarily extended from $\mathbb{X}$ to $\br^n$ to meet Asm.~\ref{ass.p}. This is possible by the following.
\begin{lemma}
If $P:\mathbb{X}\to\R^{n\times n}$ meets Asm.~{\rm \ref{kawan}}, then there exists its extension to $\br^n$ that satisfies Asm.~{\rm \ref{ass.p}}.
\end{lemma}
\begin{proof}
By entry-wise applying the classic Tietze extension theorem \cite{Engel89} to $P(\cdot)$, we first extend $P(\cdot)$ to $\br^n$ as a continuous matrix function $P_1(\cdot)$ whose values are symmetric $n \times n$-matrices (not necessarily positive definite). Then we transform $P_1(\cdot)$ into not only a continuous but also a bounded extension $P_2(x) := \frac{P_1(x)}{1+ \|P_1(x)\|_2 \dist{x}{\mathbb{X}}}$. Here the distance function $\dist{x}{\mathbb{X}}:= \min_{y \in \mathbb{X}} \|x-y\|$ is well known to be Lipschitz. Since the matrix $P_2(x) = P(x)$ is positive definite on $\mathbb{X}$, the same is true in some neighborhood of $\mathbb{X}$ by continuity: there are $\gamma>0, \alpha>0$ such that $P_2(x) \geq \alpha I_n \; \forall x \in N^{\mathbb{X}}_\gamma$. Denoting $p:= \sup_{x \in \br^n} \|P_2(x)\|_2$, it is easy to see that the function $P(x) := P_2(x) + 2p \frac{1+\gamma}{\gamma}\frac{\dist{x}{\mathbb{X}}}{1+\dist{x}{\mathbb{X}}} I_n$ does satisfy  Asm.~{\rm \ref{ass.p}} and extends $P(\cdot)$.
\end{proof}
\par
Uniform $d$-contractivity in the sense of Defn.~\ref{ass_on_X+} means that shrinking of the $d$-measure may be observed only after some and maybe, long time passed.
In this section, we show that under a special choice of a Hausdorff-Riemann elliptic $d$-measure, this shrinking takes an infinitesimally short time to come into effect and is observed in any time. Moreover, the $d$-measure decays in an enhanced sense introduced in the following.
\begin{dfn}
\label{def.decay}
Let $z_{\mathfrak{z}}(\cdot) : [0,\infty) \to \br$ be a function that depends on a parameter $\mathfrak{z} \in \mathfrak{Z}$. This function
is said to {\em uniformly decay with an exponential intensity} if there is $\ve>0$ such that
$z_{\mathfrak{z}}(t) \leq e^{- \ve(t-\tau)} z_{\mathfrak{z}}(\tau)\; \forall t \geq \tau \geq 0, \mathfrak{z} \in \mathfrak{Z}$.
%The coefficient $\ve$ is called the {\em rate of decay}.
\end{dfn}
If $z_{\mathfrak{z}}(\cdot)$ is an absolutely continuous function of $t$, the property from Defn.~\ref{def.decay} holds if and only if   $\dot{z}_{\mathfrak{z}}(t) \leq - \ve z_{\mathfrak{z}}(t)$ for almost all $t$ whenever $\mathfrak{z} \in {\mathfrak{Z}}$.
\par
The following is in fact the key result of the paper.
\begin{thmm}
\label{second_method_2+}
Suppose that Asm.~{\rm \ref{ass_on_X}} and {\rm \ref{ass_on_X+}} hold. Then
the following statements are equivalent:
\begin{enumerate}[{\bf i)}]
\item The system \eqref{given_system} is uniformly $d$-contractive on $\mathbb{X}$;
\item There exists a matrix function $P:\mathbb{X}\to\R^{n\times n}$ such that Asm.~{\rm \ref{kawan}} is true and the associated Hausdorff-Riemann elliptic $d$-measure $\pi_d^{P(\cdot)} \big[ \varphi^t(K) \big]$ decays with an exponential intensity uniformly over the compact subsets $K \subset \mathbb{X}$ with $\pi_d^{P(\cdot)} \big[K \big]<\infty$.
\end{enumerate}
\end{thmm}
\par
As will be shown, this theorem is a nearly immediate corollary of the following.
\begin{proposition}
Let Asm.~{\rm \ref{ass_on_X}} and {\rm \ref{ass_on_X+}} hold, a mapping $P(\cdot)~:~\mathbb{X}\rightarrow \br^{n \times n}$ satisfy Asm.~{\rm \ref{kawan}}, and $\Lambda:=\sup_{x\in\mathbb{X}}\Xi^{P(\cdot)}_d(x) < \infty$, where $\Xi^{P(\cdot)}_d(x)$ is defined by \eqref{XI_P}. Then for any compact set $K \subset \mathbb{X}$ with the finite Hausdorff-Riemann elliptic $d$-measure and $t \geq \tau \geq 0$,
\begin{gather}
\pi_d^{P(\cdot)}\!\big[\varphi^t(K)\big] \le\ e^{\Lambda (t-\tau)/2}\ \pi_d^{P(\cdot)}\!\big[\varphi^\tau(K)\big] .
\label{qed.inter1}
\end{gather}
If in addition $\inf\limits_{x\in\mathbb{X}} \lambda_n(x) > -\infty$, the function $\Upsilon(t):=\pi_d^{P(\cdot)} \big[ \varphi^t(K) \big] $ is absolutely continuous and $\dot{\Upsilon}(t) \leq \Lambda \Upsilon(t)/2$ for almost all $t \geq 0$.
\iffalse
\begin{enumerate}[{\bf i)}]
\item $\pi_d^{P(\cdot)} \big[ \varphi^t(K) \big] \leq e^{\Lambda(t-\tau)/2} \pi_d^{P(\cdot)} \big[ \varphi^\tau(K) \big] \; \forall t \geq \tau \geq 0$;
\item The function  $\Upsilon(t):=\pi_d^{P(\cdot)} \big[ \varphi^t(K) \big] $ is absolutely continuous and $\dot{\Upsilon}(t) \leq \Lambda \Upsilon(t)/2$ for almost all $t \geq 0$ provided that $\inf\limits_{x\in\mathbb{X}} \lambda_n(x) > -\infty$.
\end{enumerate}
\fi
\label{prop.fgg}
\end{proposition}
To prove this, we invoke the notations \eqref{def.yxode}, \eqref{def.yxode1}
and by using, e.g., \cite[Ch.~II, Cor.~2.1]{Hart82}, extend the flow $\varphi^t(x) , x \in O$ from $t \in [0,\infty)$ to $[-\delta, \infty)$ with some $\delta>0$ (maybe, while shrinking the open set $O$ from Asm.~\ref{ass_on_X}). We also assume that $e^{-\infty} := 0$, invoke $d_0$ and $s$ from \eqref{decopp}, put $d_0^{\leftarrow} := n-d_0+1$, and introduce the following notations:
\begin{gather}
\label{def.lll}
\Xi^{P(\cdot)}_{d,\leftarrow}(x)=\lambda_n(x)+\cdots+\lambda_{d_0^{\leftarrow}}(x)+s\lambda_{n-d_0}(x),
\quad \\
\Lambda_- := \inf_{x \in \mathbb{X}} \Xi^{P(\cdot)}_{d,\leftarrow}(x). \label{def.lll1}
\end{gather}
\begin{lemma}
For any compact set $K \subset \mathbb{X}$ of a finite Hausdorff-Riemann elliptic $d$-measure and $t \geq \tau \geq 0$, inequality \eqref{qed.inter1} holds and whenever $t-\tau \leq \delta$ in addition,
\begin{gather}
e^{\Lambda_- (t-\tau)/2}\ \pi_d^{P(\cdot)}\!\big[\varphi^\tau(K)\big] \leq \pi_d^{P(\cdot)}\!\big[\varphi^t(K)\big].
\label{qed.inter}
\end{gather}
\label{lem.twosided}
\end{lemma}
{\bf Proof:}
We introduce the $\theta$-dependent ($\theta \geq 0$) compact subset $\mathscr{K}_\theta:= \varphi^\theta(K)$ of $\mathbb{X}$ and note that
\begin{gather*}
\mathscr{K}_0 = K, \quad
\mathscr{K}_t = \varphi^{t-\tau}(\mathscr{K}_\tau) \quad  \text{if}\quad t - \tau \geq - \delta, t \geq 0, \tau \geq 0.
\end{gather*}
Let $t \geq \tau \geq 0$ be given. By using Thm.~\ref{thm:area-infty} and the relation $\mathscr{K}_t = \varphi^{t-\tau}[\varphi^\tau(K)]$, we infer that
\begin{gather}
\pi_d^{P(\cdot)}\big[\mathscr{K}_t\big] \leq \Omega_{t-\tau, \tau} \times \pi_d^{P(\cdot)} \big[\mathscr{K}_\tau\big] , \qquad \text{\rm where} \;
\label{sumilalrly}
\\
\nonumber
\Omega_{\theta,\eta}:=
\max_{x\in\mathscr{K}_\eta}\omegad\big\{S[\varphi^{\theta}(x)]\,D \varphi^{\theta}(x)\,S(x)^{-1}\big\} \\
\nonumber
\overset{\text{\eqref{def.yxode1}}}{=}
\max_{x\in\mathscr{K}\eta}\omegad\big[ Y_x(\theta) \big] \overset{\text{\eqref{chto.nado}}}{\leq}  e^{\frac{\Lambda}{2} \theta}
\Rightarrow \text{ \eqref{qed.inter1}}.
\end{gather}
\par
Suppose that $t - \tau \leq \delta$. If $\Lambda_- = - \infty$, \eqref{qed.inter} is trivially true. Let $\Lambda_- > - \infty$. Similarly to \eqref{sumilalrly}, Thm.~\ref{thm:area-infty} and the relation $\mathscr{K}_\tau = \varphi^{\tau - t}(\mathscr{K}_t)$ imply that
\begin{gather*}
\pi_d^{P(\cdot)} [\mathscr{K}_\tau ] \leq \Omega_{\tau-t, t} \times \pi_d^{P(\cdot)} \big[\mathscr{K}_t\big],
\\
\Omega_{\tau-t, t}  = \max_{x\in\mathscr{K}_t}\omegad\big[ Y_x(\tau-t) \big] \overset{\text{\eqref{def.lll}, Lem.~\ref{lem.hjdfg}--\ref{lem.eigenv}}}{\leq} \\
\max_{x\in\mathscr{K}_t} \exp\!
\left\{-\frac12\int_0^{t-\tau}\Xi^{P(\cdot)}_{d,\leftarrow}[\varphi^{- \theta}(x)]\,d\theta\right\},
\end{gather*}
where $\varphi^{-\theta}(x) \in \mathbb{X}\; \forall x \in \mathscr{K}_t, \theta \in [0,(t-\tau)]$. By \eqref{def.lll1},
\begin{gather*}
\pi_d^{P(\cdot)} [\mathscr{K}_\tau ] \leq e^{- \Lambda_- (t-\tau)} \pi_d^{P(\cdot)} \big[\mathscr{K}_t\big] \Rightarrow \text{\eqref{qed.inter}}. \qquad \qedhere
\end{gather*}
\par
{\bf Proof of Prop.~\ref{prop.fgg}:} Lem.~\ref{lem.twosided} gives \eqref{qed.inter1}.
\par
Let $\inf\limits_{x\in\mathbb{X}} \lambda_n(x) > -\infty$. Then $\Lambda_- > - \infty$ by \eqref{roots}, \eqref{def.lll}, and \eqref{def.lll1}.
We also
note that $e^z-1 \leq e^z z, e^{-z}-1 \geq - z \; \forall z \geq 0$. For any $T>0$ and two points from $[0,T]$, denoted by $\tau$ and $t$ in the ascending order, we have
\begin{gather}
\label{ineq.1}
\text{\eqref{qed.inter}}\Rightarrow\Upsilon(t) - \Upsilon(\tau) \leq \left[ e^{\Lambda(t-\tau)/2} - 1 \right] \Upsilon(\tau) \\ \leq \frac{\Lambda}{2}e^{\Lambda(t-\tau)/2} (t-\tau) \leq \frac{\Lambda}{2}e^{\Lambda T/2} (t-\tau),\nonumber
\\
\nonumber
\Upsilon(t) - \Upsilon(\tau) \geq \left[ e^{\Lambda_-(t-\tau)/2} - 1 \right] \Upsilon(\tau) \\ \geq \left[ e^{- |\Lambda_-|(t-\tau)/2} - 1 \right] \Upsilon(\tau)
\geq - |\Lambda_-|(t-\tau)/2, \nonumber
\\
\nonumber
|\Upsilon(t) - \Upsilon(\tau)| \leq |t-\tau| \max \left\{ \frac{\Lambda}{2}e^{\Lambda T/2}; |\Lambda_-| \right\}.
\end{gather}
Thus we see that the function $\Upsilon(\cdot)$ is Lipschitz and so absolutely continuous \cite[Thm.~2.9.19]{Federer}. The inequality $\dot{\Upsilon}(t) \leq \Lambda \Upsilon(t)/2$ is immediate from the inequality in \eqref{ineq.1}. \hfill \endproof
\par
{\bf Proof of Thm.~\ref{second_method_2+}: i) $\boldsymbol{\Rightarrow}$ ii)} This is immediate from Thm.~\ref{second_method}, Defn.~\ref{def.decay}, and Prop.~\ref{second_method_2+}. The converse implication {\bf ii) $\boldsymbol{\Rightarrow}$ i)} is straightforward from Defn.~\ref{def_uniform_contr} and \ref{def.decay}, along with \eqref{sadnwich2}. \hfill \endproof

\section{Two-contraction and orbital stability}
\label{sec.orbital}
The purpose of this section is twofold. First, it is shown that orbital stability of a periodic solution of an autonomous ODE can be verified via a constructive criterion similar to the conditions from ii) in Thm.~\ref{second_method} and referring to the roots of the algebraic equation \eqref{eigenvalue_problem}. More precisely, the criterion from Thm.~\ref{second_method} with the particular value $d=2$ is used, and its equivalence to the famous Andronov--Witt condition for orbital stability is established. Second, it is demonstrated that this condition is in fact a criterion for the uniform $2$-contractivity of the system on a narrow enough tubular neighborhood of the periodic orbit.
Tractability of the main result of this section is illustrated in Sec.~\ref{sec:langford}.
\par
We still consider the autonomous ODE \eqref{given_system}, use the notations \eqref{def.yxode}, and additionally assume that the ODE \eqref{given_system} has a $T$-periodic solution $x_\star(t) = x_\star(t+T)= x(t,x^o)$, where $T>0$ is the minimal period. Let $M:= X_{x^0}(T)$ stand for the monodromy matrix and let $\varrho_1, \ldots, \varrho_n$ be the Floquet multipliers, i.e., the eigenvalues of $M$ (repeated with regard to the algebraic multiplicity of each) enumerated in the non-increasing order of the modulus:
\[
|\varrho_1|\ge|\varrho_2|\ge\cdots\ge|\varrho_n|.
\]
The fact that $1$ is among these eigenvalues is well known to follow from the periodicity of the solution at hands.
The \emph{Andronov--Witt condition} consists in $|\varrho_2|<1$ and guarantees that $x_\star(\cdot)$ is asymptotically orbitally stable and
admits an asymptotic phase.

Since $T$ is the minimal period, $t \mapsto x:= x_\star(t)$ is a homeomorphism between the circle $[0,T]$ (where $T$ and $0$ are glued) and the (compact) image $\Gamma := x_\star\big( [0,T])$ (the periodic orbit).
This permits us to view any continuous function of $x \in \Gamma$ as a continuous and $T$-periodic function of $t$, and vice versa.
In particular, we treat mappings $P(\cdot)$ satisfying Asm.~\ref{kawan} with $\mathbb{X}:=\Gamma$, as well as the related roots of \eqref{eigenvalue_problem}, as $T$-periodic functions of time.

\begin{dfn}
\label{def.tubular}
A {\em regular tubular neighborhood} of the periodic orbit $\Gamma$ is any set $\mathbb{X}$ that meets Asm.~{\rm \ref{ass_on_X+}} and whose interior contains the orbit $\Gamma$.
\end{dfn}

Now we are ready to state the main result of Sec.~\ref{sec.orbital}.

\begin{thmm} \label{thm:orbital}
The following statements are equvalent:
\begin{enumerate}[{\bf i)}]
\item The Andronov--Witt condition holds for $x_\star(\cdot)$;
\item A $T$-periodic, continuous, positive definite, and piece-wise continuously differentiable function $P(\cdot) \in \br^{n \times n}$ exists such that $\lambda_1(t)+\lambda_2(t)<0  \forall t\geq 0$;
\item There is a regular tubular neighborhood of $\Gamma$ on which the system \eqref{given_system} is uniformly $2$-contractive.
\end{enumerate}
\end{thmm}
\par
This theorem is proven in Sec.~\ref{sec.orbitalproof} and generalizes some results from \cite{BLReitman} (Thm.~3.2.2, p.~117), \cite{Kuz_Reit}(Cor.~2.12, p.~72), and  \cite{Muldowny90}.
\par
Orbital stability of periodic solutions is traditionally verified via the \emph{Poincar\'e sections method}: whereas the periodic orbit corresponds to a fixed point of the Poincar\'e map, the asymptotic orbital stability follows from local contraction properties, on the Poincar\'e section and in the ordinary sense, of this map near the fixed point.
In contrast, Thm.~\ref{thm:orbital} yields a continuous-time and section-independent criterion.
We expect this theorem to be particularly useful when stabilizing periodic orbits of autonomous systems. Indeed, then building and analyzing a Poincar\'e map may be fairly cumbersome, whereas the conditions from Thm.~\ref{thm:orbital} on the one hand, are aligned in spirit with standard continuous-time design methods and on the other hand, suggest some promise for being constructive, tractable, and easily implementable. An illustration is offered in Sec.~\ref{sec:langford}.
\begin{remark}
\label{rem.constant}
Let the equivalent statements {\rm i)--iii)} in Thm.~{\rm \ref{thm:orbital}} be true. Then in {\rm ii)}, the function $P(\cdot)$ can be chosen so that the roots $\lambda_i(t)$ do not change with time.
\end{remark}
The proof of this remark is given in Sec.~\ref{sec.orbitalproof}.
\par
Finally, we note that Thm.~\ref{second_method_2+} and iii) in Thm.~{\rm \ref{thm:orbital}} imply the following fact. Under the Andronov-Witt condition, some regular tubular neighborhood of the periodic orbit can be equipped with a Riemannian metric so that under the action of the dynamic flow, the Hausdorff-Riemann elliptic $2$-measure of any compact subset of this neighborhood decays constantly, with an exponential intensity, and uniformly over such subsets.

\subsection{Proofs of Theorem~\ref{thm:orbital} and Remark~\ref{rem.constant}.}
\label{sec.orbitalproof}
These proofs are prefaced and interspersed with several technical lemmas.
\begin{lemma}
\label{lem.singbl}
For any matrix $G \in \br^{n \times n}$, there exists a similar real matrix $E$ whose singular values are arbitrarily close to the moduli of the eigenvalues $\eta_i$ of $G$: for any $\ve>0$, there is $E$ such that $|\sigma_i(E) - |\eta_{j_i}||<\ve$ for every $i$ with properly chosen permutation $\{j_i\}$ of $1,\ldots,n$.
\end{lemma}
\begin{proof}
For the sake of clarity, we first assume that $\eta_j \in \br \; \forall j$. In the basis $e_1, \ldots, e_n$ of $\br^n$ where $G$ has the normal Jordan form and under proper re-enumeration of $\eta_i$'s (if necessary), we have $G e_1 = \eta_1 e_1, G e_{i} = \eta_i e_i + \alpha_i e_{i-1} \; \forall i \geq 2$, where either $\alpha_i=0$ or $\alpha_i$=1. Now we pick $\kappa \in (0,1)$ and consider the matrix $E_\kappa$ of $G$ in the basis $\zeta_1, \ldots \zeta_n$ of $\br^n$ given by
$\zeta_i:= \kappa^i e_i$. This matrix is similar to $G$. Also, $G \zeta_1 = \eta_1 \zeta_1, G \zeta_{i} = \eta_i \zeta_i + \alpha_i \kappa \zeta_{i-1} \; \forall i \geq 2$ and so $E_\kappa = D + \Delta_\kappa$, where $D := \diag(\eta_1, \ldots, \eta_n)$, whereas all entries of $\Delta_\kappa$ are zero, maybe except for $\kappa$'s at some places above the diagonal.
Meanwhile, $E_\kappa\trn E_\kappa = \diag(|\eta_1|^2, \ldots, |\eta_n|^2) + \Delta_\kappa\trn D + D \Delta_\kappa + \Delta_\kappa\trn \Delta_\kappa \to \diag(|\eta_1|^2, \ldots, |\eta_n|^2) $ as $\kappa \to 0+$. Since the eigenvalues of a symmetric matrix vary continuously with its entries \cite[p.~109,Thm.~5.2]{Kato95}, it remains to invoke the definition of the singular values and pick $\kappa>0$ small enough.
\par
The general case $\eta_j \in \mathbb{C}$ is handled likewise, with using the real block-diagonal Jordan form of $G$.
\end{proof}
The following lemma addresses the linear ODE \eqref{linsys_At} with a $T$-periodic coefficient $A(t)$.
\begin{lemma}
Let $P(\cdot)$ be a $T$-periodic, continuous, and positive definite $n\times n$-matrix-function and $\lambda_1(t)\ge\lambda_2(t)\ge\ldots\ge\lambda_n(t) \geq 0$ be the roots of the related algebraic equation \eqref{eigenvalue_problem}, where $x:=t$. For any nonsingular matrix $Q \in \br^{n \times n}$, they are simultaneously the roots of the similar algebraic equation written for $\mathscr{A}(t) := Q A(t) Q^{-1}$ and $\mathscr{P}(t) := Q\trnm P(t) Q^{-1}$. Also, the function $\mathscr{P}(\cdot)$ is $T$-periodic, continuous, and positive definite; furthermore, it is piece-wise continuously differentiable if so is $P(\cdot)$.
\label{lem.subst}
\end{lemma}
\begin{proof}
The traits listed in the last sentence are evident. The latter algebraic equation is as follows
\begin{gather*}
0= \det\big[ \mathscr{A}(t)\trn \mathscr{P}(t)+\mathscr{P}(t) \mathscr{A}(t)+\dotm{\mathscr{P}}(t)-\lambda \mathscr{P}(t) \big]
\\
 = \det\big[ \big( Q A  Q^{-1} \big)\trn Q\trnm P Q^{-1} + Q\trnm P Q^{-1} \big( Q A  Q^{-1} \big) \\ +  Q\trnm \dot{P} Q^{-1} -\lambda  Q\trnm P Q^{-1} \big]
 \\
 =  \det\big[ Q\trnm A\trn P Q^{-1} + Q\trnm P A  Q^{-1}  +  Q\trnm \dot{P} Q^{-1} \\ -\lambda  Q\trnm P Q^{-1} \big]
 \\
 =  \det Q\trnm \det\big[  A\trn P  +  P A   +   \dot{P}  -\lambda   P  \big] \det Q^{-1}.
\end{gather*}
Thus the two algebraic equations mentioned in the lemma are trivially equivalent, which completes the proof.
\end{proof}
\begin{lemma}
\label{impl.1.lem}
The implication ${\bf i)}\Rightarrow {\bf ii)}$ is valid.
\end{lemma}
\begin{proof}
In the proof, the function $A(t)$ from \eqref{def.yxode} can be replaced by any similar matrix-function $Q A(t) Q^{-1}$ due to Lem.~\ref{lem.subst}.
This transforms the linearized flow $X_{x^0}(t)$ from \eqref{def.yxode} into the conjugate flow $Q X_{x^0}(t) Q^{-1}$, and the monodromy matrix $M = X_{x^0}(T)$
into the similar matrix $\mathscr{M} := Q M Q^{-1}$. Except for $1$, the spectrum of $M$ lies in the ball $B_0^{\rho} \subset \mathbb{C}$ (where $\rho:=|\varrho_2|<1$) by the Andronov-Witt condition. We pick so small $\kappa >0$ that $(1+\kappa)(\varrho + \kappa) <1$. Then
by using Lem.~\ref{lem.singbl}, we choose $Q$ so that $|\sigma_1(\mathscr{M}) - 1 | < \kappa $ and $\sigma_i(\mathscr{M}) < \varrho + \kappa\; \forall i \geq 2$ and so
$
\sigma_1(\mathscr{M}) \sigma_i(\mathscr{M}) < (1+\kappa)(\varrho + \kappa) <1$ for such $i$'s.
Summing up and reverting to the initial notations, we conclude that without any loss of generality, the proof can be carried out with assuming that $\sigma_1(M) \sigma_2(M) <1$.
\par
Let $\Ln \, R$ stand for the principal logarithm of the positive definite matrix $R = R\trn \in \br^{n \times n}$. We put
\begin{gather*}
\Xi := T^{-1} \Ln \, ( M\trn  M), \quad P(t)=  X_{x^0}(t)\trnm e^{\Xi t}  X_{x^0}(t)^{-1}.
\end{gather*}
The eigenvalues $\eta_1 \geq \eta_2 \geq \ldots \geq \eta_n $ of $\Xi$ equal $\eta_i = \frac{2}{T} \ln \sigma_i(M)$ and so
\begin{equation}
\label{key.inneqq}
\eta_1+ \eta_2 <0.
\end{equation}
The function $P(\cdot)$ is $C^1$-smooth, $P(0) = I_n$, and
\begin{gather}
\nonumber
P(T) =  X_{x^0}(T)\trnm e^{\Xi T}  X_{x^0}(T)^{-1} \\
= M\trnm ( M\trn  M) M^{-1} = I_n = P(0),
\label{p.per}
\\
\nonumber
\dot{P}(t) =  - X_{x^0}(t)\trnm e^{\Xi t}  X_{x^0}(t)^{-1} \dot{X} _{x^0}(t)  X_{x^0}(t)^{-1} \\
\nonumber
- X_{x^0}(t)\trnm \dot{X} _{x^0}(t)\trn  X_{x^0}(t)\trnm  e^{\Xi t}  X_{x^0}(t)^{-1}
\\
\nonumber
+ X_{x^0}(t)\trnm e^{\Xi t} \Xi X_{x^0}(t)^{-1}
\\
\nonumber
= -  X_{x^0}(t)\trnm e^{\Xi t}  X_{x^0}(t)^{-1} A(t)X _{x^0}(t)  X_{x^0}(t)^{-1} \nonumber\\ - X_{x^0}(t)\trnm X _{x^0}(t)\trn A(t)\trn X_{x^0}(t)\trnm  e^{\Xi t}  X_{x^0}(t)^{-1}
\\
\nonumber
+ X_{x^0}(t)\trnm e^{\Xi t/2} \Xi e^{\Xi t/2} X_{x^0}(t)^{-1}
\\
= -  P(t) A(t) - A(t)\trn P(t) + Z(t)\trnm \Xi Z(t)^{-1}, \nonumber \\ \quad \text{where} \; Z(t):=    X_{x^0}(t) e^{-\Xi t/2}
.
\nonumber
\end{gather}
\par
We proceed by noting that $P(t):=Z(t)\trnm\,Z(t)^{-1}$. Hence
\begin{multline}
\det \big[ P(t) A(t) + A(t)\trn P(t) + \dot{P}(t) - \lambda P(t) \big] \\ = \det \big[ Z(t)\trnm \Xi Z(t)^{-1} - \lambda Z(t)\trnm\,Z(t)^{-1} \big]
\\
= \det   Z(t)\trnm  \det \big[\Xi - \lambda I_n \big] \det Z(t)^{-1} .
\label{root.const}
\end{multline}
The proof is completed by invoking \eqref{key.inneqq} and extending $P(\cdot)$ from $[0,T]$ to $\br$ as a continuous $T$-periodic function, which is possible due to
\eqref{p.per}.
\end{proof}
\par
{\bf Proof of Remark~\ref{rem.constant}:} This remark is immediate from \eqref{root.const}, which shows that the roots $\lambda_i(t)$ are simultaneously the eigenvalues of the time-invariant matrix $\Xi$.
\begin{lemma}
\label{impl.2.lem}
The implication ${\bf ii)}\Rightarrow {\bf i)}$ is valid.
\end{lemma}
\par
{\bf Proof:}
By applying Lem.~\ref{lem.kkeeyy} with $d=2$ and $\mathbb{X} := x_\star([0,T])$ (and noting that the requirements to the interior of $\mathbb{X}$ from Asm.~\ref{ass_on_X+} and the continuity of the derivative $\dot{P}(\cdot)$ are not utilized in the proof of Lem.~\ref{lem.kkeeyy}), we infer that $\omega_2[Y_x(t)] \leq e^{\frac{\Lambda}{2} t} \; \forall t \geq 0$, where $\Lambda := \max_{x \in \mathbb{X}} [\lambda_1(x) + \lambda_2(x)] = \max_{t \in [0,T]} [\lambda_1(t) + \lambda_2(t)] <0$. Hence $\omega_2[Y_{x^0}(t)] \xrightarrow{t \to \infty} 0$ and
\begin{gather*}
\omega_2[X_{x^0}(t)]
\overset{\text{\eqref{def.yxode1}}}{=} \omega_2 \Big\{ S\!\big[\varphi^t(x^0)\big]^{-1}\,Y_{x^0}(t)\,S(x^0) \Big\} \\
\overset{\text{\eqref{eq:submult}}}{\leq} \omega_2[Y_{x^0}(t)] \times \sup_{t \geq 0} \omega_2 \Big\{ S\!\big[\varphi^t(x^0)\big]^{-1} \Big\} \times \omega_2 \Big\{ S(x^0) \Big\} .
\end{gather*}
Here
$$
\sup_{t \geq 0} \omega_2 \Big\{ S\!\big[\varphi^t(x^0)\big]^{-1} \Big\} =\max_{t \in [0,T]} \omega_2 \Big\{ S\!\big[\varphi^t(x^0)\big]^{-1} \Big\} <0
$$
due to the periodicity of the function $t \mapsto \varphi^t(x^0)$ and so $\omega_2[X_{x^0}(t)] \xrightarrow{t \to \infty} 0$. For the monodromy matrix $M:= X_{x^0}(T)$, this yields that $\omega_2(M^k) \to 0$ as $k \to \infty$. Meanwhile, the moduli of the eigenvalues of $M^k$ are given by $|\varrho_1|^k=1 \geq |\varrho_2|^k \geq \ldots \geq \|\varrho_n|^k$. So by the Weyl's inequality \cite[ p.~171]{Horn},
$$
 |\varrho_2|^k = |\varrho_1|^k  |\varrho_2|^k \leq \omega_2(M^k) \xrightarrow{k \to \infty}0 \Rightarrow |\varrho_2| <1. \qquad \qedhere
$$
\par
Since the minimal period $T>0$, the derivative $\dot{x}_\star(t) \neq 0\; \forall t$ and so $t \mapsto x_\ast(t)$ is a parametric representation of a regular $C^2$-smooth Jordan loop $\Gamma:=x_\star([0,T])$ (the periodic orbit). We treat $t$ as a cyclic coordinate of a point $x \in \Gamma$.
Due to the regularity of $\Gamma$, there exists $\gamma^\dagger>0$ such that for $x \in \overset{\circ}{N}\vphantom{N}^\Gamma_{\gamma^\dagger} := \{x: \dist{x}{\Gamma} < \gamma^\dagger\}$, the minimum distance point $q(x) \in \Gamma, \|q(x) - x\| = \dist{x}{\Gamma}$ is unique and the mapping $x \mapsto q(x)$ is $C^1$-smooth \cite{Sternb83,Thorpe79}.
\begin{lemma}
If $\gamma^\dagger >0$ is small enough, the mapping $x \in \overset{\circ}{N}\vphantom{N}^\Gamma_{\gamma^\dagger} \mapsto \mathscr{T}(x):=t[q(x)]$ is $C^1$-smooth and
\begin{equation}
\label{denom}
 D \mathscr{T}(x) v =  \frac{\spr{v}{f(y)}}{\|f(y)\|_2^2  + \spr{ y - x}{Df(y) f(y) } },
\end{equation}
 where $y:= q(x)$ and the denominator is nonzero.
\end{lemma}
\begin{proof}
The $C^1$-smoothness of the mapping $x \in \overset{\circ}{N}\vphantom{N}^\Gamma_{\gamma^\dagger} \mapsto q(x)$ for small enough $\gamma^\dagger >0$ is a classic fact; see, e.g., \cite{Thorpe79,Sternb83,MerrRuu07}. This implies the smoothness of $\mathscr{T}(\cdot)$.
\par
Since $\min_t\|x(t,x^0)-x\|^2$ is attained at $t= \mathscr{T}(x)$, the Fermat's theorem yields that
\begin{gather*}
0 = \frac{d}{dt} \| x(t,x^0) -x\|^2|_{t=\mathscr{T}(x)} =
2 \spr{q(x)-x}{f[q(x)]}
\\
 = \spr{ x \big[ \mathscr{T}(x) , x^0\big]-x}{f\big\{ x \big[ \mathscr{T}(x) , x^0\big] \big\}} \qquad \forall x.
\end{gather*}
By differentiating, we see that for any $v \in \br^n$,
\begin{gather*}
0 = \spr{ \dot{x} \big[ \mathscr{T}(x) , x^0\big] }{f\big\{ x \big[ \mathscr{T}(x) , x^0\big] \big\}} D \mathscr{T}(x) v \\
- \spr{v}{f\big\{ x \big[ \mathscr{T}(x) , x^0\big] \big\}}
\\
+
\spr{ x \big[ \mathscr{T}(x) , x^0\big]-x}{Df\big\{ x \big[ \mathscr{T}(x) , x^0\big] \big\}\dot{x} \big[ \mathscr{T}(x) , x^0\big] } D \mathscr{T}(x) v
\\
=
\Big\{\|f[q(x)]\|_2^2  + \spr{ q(x) - x}{Df[q(x)] f[q(x)] } \Big\} D \mathscr{T}(x) v \\ - \spr{v}{f[q(x)]]} .
\end{gather*}
To complete the proof, it remains to note that
\begin{gather*}
\inf_{x \in \overset{\circ}{N}\vphantom{N}^\Gamma_{\gamma^\dagger}} \|f[q(x)]\|_2^2 = \min_{x \in \Gamma}\|f[x]\|_2^2 >0,
\\
\sup_{x \in \overset{\circ}{N}\vphantom{N}^\Gamma_{\gamma^\dagger}}
\Big\|\spr{ q(x) - x}{Df[q(x)] f[q(x)] } \Big\|
\\
\leq \gamma^\dagger \max_{x \in \Gamma }\| Df(x) f(x) \| < \infty
\end{gather*}
and so the denominator in \eqref{denom} is nonzero if $\gamma^\dagger \approx 0$.
\end{proof}
Along with Asm.~\ref{kawan}, we shall consider its relaxation that comes to replacement of the requirement of orbital continuity of the orbital derivative by the weaker property of orbital piece-wise continuity of this derivative. It is easy to see by inspection of the proofs that Thm.~\ref{second_method} remains valid under this relaxation of Asm.~\ref{kawan}.
\begin{lemma}
\label{lem.extend}
Let $\mathscr{P}(\cdot)$ be a $T$-periodic, continuous, positive definite, and piece-wise continuously differentiable $n\times n$-matrix-function. For small enough $\gamma\in (0,\gamma^\dagger)$, the function $x \in \overset{\circ}{N}\vphantom{N}^\Gamma_\gamma \mapsto P(x):= \mathscr{P}[\mathscr{T}(x)]$ meets the relaxed Asm.~{\rm \ref{kawan}} and, for all $x \in \overset{\circ}{N}\vphantom{N}^\Gamma_\gamma$, its orbital derivative $\dot{P}(x)$ is as close to  $\dot{\mathscr{P}}[\mathscr{T}(x)]$ as desired.
\end{lemma}
\begin{proof}
Let $\gamma \to 0+$. Then uniformly over $\xi \in \overset{\circ}{N}\vphantom{N}^\Gamma_\gamma $, first, $\|x(t,\xi) - q[x(t,\xi)]\| \to 0$ and second, for $t \approx 0$,
\begin{gather}
\frac{d \mathscr{T}[x(t,\xi)]}{dt} = D \mathscr{T}[x(t,\xi)] \dot{x} (t,\xi) \nonumber \\
\overset{\text{\eqref{denom}}}{=}
\frac{\spr{f[x(t,\xi)] }{f[q(x(t,\xi))]}}{\|f[q(x(t,\xi))]\|_2^2  + {\mathcal R} }
\to 1, \quad \text{where}
\label{converge}\\
{\mathcal R}:=\spr{ q(x(t,\xi)) - x(t,\xi)}{Df[q(x(t,\xi))] f[q(x(t,\xi))] }.\nonumber
\end{gather}
If $\gamma>0$ is small,
the function $t \mapsto \mathscr{T}[x(t,\xi)]$ is thereby, monotone.
Hence the function $t \mapsto P[x(t,\xi)] = \mathscr{P}\big\{\mathscr{T} [x(t,\xi)] \big\}$ is piece-wise continuously differentiable; so it satisfies the relaxed Asm.~{\rm \ref{kawan}}. Furthermore,
\begin{gather*}
\dot{P}(\xi) = \frac{d}{dt} P[x(t,\xi)]|_{t=0} =  \left. \frac{d}{dt} \mathscr{P}\big\{\mathscr{T} [x(t,\xi)] \big\}\right|_{t=0}
\\
= \dot{\mathscr{P}}[\mathscr{T}(\xi)] \left. \frac{d}{dt}\frac{d \mathscr{T}[x(t,\xi)]}{dt}\right|_{t=0}
\\ =  \dot{\mathscr{P}}[\mathscr{T}(\xi)]
 \frac{\spr{f(\xi)}{f[q(\xi)]}}{\|f(\xi)\|_2^2  + \spr{ q(\xi) - \xi}{Df(\xi) f(\xi) } } .
\end{gather*}
The proof is completed by invoking \eqref{converge}.
\end{proof}
\par
{\bf Proof of Theorem~\ref{thm:orbital}:} The equivalence {\bf i)} $\Leftrightarrow$ {\bf ii)} is established by Lem.~\ref{impl.1.lem} and \ref{impl.2.lem}.
\par
{\bf iii)} $\boldsymbol{\Rightarrow}$ {\bf ii):} This is established by applying the part i) $\Rightarrow$ ii) of Thm.~\ref{second_method} to the considered tubular neighborhood of the periodic solution $x_\star(\cdot)$ and subsequently forming the superposition of the matrix-function $P(\cdot)$ from ii) in Thm.~\ref{second_method} with $x_\star(\cdot)$.
\par
{\bf i)} $\boldsymbol{\Rightarrow}$ {\bf iii):} By the foregoing, the statement ii) is also true. Let $\mathscr{P}(\cdot)$ be a matrix function that possesses the properties from ii). By applying Lem.~\ref{lem.extend} to $\mathscr{P}(\cdot)$ , we introduce the function $P(x)= \mathscr{P}[\mathscr{T}(x)]$ of $x \in \overset{\circ}{N}\vphantom{N}^\Gamma_\gamma$ and $\gamma_\ast \in (0,\gamma^\dagger)$ such that $P(\cdot)$ satisfies the relaxed Asm.~{\rm \ref{kawan}} for $\gamma:=\gamma_\ast$.
The last statement of Lem.~\ref{lem.extend} implies that uniformly over $x \in \overset{\circ}{N}\vphantom{N}^\Gamma_{\gamma_\ast}$, the coefficients of the algebraic equation \eqref{eigenvalue_problem} related to $P(x)$ can be made arbitrarily close to the eponymous coefficients of the similar equation related to $\mathscr{P}[\mathscr{T}(x)]$ if $\gamma_\ast>0$ is properly reduced. Since the roots vary continuously with the coefficients (see, e.g., \cite{HarMar87}), the eponymous roots can also be made as close as desired. By doing so, the condition $\lambda_1(x)+\lambda_2(x) <0 \; \forall x \in x \in \overset{\circ}{N}\vphantom{N}^\Gamma_{\gamma_\ast}$ can be ensured by inheritance from the similar condition on $\mathscr{P}(\cdot)$ from ii).
\par
By the Andronov--Witt theorem, $x_\star(\cdot)$ is asymptotically orbitally stable. Hence there is $\ve>0$ such that
\begin{gather}
\nonumber
x(t,\xi) \in \overset{\circ}{N}\vphantom{N}^\Gamma_{\gamma_\ast} \; \forall t \geq 0, \xi \in N ^\Gamma_{\ve} := \{x: \dist{x}{\Gamma} \leq \ve\}
~ \text{and}
\\
\label{to.infty}
\max_{\xi \in N^\Gamma_{\ve}} \dist{x(t,\xi)}{\Gamma} \to 0 ~ \text{as} ~ t \to \infty.
\end{gather}
Our next step is to show that the set $\mathbb{X}_\star := \{y = x(t,\xi): t \geq 0, \xi \in  N ^\Gamma_{\ve}\}$ is a regular tubular neighborhood of the periodic orbit $\Gamma$ in the sense of Defn.~\ref{def.tubular}.
\par
To demonstrate that $\mathbb{X}_\star $ is compact, we consider an arbitrary sequence $\{y_i\}_{i=1}^\infty \subset \mathbb{X}_\star$. It is associated with two sequences $\{t_i\}_{i=1}^\infty \subset [0,\infty)$ and $\{\xi_i\}_{i=1}^\infty \in  N ^\Gamma_{\ve}$ such that $y_i = x(t_i,\xi_i) \; \forall i$. If the first of them is not bounded, there exists a subsequence $\{t_{i_k}\}_{k=1}^\infty$ converging to $+\infty$.
Then $\dist{y_{i_k}}{\Gamma} = \dist{x(t_{i_k},\xi_{i_k})}{\Gamma} \to 0$ as $k \to \infty$ by \eqref{to.infty} and so there exists a third sequence $\{z_{k}\}_{k=1}^\infty \subset \Gamma$ such that $\|y_{i_k}-z_k\|_2 \to 0$ as $k \to \infty$. Since the set $\Gamma$ is compact, there is a converging subsequence of $\{z_{k}\}_{k=1}^\infty$. The matching subsequence of $\{y_{i}\} \subset \mathbb{X}_\star$ also converges.
\par
If conversely the sequence $\{t_{i}\}$ is bounded, then there exist matching subsequences $\{t_{i_k}\}_{k=1}^\infty$ and $\{\xi_{i_k}\}_{k=1}^\infty$ that converge to finite limits $t_\infty \in [0,\infty)$ and $\xi_\infty \in N^\Gamma_{\ve}$, wherein the second one exists since the set $N^\Gamma_{\ve}$ is compact. Then $y_{i_k} = x[t_{i_k}, \xi_{i_k}] \to x(t_\infty,\xi_\infty) \in \mathbb{X}_\star$ as $k \to \infty$ \cite[Ch.~V, Thm.~2.1]{Hart82}. Thus any sequence $\{y_i\}_{i=1}^\infty \subset \mathbb{X}_\star$ has a subsequence that converges to some element of $\mathbb{X}_\star$, which means that $\mathbb{X}_\star$ is compact.
\par
Finally, the set $\overset{\circ}{N}\vphantom{N}^\Gamma_{\ve}$ is open and its closure equals $N^\Gamma_{\ve}$.
For any $t \geq 0$, the mapping $\xi \mapsto x(t,\xi)$ is a diffeomorphism. So the set $\overset{\circ}{\mathbb{X}}\vphantom{\mathbb{X}}_\star := \{y = x(t,\xi): t \geq 0, \xi \in \overset{\circ}{N}\vphantom{N}^\Gamma_{\ve}\}$ is open, covers $\Gamma$, and its closure contains $\mathbb{X}_\star$. Since $\mathbb{X}_\star$ is positively invariant by construction, it is a regular tubular neighborhood of $\Gamma$ by Defn.~\ref{def.tubular}. Thm.~\ref{second_method} and the remark preceding Lem.~\ref{lem.extend} complete the proof of iii). \qquad \endproof

\section{Illustrative examples}
\label{illustr.examples}
This section aims to illustrate the effectiveness of the approach developed above.
\subsection{ Motion of a rigid body in a dissipative environment}
\label{Tennis_bat}
We start with a system described by the classical Euler's rotation equations:
\begin{equation}
\begin{array}{l}
J_1\dot\omega_1 = (J_2-J_3)\omega_2\omega_3 - J_1\delta\omega_1, \\
J_2 \dot \omega_2 = (J_3-J_1)\omega_1\omega_3 - J_2\delta\omega_2 + \tau,
\\
J_3\dot\omega_3 = (J_1-J_2)\omega_1\omega_2 -J_3\delta\omega_3 ,
\end{array}\label{Jun_sys}
\end{equation}
where $ 0<J_1<J_2<J_3 $
are the moments of inertia. Also, $\tau$ is an external torque and $\delta>0$ is the coefficient of friction.
The system (\ref{Jun_sys}) describes the rotational motion of a rigid body in the body-fixed frame composed by the axes of inertia;
$\omega_1,\omega_2,\omega_3$ are the respective angular velocities. The body is actuated by the torque $\tau$ applied around the second principal axis.
\par
The system \eqref{Jun_sys} has attracted much attention not only due to its fundamentality but also for related interesting effects.
One of them is addressed by the tennis racket theorem, also known as the intermediate axis theorem. It states that if
there is neither an outer actuation $\tau=0$ nor energy dissipation $\delta=0$, rotation around both the first and third axes is stable, whereas that around the intermediate axis is unstable.
An experimental manifestation of this phenomenon is called the Dzhanibekov effect, after cosmonaut Vladimir Dzhanibekov, who surprisingly for himself noticed the outcome of this theorem whilst in space in 1985.
\par
Application of Thm.~\ref{second_method} to the system \eqref{Jun_sys} entails the following main result of this section.
\begin{thmm}
\label{thm.jan}
The system \eqref{Jun_sys} with a constant torque $\tau$ is uniformly 2-contractive on any compact positively invariant subset $\mathbb{X}$ of its phase space if and only if
\begin{equation}
\label{result_tau}
|\tau|<2\delta^2J_2\sqrt{\frac{J_1J_3}{(J_3-J_2)(J_2-J_1)}}.
\end{equation}
This inequality necessarily holds if the system \eqref{Jun_sys} is uniformly 2-contractive on at least one compact set $\mathbb{X}$ that meets Asm.~{\rm \ref{ass_on_X+}} and contains the equilibrium $\boldsymbol{\omega}_{\text{\rm eq}}=\left(0,\tfrac{\tau}{J_2\delta},0\right)$ in its interior.
\end{thmm}
\begin{proof}
{\bf Sufficiency:} Let \eqref{result_tau} hold. We start with differentiating the quadratic energy-like function
\(
W(\boldsymbol{\omega})=\tfrac12\big(J_1\omega_1^2+J_2\omega_2^2+J_3\omega_3^2\big)
\)
of $\boldsymbol{\omega}:=(\omega_1,\omega_2,\omega_3)$
along the trajectories of \eqref{Jun_sys}
\begin{equation}
\label{dotmm}
\dotm{W} = -J_1\delta\omega_1^2 -J_2\delta\omega_2^2-J_3\delta\omega_3^2+\tau\omega_2.
\end{equation}
An elementary upper estimate of the r.h.s. shows that \(
\dotm{W} \le -\delta W+\beta_\star
\),
where $\beta_\star={\tau^2}/{(2\delta J_2)}$. Hence for any $\beta > \beta_\star$, all the solutions eventually enter the ellipsoid
\begin{equation}
\label{ellip.ss}
\mathbb{X}_\beta:=\{\boldsymbol{\omega}~|~W(\boldsymbol{\omega})\le \beta/\alpha\}
\end{equation}
and remain there afterwards. Also,
Asm.~\ref{ass_on_X} and \ref{ass_on_X+} hold true with $\mathbb{X} := \mathbb{X}_\beta$.
\par
Our next step is based on Thm.~\ref{second_method}. To use it, we first find the Jacobi matrix of the right-hand side in \eqref{Jun_sys}:
\begin{equation}
\label{jac.rhs}
\mathscr{A}(\boldsymbol{\omega})=
\begin{bmatrix}
-\delta & \tfrac{(J_2-J_3)}{J_1}\omega_3 & \tfrac{(J_2-J_3)}{J_1}\omega_2 \\[6pt]
\tfrac{(J_3-J_1)}{J_2}\omega_3 & -\delta & \tfrac{(J_3-J_1)}{J_2}\omega_1 \\[6pt]
\tfrac{(J_1-J_2)}{J_3}\omega_2 & \tfrac{(J_1-J_2)}{J_3}\omega_1 & -\delta
\end{bmatrix}
\end{equation}
and put $P_0:=\diag\left[ 1,~~\tfrac{J_2}{J_1}\tfrac{(J_3-J_2)}{(J_3-J_1)},~~\tfrac{J_3}{J_1}\tfrac{(J_3-J_2)}{(J_2-J_1)} \right]$. This matrix is positive definite since $ 0<J_1<J_2<J_3 $.
Straightforward calculations yield that
\[
\mathscr{A}\trn P_0+P_0\mathscr{A}=
2\left[
\begin{array}{ccc}
-\delta & 0 & \frac{J_2-J_3}{J_1}\omega_2 \\
0 & -\frac{\delta J_2(J_3-J_2)}{J_1(J_3-J_1)} & 0\\
 \frac{J_2-J_3}{J_1}\omega_2 & 0 & - \frac{\delta J_3(J_3-J_2)}{J_1(J_2-J_1)}
\end{array}
\right].
\]
For any $\chi \in \mathbb{C}$,
the determinant $\Delta(\chi):= \det (P_0\mathscr{A}+\mathscr{A}^\top P_0-\chi P_0)$ reduces to
\(
(\chi+2\delta)\bigl[(\chi+2\delta)^2-4\omega_2^2\varrho^2\bigr]
\)
up to a constant factor, where
\[
\varrho := \sqrt{\tfrac{(J_3-J_2)(J_2-J_1)}{J_1J_3}} .
\]
The roots of the equation $\Delta(\chi)=0$ are as follows
\(
\chi_{1,3} = -2\delta \pm 2|\omega_2|\varrho,
\chi_2 = -2\delta.
\)

Now we introduce a state-dependent metric given by the following positive definite matrix
\(
P(\boldsymbol{\omega})=P_0e^{\gamma W(\boldsymbol{\omega})},
\)
where $\gamma >0$ is a parameter to be specified later on.
The previous calculations readily imply that the solutions to $\det(\mathscr{A}\trn P+P\mathscr{A}+\dotm{P}-\lambda P)=0$
are given by
\(
\lambda_i=\chi_i+\gamma \dotm{W}, i=1,2,3.
\)
Hence,
\(
\lambda_1+\lambda_2 = 2\bigl(-2\delta + |\omega_2|\varrho + \gamma\dotm{W}\bigr).
\)
\par
By Thm.~\ref{second_method} and \eqref{dotmm}, the system \eqref{Jun_sys} is uniformly $2$-contractive on $\mathbb{X}_\beta$
whenever there is $\gamma >0$ such that
\(
-\gamma J_2 \delta \omega_2^2 +\omega_2(\varrho{\rm sign}\omega_2+\gamma \tau )-2\delta<0
\)
for arbitrary $\omega_2$.
In turns, this holds if the discriminant of the left hand side is negative, i.e.,
\(
(\varrho{\rm sign}\,\omega_2+\gamma\tau)^2-8\delta^2\gamma J_2<0
\),
or equivalently,
\(
\gamma^2 \tau^2 + 2\gamma\bigl(\tau\varrho \,\mathrm{sign}\,\omega_2 - 4\delta^2J_2\bigr) + \varrho^2 < 0.
\)
This is a quadratic in $\gamma$ with nonnegative leading coefficient ($\tau^2\ge 0$).
Hence a feasible $\gamma>0$ exists if and only if:
1) the linear coefficient is strictly negative, i.e.
\(
\tau\varrho \,\mathrm{sign}\,\omega_2 - 4\delta^2J_2 < 0,
\)
and
2) the discriminant is positive, i.e.
\(
(\tau\varrho \,\mathrm{sign}\,\omega_2 - 4\delta^2J_2)^2 - \tau^2\varrho^2 > 0.
\)
The first condition is equivalent to
\(
|\tau| < {4\delta^2 J_2}/{\varrho},
\)
whereas the second condition comes to the same bound, but with a factor $1/2$.
Combining the both yields \eqref{result_tau}.
\par
So the system \eqref{Jun_sys} is uniformly 2-contractive on the ellipsoid
\eqref{ellip.ss} by Thm.~\ref{second_method}. It remains to note that any compact set $\mathbb{X}$ can be covered by the ellipsoid \eqref{ellip.ss} with a large enough $\beta > \beta_\star$ and so this set inherits uniform 2-contractivity from $\mathbb{X}_\beta$ by Defn.~\ref{def.contr}.
\par
\noindent\textbf{Necessity.}
It suffices to prove the statement given by the second sentence in Thm.~\ref{thm.jan}. Suppose to the contrary that \eqref{result_tau} is untrue but
the system \eqref{Jun_sys} is uniformly 2-contractive on a compact set $\mathbb{X}$ satisfying Asm.~{\rm \ref{ass_on_X+}} and containing the equilibrium $\boldsymbol{\omega}_{\text{\rm eq}}=\left(0,\tfrac{\tau}{J_2\delta},0\right)$ in its interior.
Since \eqref{result_tau} is untrue, $u:= \frac{|\tau|}{\tau_c} \geq 1$, where
$\tau_c:=2J_2\delta^2/\varrho$.
Linearizing \eqref{Jun_sys} at the equilibrium $\boldsymbol{\omega}_{\text{\rm eq}}$ gives rise to the characteristic polynomial
\(
(\lambda+\delta)\bigl[(\lambda+\delta)^2-(2u\delta)^2\bigr]
\).
Its roots are
\(
\lambda_1=\delta(2u-1),\quad \lambda_2=-\delta,\quad \lambda_3=-\delta(1+2u)
\)
and so
\(
\lambda_1+\lambda_2=2\delta(u-1)>0.
\)
The Jacobi matrix of the right hand side $\mathscr{A}(\boldsymbol{\omega}_{\text{\rm eq}})$ can be made diagonal via a similarity transformation since the eigenvalues are real and different. Let $\Lambda_{\boldsymbol{\omega}_{\text{\rm eq}}}$ denote the diagonal matrix of the eigenvalues
of $\mathscr{A}(\boldsymbol{\omega}_{\text{\rm eq}})$, and $S={\rm const}, {\rm det}S\ne 0$ the similarity transformation bringing $\mathscr{A}(\boldsymbol{\omega}_{\text{\rm eq}})$ to  $\Lambda_{\boldsymbol{\omega}_{\text{\rm eq}}}$. Then Lemma \ref{lem.cocycle} implies that
\begin{gather*}
{\bf \Sigma}_2\stackrel{\rm Lem. \ref{lem.cocycle}}{\ge}
\limsup_{t\rightarrow\infty}\Sigma_2(t,\boldsymbol{\omega}_{\text{\rm eq}}) \\ \stackrel{\eqref{def.sigd}}=\varlimsup_{t\rightarrow\infty}\frac1t\ln\omega_2\bigl(Se^{\Lambda_{\boldsymbol{\omega}_{\text{\rm eq}}} t}S^{-1}\bigr)  \overset{S={\rm const}, \det S\ne 0}{=\!=\!=\!=\!=\!=\!=}\lambda_1+\lambda_2>0.
\end{gather*}
So the system is not uniformly 2-contractive by Thm. \ref{firstmethod}. This contradiction proves the necessity.
\end{proof}

\begin{remark}
Since the Jacobian matrix \eqref{jac.rhs} does not explicitly depend on $\tau$, the method from \cite{Kanevski} is not directly applicable to the problem considered in this section.
\end{remark}

\subsection{R\"ossler system}
\label{ross.sec}
The next example illustrates numerical computation of a bound for uniform $d$-contraction, where $d$ is not necessarily integer.

The following system with parameters $a,b>0$ is traditionally attributed to R\"ossler \cite{KMV}:
\begin{equation}
\dot x=-y-z,\qquad \dot y=x,\qquad
\dot z=-bz+a(y-y^2) .
\label{rossler1}
\end{equation}
It is among a number of others, also attributed to R\"ossler; most of them exhibit a chaotic attractor of the Shilnikov type. The authors are unaware of a formal proof of this fact for the system \eqref{rossler1}; for some other R\"ossler systems, this fact was established via semi-analytical methods, see e.g., \cite{DING2023770,Shiln}. The main argument in favor that the system \eqref{rossler1}
possesses a Shilnikov-type chaotic attractor spiraling around the origin commonly comes to establishing the existence of a homoclinic orbit of the zero equilibrium. This can be done via approximating the homoclinic solution with a series expansion, see, e.g., \cite{DING2023770}.

We are going to illustrate the procedure of estimating the values of $d$ for which the system is uniformly $d$-contractive. We start by noting that the system has two equilibria $E_0=(0,0,0)\trn$ and $E_1=(0,e, -e)\trn$, where $e:=1+b/a$. In \eqref{rossler1},
the Jacobian of the right-hand side depends only on $y$ and is as follows:
$$
A(y)=\begin{bmatrix}
0 & -1 & -1\\[2pt]
1 & 0  & 0\\[2pt]
0 & a-2ay & -b
\end{bmatrix} .
$$
Hence the characteristic equation of the matrix $A(0)$ is
\begin{equation}
\label{int.eq}
\lambda^3+b\lambda^2+\lambda+(a+b)=0.
\end{equation}
It follows that the matrix $A(0)$ satisfies the Shilnikov's saddle-focus condition:
it has a negative eigenvalue $-\gamma$, where $\gamma>0$, and a pair of complex conjugate ones $\sigma\pm \imath \delta$, where $\sigma>0$ and $\delta \in \br$.
Indeed, since $a,b>0$, any real root of \eqref{int.eq} is negative, whereas at least one real root does exist due to the odd degree of the equation. By the Vieta's formula related to the constant term, the product of two other roots is positive and so if they were real, they would be of the same sign. However, since the polynomial in the left hand side is not Hurwitz by the Vyshnegradsky's criterion, these roots are positive, which is impossible by the foregoing. Hence the other two roots form a pair $\sigma\pm \imath \delta$ of conjugate ones with $\sigma>0$. So long as $- \gamma + 2 \sigma = - b <0$ by the Vieta's formula related to the coefficient before $\lambda^2$, we infer that $\gamma>2\sigma$.

Hence, the combination \eqref{def.sigd} of the characteristic exponents along the equilibrium solution $E_0$ is positive for $d \leq 2$. For $d >2$,
\begin{multline*}
\Sigma_d(t,E_0) =2\sigma+s(-\gamma)= \underbrace{2\sigma-\gamma}_{=\tr(A(0))=-b} +\gamma-s\gamma \\ =-b+(1-s)\gamma
\end{multline*}
and so $\Sigma_d(t,E_0) <0 \Leftrightarrow d> 3-b/\gamma$.
With regard to \eqref{three,eq} and Thm.~\ref{firstmethod}, we thus infer that $d_L:=3-b/\gamma$ is a lower bound on $d$'s for which the system \eqref{rossler1} can be uniformly $d$-contractive on a positively invariant compact set $\mathbb{X}$ satisfying Asm.~\ref{ass_on_X+}.
\par
G.A.~Leonov conjectured \cite{KMV} that the Lyapunov dimension of the R\"ossler attractor of the system~\eqref{rossler1} is exactly \(3 - b/\gamma\). {An available idea of proving  this conjecture  basically comes to establishing the following two statements:}
\begin{enumerate}[{\bf i)}]
    \item The saddle-focus equilibrium \(E_0\) possesses a homoclinic trajectory.
    \item For any set \(\mathbb{X}\) satisfying Asm.~\ref{ass_on_X+} and containing \(E_0\), the system~\eqref{rossler1} is uniformly \(d\)-contractive for all \(d>3-b/\gamma\).
\end{enumerate}
Specifically, i) implies that there is an invariant set that meets  Asm.~\ref{ass_on_X+} and contains the equilibrium \(E_0\), whereas ii) ensures that \(d_L:=3-b/\gamma\) gives an upper bound for the Lyapunov dimension of this set.

In the remainder of the section, we handle the classical values $a=0.386, b=0.2$ and by using numerical analysis, show that ii) is true with a high precision.
To this end, we employ Thm.~\ref{second_method}.

With this in mind, we pick a positive definite $3 \times 3$-matrix $P_0$ and put $P(x,y,z)=P_0\,e^{v(x,y,z)}$, where $v(x,y,z) :=\tau(z-bx)$ and $\tau>0$ is a free parameter.
It is easy to see that the roots of equation $(\ref{eigenvalue_problem})$ depend solely on $y$ and are as follows
\[
\lambda_1=\eta_1+\dot v,\qquad
\lambda_2=\eta_2+\dot v,\qquad
\lambda_3=\eta_3+\dot v,
\]
where $\dot v=\tau\big[(a+b)y-ay^2\big]$
is the orbital derivative of $v(\cdot)$ and $\eta_1 \geq \eta_2 \geq \eta_3$ are the roots of the equation
\[
\det\!\big[A(y)^\top P_0+P_0A(y)-\eta P_0\big]=0.
\]
ChatGPT 5 was employed to numerically solve the following minimization problem
\begin{gather}
\nonumber
s\longrightarrow \min_{\tau, P_0}
\\
\textbf{subject to} \; \tau >0, P_0=P_0\trn >0, s \in [0,1], \nonumber \\ \lambda_1(y)+\lambda_2(y)+s\,\lambda_3(y)<0 \ \text{for all } y\in[-20,20] .
\label{constr}
\end{gather}
The suggested solution $s^\star \approx 0.60557$, $\tau^\star\approx 0.25$, and
\[
P^\star \approx
  \begin{pmatrix}
   0.50578332 & -0.03189052 & -0.15406100\\
  -0.03189052 &  0.36983503 &  0.26733901\\
  -0.15406100 &  0.26733901 &  0.52428427
  \end{pmatrix}
\]
was independently verified to find that it meets the constraints from \eqref{constr} and $\lambda_3(y) <0 \; \forall y \in [-20,20]$.
By Thm.~\ref{second_method}, this means that if the system~\eqref{rossler1} admits a positively invariant set $\mathbb{X} \subset \{(x,y,z): |y| \leq 20\}$ satisfying Asm.~\ref{ass_on_X+}, then it is uniformly $d$-contractive on $\mathbb{X}$ with any $d \geq d_\star:=2+s^\star\approx 2.60557$. This provides evidence in favor of the Leonov's conjecture (see statement ii)) since the numerical value of the conjectured threshold $3-b/\gamma\approx 2.6055653$ equals the just established one $d_\star$ with a high precision. The authors, however, are unaware if for the classical parameters, the attractor contains the origin (see statement i)).

%Our analysis also shows that $d \geq 2.60557$ is (an approximate) {\em necessary and sufficient} condition for the uniform $d$-contraction of (\ref{rossler1}) with $a=0.386, %b=0.2$ on sets $\mathbb{X} \subset \{(x,y,z): |y| \leq 20\}$ satisfying Asm.~\ref{ass_on_X+} and containing the origin $E_0$.

\subsection{Langford system}
\label{sec:langford}

The following system is commonly attributed to Langford (also spelled Lanford):
\begin{align}\label{lanford}
\begin{split}
  \dot x &= (a-1)x-y+xz, \\
  \dot y &= x +(a-1)y+yz,\qquad x,y,z\in\mathbb{R},\ a>0, \\
  \dot z &= az-(x^2+y^2+z^2),
\end{split}
\end{align}
and has been extensively studied (see, e.g., \cite{belozyorov,Nik_Bozh,Yang_Yang}), particularly in connection with the stability of periodic orbits of nonlinear systems.
A convenient feature of \eqref{lanford} is that for $a\in(\tfrac12,1)$, there is the explicit $2\pi$-periodic solution
\begin{equation}
\label{eq:periodic_Lanford}
x_\star(t)=R\cos t,\quad
y_\star(t)=R\sin t,\quad
z_\star(t)=1-a,
\end{equation}
where
\(
R :=\sqrt{(1-a)(2a-1)}.
\)

The theory developed in Sec.~\ref{sec.orbital} allows us to establish the following result.

\begin{thmm}
The orbit of the solution \eqref{eq:periodic_Lanford}
satisfies the Andronov--Witt condition if and only if
\(\tfrac12<a<\tfrac23\).
\end{thmm}

\begin{proof}
Along the periodic solution \eqref{eq:periodic_Lanford}, the Jacobian matrix
$A(t)=A[x_\star(t),y_\star(t),z_\star(t)]$ takes the form
\begin{align*}
A(t)=
\begin{pmatrix}
0 & -1 & R\cos t\\
1 & 0 & R\sin t\\
-2R\cos t & -2R\sin t & 3a-2
\end{pmatrix}\\
=
\begin{pmatrix}
\Omega & R\,u(t)\\
-2R\,u(t)^{\trn} & 3a-2
\end{pmatrix},
\end{align*}
where
\[
\Omega:=\begin{pmatrix}0&-1\\1&0\end{pmatrix},\qquad
u(t):=\binom{\cos t}{\sin t}.
\]
We introduce the rotating orthonormal basis
\[
v(t)=\binom{-\sin t}{\cos t},\qquad
Q(t)=\big[u(t)\ \ v(t)\big],
\]
for which $Q(t)^{\trn}Q(t)=I_2$ and
\[
\dot Q(t)=Q(t)\Omega,\qquad Q(t)^{\trn}\dot Q(t)=\Omega.
\]
We extend this transformation to $\mathbb{R}^3$ by putting
\[
T(t)=\diag [Q(t),1 ],
\]
and consider the linear ODE $\dot\xi=A(t)\xi$. The change of the variable $\xi \mapsto \zeta:= T(t)^{-1} \xi$
transforms this ODE into
\[
\dot\zeta=B\,\zeta,
~
B :=T^{-1}A(t)T-T^{-1}\dot T
=
\begin{pmatrix}
0 & 0 & R\\
0 & 0 & 0\\
-2R & 0 & 3a-2
\end{pmatrix}.
\]
The matrix $B$ has the zero eigenvalue with the eigenvector $(0,1,0)^\top$.
So the permutation of the coordinates that flips the first and second of them shapes the system into
\[
\dot{z}
=
\diag\{0,A_2\} z, \quad \text{where}
\;
A_2=
\begin{pmatrix}
0 & R\\
-2R & 3a-2
\end{pmatrix}.
\]
The matrix $A_2$ is Hurwitz with a real $R$ if and only if
\(
\tfrac12<a<\tfrac23.
\)

Hence, similarly to the argument used in the proof of
Lemma~\ref{lem.singbl}, there exists a positive definite matrix $P_2$
such that the generalized eigenvalues $\lambda_1,\lambda_2$ of
\[
\det(A_2^{\trn}P_2+P_2A_2-\lambda P_2)=0
\]
can be made arbitrarily close to the doubled real parts of the eigenvalues of $A_2$.
Letting $P_1=\operatorname{blkdiag}\{1,P_2\}$ and
$P_0=UP_1U^{\trn}$ yields
\(
\lambda_1+\lambda_2<0
\)
for the two largest solutions of
\(
\det(\tilde J^{\trn}P_0+P_0\tilde J-\lambda P_0)=0.
\)

Returning to the original coordinates, define the periodic matrix
\(
P(t)=T(t)P_0T(t)^{\trn}.
\)
Then
\[
A(t)^{\trn}P(t)+P(t)A(t)+\dot P(t)
=
T(t)\big(\tilde J^{\trn}P_0+P_0\tilde J\big)T(t)^{\trn},
\]
and therefore the generalized eigenvalues of
\[
\det(A(t)^{\trn}P(t)+P(t)A(t)+\dot P(t)-\lambda P(t))=0
\]
coincide with those of the constant pencil
\[
\det(\tilde J^{\trn}P_0+P_0\tilde J-\lambda P_0)=0.
\]
Consequently $\lambda_1+\lambda_2<0$, and condition {\bf ii)} of
Theorem~\ref{thm:orbital} holds provided $\tfrac12<a<\tfrac23$.
\end{proof}

\begin{figure}[t]
\centering
\includegraphics[width=\columnwidth]{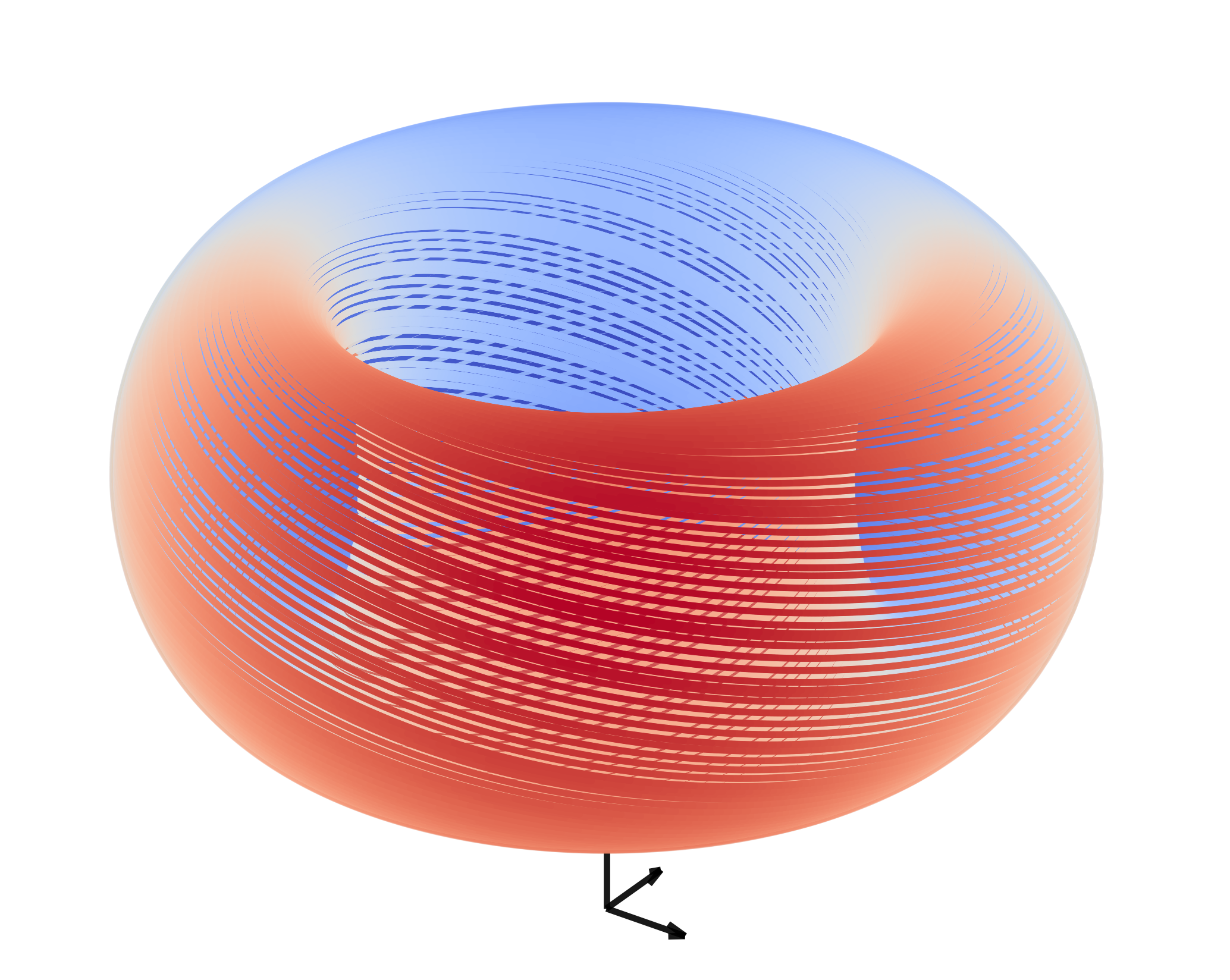}
\caption{
A  trajectory of the Langford system \eqref{lanford} for $a=\tfrac{2}{3}$.
The points are colored according to their depth relative to the viewing direction:
warmer (red) tones correspond to points closer to the viewer, while colder (blue) tones indicate points farther away.
A small coordinate frame is shown at the origin for reference.
}
\label{fig:langford_a_boundary}
\end{figure}

\section{Conclusions}

This paper develops a measure-theoretic extension of the classical $k$-contraction theory by introducing the concept of uniform $d$-contraction, where the contraction dimension $d$ is allowed to take arbitrary real values in the range $(0,n]$. A key element of the proposed framework is a family of metric-dependent {\em Hausdorff--Riemann elliptic $d$-measures}. They naturally arise from state-dependent Riemannian metrics and their evolution along the trajectories of a dynamical system underlies a quantitative characterization of the contraction of compact sets.

A central result of the paper states that uniform $d$-contraction is equivalent to the existence of a Riemannian metric for which the associated Hausdorff--Riemann elliptic $d$-measure of any compact set decreases exponentially along the system's trajectories. In this way, these measures may be viewed as {\em set-based Lyapunov functions} and carry potential for extending the classical Lyapunov's approaches via broadening the focus from trajectory-wise stability analysis to examining the evolution of dimensional volumes of certain sets. Necessary and sufficient conditions are obtained using both Lyapunov characteristic exponents (the first Lyapunov method) and metric-based constructions (the second Lyapunov method) that provide complementary analytical and computational tools for verifying contraction properties.

The proposed theory is illustrated through several examples, including systems where analytical estimates are tractable as well as cases requiring numerical studies. Beyond providing new criteria for verifying contraction and orbital stability, the introduced Hausdorff--Riemann measures offer a flexible geometric framework for studying dimensional contraction, estimating invariant-set dimensions, and designing contraction-inducing metrics. These results open the way for further development of numerical algorithms for computing contraction metrics and for applications to high-dimensional nonlinear systems where classical contraction analysis is too conservative.

\section*{Acknowledgements}

The authors are grateful to Dr. R. Kato for drawing our attention to the paper \cite{Ilyushenko}.

\appendix
\section{Proofs of the supplimentary results.}
\textbf{Proof of Lemma \ref{lem.epony}:}
{\bf i)} We observe that
\begin{gather*}
\leftindex^{P_2\!}{B}_x^r = \{ y: \text{\raisebox{6.0pt}{\rotatebox{180}{\ding{226}}}} y-x,y-x \text{\raisebox{-1.0pt}{\ding{226}}}^{P_2}\leq r^2 \} 
\\
= \{y: \spr{P_2(y-x)}{y-x} \leq r^2\}
\\
= \{y: \spr{S_2(y-x)}{S_2(y-x)} \leq r^2 
\\
= \{y : S_2(y-x) \in B_0^r\} = x+S^{-1}_2 B_0^r,
\end{gather*}
i.e., the first equation in i) is true 
The second equation follows from the first one.
\par
{\bf ii)} Let $A^\star$ stand for the operator adjoint to $A$ w.r.t.  \raisebox{6.0pt}{\rotatebox{180}{\ding{226}}}$\cdot,\cdot$\raisebox{-1.0pt}{\ding{226}}$^{P_2}$. Then for any $x,y \in \br^n$, we have
\begin{gather*}
\text{\raisebox{6.0pt}{\rotatebox{180}{\ding{226}}}}Ax,y\text{\raisebox{-1.0pt}{\ding{226}}}^{P_2}= \text{\raisebox{6.0pt}{\rotatebox{180}{\ding{226}}}}x,A^\star y\text{\raisebox{-1.0pt}{\ding{226}}}^{P_2}
\Leftrightarrow \\
\spr{P_2Ax}{y} = \spr{P_2x}{A^\star y}  = \spr{(A^\star)^\top P_2x}{ y} \Leftrightarrow \\ P_2A = (A^\star)^\top P_2 \Leftrightarrow A^\star = P_2^{-1} A^\top P_2 \Rightarrow 
\\ 
A^\star A = P^{-1}_2 A^\top P_2 A = S^{-1}_2 S^{-1}_2 A^\top S_2S_2 A 
\\
= S^{-1}_2 \big[ S^{-1}_2 A^\top S_2 S_2 A  S^{-1}_2\big] S_2 = S^{-1}_2 \Xi^\top \Xi S_2, 
\end{gather*}
where $\Xi := S_2 A  S^{-1}_2$. It follows that the matrices $A^\star A$ and $\Xi^\top \Xi$ are similar and so have common eigenvalues. The definition of the singular values completes the proof of the first formula in ii). By applying it to $P_2:= P_1$ and $A:= S_{1 \to 2}  A S_{1 \to 2}^{-1}$, we arrive at the second formula.
The third and fourth formulas follow from the first and second ones, respectively, due to \eqref{def.omega}.
\par
{\bf iii)} We represent the ellipsoid $E$ in the form $E = x+ \mathscr{A} (\leftindex^{P_2\!}{B}_0^1)$ with some $x \in \br^n$ and $\mathscr{A} \in \br^{n \times n}, \det \mathscr{A} \neq 0$.
By ii) and Lem.~\ref{lem.same}, 
\begin{gather*}
\varpi_d^{P_2} (E) = \omega_d^{P_2}(\mathscr{A}) = \omega_d(S_2\mathscr{A}S_2^{-1}),
\\
\varsigma_i^{P_2} (E) = \sigma_i^{P_2}(\mathscr{A}) = \sigma_i(S_2\mathscr{A}S_2^{-1}). 
\end{gather*}
By using Lem.~\ref{lem.same} once more and noting that $E = x+ \mathscr{A} (\leftindex^{P_2\!}{B}_0^1) = x+ \mathscr{A} S_2^{-1} B_0^1 \Rightarrow S_2 E = S_2x + S_2\mathscr{A} S_2^{-1} B_0^1 $, we arrive at the first and third relations in iii). In turns, they imply the second and fourth relations.
\par
{\bf iv)}
We represent the ellipsoid $E$ in the form $E = x+ \mathscr{A} (\leftindex^{P_1\!}{B}_0^1)$, where $x \in \br^n$ and $\mathscr{A} \in \br^{n \times n}, \det \mathscr{A} \neq 0$.
Then
\begin{gather*}
\varpi_d^{P_2}(E) \overset{\text{iii)}}{=} \varpi_d^{P_1}(S_{1 \to 2}E) = \varpi_d^{P_1} \big[ S_{1 \to 2}x + S_{1 \to 2}\mathscr{A} (\leftindex^{P_1\!}{B}_0^1) \big] \\ \overset{\text{Lem.~\ref{lem.same}}}{=\!=\!=}  \omega_d^{P_1} ( S_{1 \to 2}\mathscr{A})
\overset{\text{\eqref{eq:submult}}}{\leq} \omega_d^{P_1} ( S_{1\to 2})\omega_d^{P_1} (\mathscr{A})
\\
\overset{\text{Lem.~\ref{lem.same}}}{=\!=\!=}  \omega_d^{P_1} ( S_{1 \to 2}) \varpi_d^{P_1} (E) \overset{\text{ii)}}{=} \omega_d ( S_1 S_{1 \to 2} S_1^{-1}) \varpi_d^{P_1} (E) \\
= \omega_d (S_{2} S_1^{-1}) \varpi_d^{P_1} (E).
\end{gather*}
Thus we see that in iv), the second inequality is true. By flipping the indices in it, we arrive at the first inequality.
\iffalse
By noting that $S^{-1}E = S^{-1} x + S^{-1}  A (\leftindex^{P\!}{B}_0^1)$, we similarly see that
\begin{gather*}
\varpi_d(E) = \varpi_d^P(S^{-1}E) \overset{\text{Lem.~\ref{lem.same}}}{=\!=\!=} \omega_d^P (S^{-1}  A ) \overset{\text{\eqref{eq:submult}}}{\leq} \omega_d^P ( S^{-1})\omega_d^P (A) \overset{\text{Lem.~\ref{lem.same}}}{=\!=\!=}  \omega_d^P ( S^{-1}) \varpi_d^P(E)
\overset{\text{ii)}}{=} \omega_d ( S S^{-1} S^{-1})\varpi_d^P(E)
\\
 = \omega_d (S^{-1})\varpi_d^P(E) \Rightarrow \varpi_d^P(E) \geq \frac{1}{\omega_d (S^{-1})} \varpi_d(E) .
\end{gather*}
\fi
\par
{\bf v)} We first note that for any $A,B \in \br^{n \times n}$,
\begin{gather}
\label{mmppll}
\sigma_1(A^{-1}) = \frac{1}{\sigma_n(A)} \quad \text{if}\; \det A \neq 0, \\ \sigma_1(AB) = \|AB\|_2 \leq \|A\|_2\|B\|_2 = \sigma_1(A) \sigma_1(B) .\nonumber
%\\
%\label{mmppl2}
%\sigma_1(A^{-1}) = \frac{1}{\sigma_n(A)} \quad \text{if}\; \det A \neq 0.
\end{gather}
The second inequality in v) is proven as follows:
\begin{gather*}
\varsigma_1^{P_2} (E) = \max_{x \in E} \sqrt{\spr{P_2 x}{x}} = \max_{x \in E} \|S_2 x\| \\ = \max_{x \in E} \|S_2 S^{-1}_1 S_1x\|
\leq \sigma_1(S_2 S^{-1}_1) \max_{x \in E} \|S_1x\| = \\ \sigma_1(S_2 S^{-1}_1) \varsigma_1^{P_1} (E)
\overset{\text{\eqref{mmppll}}}{\leq} \sigma_1(S_2)\sigma_1( S^{-1}_1) \varsigma_1^{P_1} (E) \\
\overset{\text{\eqref{mmppll}}}{=} \frac{\sigma_1(S_2)}{\sigma_n( S_1)} \varsigma_1^{P_1} (E).
\end{gather*}
The first inequality follows from the second one by flipping the indices 1 and 2.
Similarly
\begin{gather*}
\varsigma_n^{P_2} (E) = \min_{x \not\in E} \sqrt{\spr{P_2 x}{x}} = \min_{x \not\in E} \|S_2 S^{-1}_1 S_1x\|
\\ \leq \sigma_1(S_2 S^{-1}_1) \min_{x \not\in E} \|S_1x\| = \sigma_1(S_2 S^{-1}_1) \varsigma_n^{P_1} (E) \\
\overset{\text{\eqref{mmppll}}}{\leq} \frac{\sigma_1(S_2)}{\sigma_n (S_1)} \varsigma_n^{P_1} (E)
,
\end{gather*}
i.e., the forth inequality in v) holds. This inequality implies the third one by flipping the indices 1 and 2.

\par
{\bf vi)} It suffices to note that
\begin{gather*}
\ecc^{P_2}(E)=\frac{\varsigma_1^{P_2}(E)}{\varsigma_n^{P_2}(E)} \overset{\text{v)}}{\leq} \frac{\frac{\sigma_1(S_2)}{\sigma_n( S_1)} \varsigma_1^{P_1} (E)}{\frac{\sigma_n(S_2)}{\sigma_1 (S_1)} \varsigma_n^{P_1} (E)} \\ =
\frac{\sigma_1(S_2)}{\sigma_n(S_2)} \frac{\sigma_1(S_1)}{\sigma_n (S_1)} \ecc^{P_1}(E),
\end{gather*}
i.e., the second inequality from vi) holds. The first inequality follows from the second one.
\par
{\bf vii)} We represent $E$ in the form $E = x+ C (B_0^1)$, where $x \in \br^n$ and $C \in \br^{n \times n}, \det C \neq 0$. Then
\begin{gather*}
\varpid^{P_2} (AE) \overset{\text{iii)}}{=} \varpi_d^{P_1}(S_{1 \to 2} A E)  \overset{\text{iii)}}{=} \varpi_d(S_1 S_{1 \to 2} A E) \\ = \varpi_d(S_{2} A S_1^{-1} S_1 E)
\overset{\text{Lem.~\ref{lem.same}}}{=\!=} \omega_d (S_{2} A S_1^{-1} S_1 C)
\\
\overset{\text{\eqref{eq:submult}}}{\leq} \omega_d (S_{2} A S_1^{-1}) \omega_d (S_1 C) \overset{\text{Lem.~\ref{lem.same}}}{=\!=} \omega_d (S_{2} A S_1^{-1}) \varpi_d (S_1 E) \\
\overset{\text{iii)}}{=}  \omega_d (S_{2} A S_1^{-1}) \varpi_d^{P_1} (E). \qquad \text{\qed}
\end{gather*}

\iffalse
\begin{cor}
Let $P \in \br^{n \times n}$ be a positive definite matrix and let $S:= \sqrt{P} \in \br^{n \times n}$ stand for its positive definite square root.
For any compact set $K \subset \br^n$, we have $\mu_d^P(K) = \mu_d(SK)$.
\end{cor}
\fi
\textbf{Proof of Lemma \ref{lemm.elip}:}
\begin{proof}
{\bf i)} The first two relations are trivial and imply the third one by \eqref{def.ecc}.
\\
{\bf ii)} We put $S:= \sqrt{P}$. By iii) in Lem.~\ref{lem.epony}, $\varpid^P(AE) = \varpid (SAE) = \varpid (SAS^{-1} SE) $.
We proceed by writing the ellipsoid $E$ as $E=  x + C(\Ball{1}{0})$ with some $n \times n$ matrix $C$ and $x \in \br^n$. 
By Lem.~\ref{lem.same} and \eqref{eq:submult},
\begin{gather*}
\varpid^P(AE) = \omegad (SAS^{-1} SC) 
\\
\leq \omegad (SAS^{-1}) \omegad (SC) =  \omega_d^P(A)\, \varpid^P(E), 
\end{gather*}
where the last equation is given by Lem.~\ref{lem.same}, along with ii) and iii) in Lem.~\ref{lem.epony}.
\par
{\bf iii)} Since $E$ is a convex set, 
\begin{gather*}
\Ball{\lambda}{0} = \frac{\lambda}{\varsigma_n(E)}\bigl[\varsigma_n(E)\Ball{1}{0}\bigr]\subset \frac{\lambda}{\varsigma_n(E)}[E-\mathbf{c}(E)],
\\
E+\Ball{\lambda}{0} = \mathbf{c}(E)+ E - \mathbf{c}(E) +\Ball{\lambda}{0} \\ \subset \mathbf{c}(E) + \left[ 1+ \frac{\lambda}{\varsigma_n(E)}\right] [E - \mathbf{c}(E)] = E^+_\lambda.
\end{gather*}
This implies the inequality in iii) by the last fact in i).
\par
{\bf iv)} We write the ellipsoid $E$ in the form $E=y+T[\leftindex^{P\!}B_0^1]$, where $T \in \br^{n \times n}$. This yields a similar representation $x+AE = x+Ay + AT[\leftindex^{P\!}B_0^1] $. It is easy to see that $\varsigma_i^P(E) = \sigma_i^P(T)$
and similarly $\varsigma_i^P(x+AE) = \sigma_i^P(AT)$
for $i=1, \ldots, n$. By denoting $\|x\|_P := \sqrt{\spr{Px}{x}}$ and $\|C\|_P := \max_{\|x\|_P \leq 1} \|C\|_P \; \forall C \in \br^{n \times n}$, we see that the first two inequalities in iv) hold since
\begin{gather*}
\varsigma_1^P(x+AE) = \sigma_1^P(AT) = \|AT\|_P \leq \|A\|_P \|T\|_P \\ = \sigma_1^P(A) \sigma_1^P(T)  = \sigma_1^P(A) \varsigma_1^P(E);\\
\varsigma_n^P(x+AE) = \sigma_n^P(AT) = \frac{1}{\sigma_1^P(T^{-1}A^{-1})} \\ = \frac{1}{\|T^{-1}A^{-1}\|_P} \geq \frac{1}{\|T^{-1}\|_P}\frac{1}{\|A^{-1}\|_P} \\ =
\frac{1}{\sigma_1^P(T^{-1})} \frac{1}{\sigma_1^P(A^{-1})} = \underbrace{\sigma_n^P(T)}_{\varsigma_n(E)} \sigma_n^P(A).
\end{gather*}
In turns, they imply the third inequality.
\end{proof}

\begin{lemma}
For any $\gamma \in (0,\gamma_\star)$, $\ve \in \left(0, \frac{\gamma m_{P(\cdot)}}{2} \right)$, $\alpha \geq 1$, and compact set $K\subset O$,
the following inequality holds:
\begin{gather}
\label{def.max}
\leftindex^{P(\cdot)\!}\pi^{\langle a(\alpha)\rangle}_{d,\epsilon_\gamma(\varepsilon,\alpha)}[g(K)] \leq \leftindex^{P(\cdot)\!}\pi_{d,\varepsilon}^{\langle \alpha \rangle}(K) (1+\xi_{\ve,\alpha,\gamma})^d \Omega_\gamma .
\end{gather}
\end{lemma}
\textbf{Proof of Lemma \ref{thm:unit-constant}:}
\begin{proof}
Given $\eta>0$, formula \eqref{def.meas+} implies that there exists a finite covering $\{E_i\}$ of $K$ with ellipsoids $E_i$ for which
\begin{gather}
\label{many.rel}
\sum_i \varpi_d^{P[\mathbf{c}(E_i)]}(E_i) \leq \leftindex^{P(\cdot)\!}\pi_{d,\varepsilon}^{\langle \alpha \rangle}(K) + \eta, \; \\
K\subset \bigcup_i E_i, \; \varsigma_1^{P[\mathbf{c}(E_i)]}(E_i)<\varepsilon, \; \ecc^{P[\mathbf{c}(E_i)]}(E_i) \leq \alpha \; \forall i.\nonumber
\end{gather}
By dropping the ellipsoids not intersecting $K$, we retain the above relations and ensure that $E_i \cap K \neq \emptyset \; \forall i$.
We denote $\mathbf{c}_i := \mathbf{c}(E_i)$ and use
v) in Lem.~\ref{lem.epony} (where $P_2 := I_n, P_1:= P[\mathbf{c}(E_i)]$) to see that
\begin{gather}
\varsigma_1 (E_i) \leq \frac{1}{\sigma_n(S[\mathbf{c}(E_i)])} \varsigma_1^{P[\mathbf{c}(E_i)]} (E_i)
\overset{\text{\eqref{two.ineq}}}{\leq} \frac{\ve}{m_{P(\cdot)}} \nonumber \\ \Rightarrow r_i:= \mathbf{diam}\,E_i \leq \frac{2\ve}{m_{P(\cdot)}} < \gamma
\overset{E_i \cap K\neq \emptyset}{=\!=\!=\!=\!\Rightarrow} \mathbf{c}_i \in N_\gamma^{K}.\label{last.inn}
\end{gather}
Applying Lem.~\ref{lem.dif0} to $M:=E_i, x_\ast:= \mathbf{c}_i := \mathbf{c}(E_i)$ and invoking \ref{lemm.elip3}) in Lem.~\ref{lemm.elip} yield that
\begin{gather}
 \ff(E_i) \subset \ff(\mathbf{c}_i) + D \ff (\mathbf{c}_i)[E_i - \mathbf{c}_i] +  \Ball{r_i \delta_{g,\gamma}(r_i)}{0} \subset \mathcal{E}_i \nonumber\\ := \ff(\mathbf{c}_i)+ \left[1+\frac{r_i \delta_{g,\gamma}(r_i)}{\varsigma_n(\mathfrak{E}_i)} \right] \mathfrak{E}_i,
 \label{incl.dif}
 \\
 \label{dwef.egot}
 \text{where}  ~ \mathfrak{E}_i:= D \ff (\mathbf{c}_i)\big[ E_i - \mathbf{c}_i \big];
 \\
 \label{cover}
 \ff(K) \subset \bigcup_i \ff(E_i) \subset \bigcup_i \mathscr{E}_i,   \text{where} \quad \mathbf{c}(\mathscr{E}_i) = \ff[\mathbf{c}_i].
\end{gather}
We proceed by noting that  $r_i := \mathbf{diam}\,E_i = 2 \varsigma_1(E_i)$ and
\begin{gather}
\frac{r_i \delta_{g,\gamma}(r_i)}{\varsigma_n(\mathfrak{E}_i)} \overset{\text{\ref{lemm.elip4}) in Lem.~\ref{lemm.elip}}}{\leq} \frac{r_i \delta_{g,\gamma}(r_i)}{\sigma_n [Df(\mathbf{c}_i)]\varsigma_n(E_i)} \nonumber\\
\overset{\text{\eqref{def.sigg},\eqref{last.inn}}}{\leq} \frac{r_i \delta_{g,\gamma}(r_i)}{\sigma^-_\gamma \varsigma_n(E_i)}
= \ecc (E_i) \frac{r_i \delta_{g,\gamma}(r_i)}{\sigma^-_\gamma\varsigma_1(E_i)} =  2\ecc (E_i) \frac{\delta_{g,\gamma}(r_i)}{\sigma^-_\gamma}
\nonumber\\ \overset{\text{vi) in Lem\ref{lem.epony}}}{\leq} 2 \frac{ \sigma_1 \left[S(\mathbf{c}_i)\right]}{\sigma_n \left[S(\mathbf{c}_i)\right]} \ecc^{P(\mathbf{c}_i)} (E_i) \frac{\delta_{g,\gamma}(r_i)}{\sigma^-_\gamma} \label{nbnb1}
\\
\label{nbnb2}
\overset{\text{\eqref{two.ineq}}}{\leq} 2 \frac{M_{P(\cdot)}}{ m_{P(\cdot)}} \ecc^{P(\mathbf{c}_i)}(E_i) \frac{\delta_{g,\gamma}(r_i)}{\sigma^-_\gamma}
\\
\overset{\text{\eqref{defki}}}{=} 2 e_{P(\cdot)} \ecc^{P(\mathbf{c}_i)}(E_i) \frac{\delta_{g,\gamma}(r_i)}{\sigma^-_\gamma} \nonumber
\\ 
\nonumber
\overset{\text{\eqref{many.rel}}}{\leq}
2 e_{P(\cdot)} \alpha \frac{\delta_{g,\gamma}(r_i)}{\sigma^-_\gamma} \overset{\text{\eqref{modulus},\eqref{defki},\eqref{last.inn}}}{\leq} \xi:= \xi_{\ve,\alpha,\gamma},
\\
\nonumber
\varsigma_1^{P[\mathbf{c}(\mathscr{E}_i)]}(\mathscr{E}_i) \overset{\text{\ref{lemm.elip1}) in Lem.~\ref{lemm.elip}}}{=\!=\!=\!=\!=\!=} \left[1+\frac{r_i \delta_{g,\gamma}(r_i)}{\varsigma_n(E_i)} \right] \varsigma_1^{P[\mathbf{c}(\mathscr{E}_i)]}(\mathfrak{E}_i)
\nonumber \\\overset{\text{v) in Lem.~\ref{lem.epony}}}{\leq}
( 1+\xi) \frac{\sigma_1 \left\{ S[g(\mathbf{c}_i)] \right\}}{ \sigma_n \left\{ S(\mathbf{c}_i)\right\}}\varsigma_1^{P(\mathbf{c}_i)}(\mathfrak{E}_i)
\\
\nonumber
 \overset{\text{\eqref{two.ineq}}}{\leq} \frac{M_{P(\cdot)}}{ m_{P(\cdot)}} (1+\xi)
\varsigma_1^{P(\mathbf{c}_i)}(\mathfrak{E}_i)
\overset{\text{\eqref{defki}}}{=} e_{P(\cdot)} (1+\xi)
\varsigma_1^{P(\mathbf{c}_i)}(\mathfrak{E}_i) \nonumber\\ \overset{\text{iii) in Lem.~\ref{lemm.elip}}}{=\!=\!=\!=\!=\!=} e_{P(\cdot)} (1+\xi)
\varsigma_1\big[S(\mathbf{c}_i)\mathfrak{E}_i \big]
\\
\nonumber
\overset{\text{\ref{lemm.elip4}) in Lem.~\ref{lemm.elip}}}{\leq} e_{P(\cdot)} (1+\xi) \sigma_1\big[ S(\mathbf{c}_i) \big]
\varsigma_1\left\{ D \ff (\mathbf{c}_i)\big[ E_i - \mathbf{c}_i \big] \right\} \\ \nonumber
\overset{\text{\ref{lemm.elip4}) in Lem.~\ref{lemm.elip}}}{\leq} e_{P(\cdot)} (1+\xi) \sigma_1\big[ S(\mathbf{c}_i) \big] \sigma_1 \big[  D \ff (\mathbf{c}_i)\big]
\varsigma_1 (E_i)
\\
\nonumber
\overset{\text{\eqref{two.ineq},\eqref{def.sigg}}}{\leq} e_{P(\cdot)} (1+\xi) M_{P(\cdot)} \sigma^+_\gamma \varsigma_1 (E_i)\nonumber\\
\overset{\text{\eqref{defki},\eqref{many.rel}}}{\leq} e_{P(\cdot)}  M_{P(\cdot)} \sigma^+_\gamma \left\{1+2 \alpha \frac{e_{P(\cdot)}}{\sigma_\gamma^-} \delta_{\gamma}\left[ \frac{2 \ve}{m_{P(\cdot)}} \right] \right\}  \ve 
\nonumber
\\
\overset{\text{\eqref{defki1}}}{=} \epsilon_\gamma(\ve,\alpha).
\nonumber
\end{gather}
\par
By using ii), vi), and vii) in Lem.~\ref{lem.epony}, we also see that
\begin{gather*}
\ecc^{P[\mathbf{c}(\mathscr{E}_i)]}(\mathscr{E}_i)
\overset{\text{\eqref{incl.dif},\eqref{cover}}}{=\!=\!=\!=} \ecc^{P[g(\mathbf{c}_i)]}(\mathfrak{E}_i)\nonumber\\
\overset{\text{\eqref{dwef.egot}}}{=} \ecc^{P[g(\mathbf{c}_i)]}\big[ D \ff(\mathbf{c}_i) E_i \big] \nonumber\\
\overset{\text{vi)}}{\leq} \frac{\sigma_1(\sqrt{P[\mathbf{c}(\mathscr{E}_i)]})}{\sigma_n\sqrt{P[\mathbf{c}(\mathscr{E}_i)]}} \frac{\sigma_1(\sqrt{P[\mathbf{c}(E_i)]})}{\sigma_n (\sqrt{P[\mathbf{c}(E_i)]})} \ecc^{P[\mathbf{c}(E_i)]}\big[ D \ff(\mathbf{c}_i) E_i \big]
\\
\overset{\text{\eqref{two.ineq},\eqref{defki}}}{\leq} \big[e_{P(\cdot)} \big]^2 \ecc^{P(\mathbf{c}_i)}\big[ D \ff(\mathbf{c}_i) E_i \big] \nonumber\\
\overset{\text{\ref{lemm.elip4}) in Lem.~\ref{lemm.elip}}}{\leq} \big[e_{P(\cdot)} \big]^2 \frac{\sigma_1^{P(\mathbf{c}_i)}[D \ff(\mathbf{c}_i)]}{\sigma_n^{P(\mathbf{c}_i)}[D \ff(\mathbf{c}_i)]}\ecc^{P(\mathbf{c}_i)}\big[E_i \big] \nonumber\\
\overset{\text{\eqref{many.rel}}}{\leq} \big[e_{P(\cdot)} \big]^2 \frac{\sigma_1^{P(\mathbf{c}_i)}[D \ff(\mathbf{c}_i)]}{\sigma_n^{P(\mathbf{c}_i)}[D \ff(\mathbf{c}_i)]} \alpha
\\
\overset{\text{ii) in Lem.~\ref{lem.epony}},S:= \sqrt{P(\mathbf{c}_i})}{=\!=\!=\!=\!=\!=\!=\!=\!=\!=\!=\!=\!=}
\big[e_{P(\cdot)} \big]^2 \frac{\sigma_1[SD \ff(\mathbf{c}_i)S^{-1}]}{\sigma_n[SD \ff(\mathbf{c}_i)S^{-1}]} \alpha \nonumber\\
\overset{\text{\eqref{mmppll}}}{\leq}
\big[e_{P(\cdot)} \big]^2 \frac{\sigma_1(S) \sigma_1(S^{-1}) \sigma_1[D \ff(\mathbf{c}_i)]}{\sigma_n(S) \sigma_n(S^{-1}) \sigma_n[D \ff(\mathbf{c}_i)]} \alpha \nonumber\\ 
\overset{\text{\eqref{mmppll}}}{=}
\big[e_{P(\cdot)} \big]^2 \frac{\sigma_1(S)^2 \sigma_1[D \ff(\mathbf{c}_i)]}{\sigma_n(S)^2 \sigma_n[D \ff(\mathbf{c}_i)]} \alpha
\\
\overset{\text{\eqref{two.ineq},\eqref{defki}}}{\leq} \big[e_{P(\cdot)} \big]^4 \frac{ \sigma_1[D \ff(\mathbf{c}_i)]}{\sigma_n[D \ff(\mathbf{c}_i)]} \alpha
\overset{\text{\eqref{def.sigg}}}{\leq}
\big[e_{P(\cdot)} \big]^4 \frac{ \sigma_\gamma^+}{\sigma_\gamma^-} \alpha
\overset{\text{\eqref{defki}}}{=} a(\alpha).
\end{gather*}
\par
By the foregoing, the collection $\{\mathcal{E}_i^+\}$ of the ellipsoids is among those from the definition \eqref{def.meas+} of
$\leftindex^{P(\cdot)\!}\pi^{\langle a(\alpha)\rangle}_{d,\epsilon_\gamma(\varepsilon,\alpha)}[g(K)]$ and so
\begin{gather}
\label{start.ineq}
\leftindex^{P(\cdot)\!}\pi^{\langle a(\alpha)\rangle}_{d,\epsilon_\gamma(\varepsilon,\alpha)}[g(K)] \leq \sum_i \varpi_d^{P[\mathbf{c}(\mathscr{E}_i)]}\big(\mathscr{E}_i\big).
\end{gather}
With taking into account vii) in Lem.~\ref{lem.epony} and the last relation in \eqref{incl.dif}, we see that
\begin{gather*}
\varpi_d^{P[\mathbf{c}(\mathscr{E}_i)]}\big(\mathscr{E}_i\big) \overset{\text{\eqref{incl.dif},\eqref{cover}}}{=\!=\!=} \left[1+\frac{r_i \delta_{g,\gamma}(r_i)}{\varsigma_n(\mathfrak{E}_i)} \right]^{\!d}\, \varpid^{P[g(\mathbf{c}_i)]}(\mathfrak{E}_i) \\ \overset{\text{\eqref{nbnb1},\eqref{nbnb2}}}{\leq}
(1+\xi)^d \varpid^{P[g(\mathbf{c}_i)]}(\mathfrak{E}_i) \\ \overset{\text{\eqref{dwef.egot}}}{=} (1+\xi)^d \varpid^{P[g(\mathbf{c}_i)]}\big[ D \ff (\mathbf{c}_i) E_i \big]
\\
\overset{\text{iii) in Lem.~\ref{lem.epony}}}{=\!=\!=\!=\!=\!=} (1+\xi)^d \varpid\big[ S[g(\mathbf{c}_i)] D \ff (\mathbf{c}_i) S(\mathbf{c}_i)^{-1}S(\mathbf{c}_i) E_i \big]
\\
\overset{\text{ii) in Lem.~\ref{lemm.elip}}}{\leq}
(1+\xi)^d \omega_d\big[ S[g(\mathbf{c}_i)] D \ff (\mathbf{c}_i) S(\mathbf{c}_i)^{-1}\big] 
\\
\nonumber
\times \varpid \big[S(\mathbf{c}_i) E_i \big]
\\
\overset{\text{iii) in Lem.~\ref{lem.epony}}}{=\!=\!=\!=\!=\!=} (1+\xi)^d \omega_d\big[ S[g(\mathbf{c}_i)] D \ff (\mathbf{c}_i) S(\mathbf{c}_i)^{-1}\big]
\\
\nonumber
\times
\varpid^{P[\mathbf{c}_i]} \big[S(\mathbf{c}_i) E_i \big] 
\overset{\text{\eqref{quant0}}}{\leq}
(1+\xi)^d \Omega_\gamma \varpid^{P[\mathbf{c}_i]} \big[ E_i \big].
\end{gather*}
By invoking \eqref{many.rel} and \eqref{start.ineq}, we thus see that
\begin{align*}
\leftindex^{P(\cdot)\!}\pi^{\langle a(\alpha)\rangle}_{d,\epsilon_\gamma(\varepsilon,\alpha)}[g(K)] \leq \sum_i \varpi_d^{P[\mathbf{c}(\mathscr{E}_i)]}\big(\mathscr{E}_i\big)
\\ \leq (1+\xi)^d \Omega_\gamma \sum_i \varpid^{P(\mathbf{c}_i)} \big( E_i \big) \\ \leq  (1+\xi)^d \Omega_\gamma \left[  \leftindex^{P(\cdot)\!}\pi_{d,\varepsilon}^{\langle \alpha \rangle}(K) + \eta\right].
\end{align*}
The proof is completed by letting $\eta \to 0+$.
\end{proof}

\bibliographystyle{apsrev4-2}
\bibliography{chaos_references}

\end{document}